\documentclass[reqno]{amsart}
\usepackage{amssymb}

 \newtheorem{Theorem}{Theorem}[section]
 \newtheorem{Corollary}[Theorem]{Corollary}
 \newtheorem{Lemma}[Theorem]{Lemma}
 \newtheorem{Proposition}[Theorem]{Proposition}
 
 \newtheorem{Definition}[Theorem]{Definition}
\newtheorem{Problem}[Theorem]{Problem}
 \newtheorem{Conjecture}[Theorem]{Conjecture}

 \newtheorem{Remark}[Theorem]{Remark}

 \numberwithin{equation}{section}

\begin{document}

\title[Concavity property of minimal $L^2$ integrals \uppercase\expandafter{\romannumeral4}]
{Concavity property of minimal $L^2$ integrals with Lebesgue measurable gain \uppercase\expandafter{\romannumeral4}: product of open Riemann surfaces}

\author{Qi'an Guan}
\address{Qi'an Guan: School of
Mathematical Sciences, Peking University, Beijing 100871, China.}
\email{guanqian@math.pku.edu.cn}

\author{Zheng Yuan}
\address{Zheng Yuan: School of
Mathematical Sciences, Peking University, Beijing 100871, China.}
\email{zyuan@pku.edu.cn}

\thanks{}

\subjclass[2010]{32D15, 32E10, 32L10, 32U05, 32W05}

\keywords{plurisubharmonic function, multiplier ideal sheaf, minimal $L^{2}$ integral, concavity, optimal $L^{2}$ extension theorem}

\date{\today}

\dedicatory{}

\commby{}


\begin{abstract}
In this article, we present characterizations of the concavity property of minimal $L^2$ integrals degenerating to linearity in the case of products of analytic subsets on products of open Riemann surfaces.
As applications, we obtain characterizations of the holding of equality in optimal jets $L^2$ extension problem from products of analytic subsets to products of open Riemann surfaces,
which implies characterizations of the product versions of the equality parts of Suita conjecture and extended Suita conjecture, and the equality holding of a conjecture of Ohsawa for products of open Riemann surfaces.
\end{abstract}

\maketitle

\section{Introduction}\label{introduction}

The strong openness property of multiplier ideal sheaves (i.e. $\mathcal{I}(\varphi)=\mathcal{I}_+(\varphi):=\mathop{\cup} \limits_{\epsilon>0}\mathcal{I}((1+\epsilon)\varphi)$)
conjectured by Demailly \cite{DemaillySoc} and proved by Guan-Zhou \cite{GZSOC} (Jonsson-Musta\c{t}$\breve{a}$ \cite{JonssonMustata} proved 2-dim case) has opened the door to new types of approximation techniques
in the study of several complex variables, complex algebraic geometry and complex differential geometry
(see e.g. \cite{GZSOC,K16,cao17,cdM17,FoW18,DEL18,ZZ2018,GZ20,berndtsson20,ZZ2019,ZhouZhu20siu's,FoW20,KS20,DEL21}),
where $\varphi$ is a plurisubharmonic function of a complex manifold $M$ (see \cite{Demaillybook}), and multiplier ideal sheaf $\mathcal{I}(\varphi)$ is the sheaf of germs of holomorphic functions $f$ such that $|f|^2e^{-\varphi}$ is locally integrable (see e.g. \cite{Tian,Nadel,Siu96,DEL,DK01,DemaillySoc,DP03,Lazarsfeld,Siu05,Siu09,DemaillyAG,Guenancia}).

When $\mathcal{I}(\varphi)=\mathcal{O}$, the strong openness property degenerates to the openness property conjectured by Demailly-Koll\'ar \cite{DK01}.
Berndtsson \cite{Berndtsson2} proved the openness property (the 2-dimensional case was proved by Favre-Jonsson in \cite{FavreJonsson}) by establishing an effectiveness result of the openness property.
Stimulated by Berndtsson's effectiveness result, and continuing the solution of the strong openness property \cite{GZSOC},
Guan-Zhou \cite{GZeff} established an effectiveness result of the strong openness property by considering the minimal $L^{2}$ integral on the pseudoconvex domain $D$.

Considering the minimal $L^{2}$ integrals on the sublevels of the weight $\varphi$,
Guan \cite{G16} obtained a sharp version of Guan-Zhou's effectiveness result,
and established a concavity property of the minimal $L^2$ integrals on the sublevels of the weight $\varphi$,
which deduces a proof of Saitoh's conjecture for conjugate Hardy $H^2$ kernels \cite{Guan2019},
and the sufficient and necessary condition of the existence of decreasing equisingular approximations with analytic singularities for the multiplier ideal sheaves with weights $\log(|z_{1}|^{a_{1}}+\cdots+|z_{n}|^{a_{n}})$ \cite{guan-20}.

In \cite{G2018} (see also \cite{GM}), Guan gave the concavity property for smooth gain on Stein manifolds (the case of weakly pseudoconvex K\"{a}hler manifolds was obtained by Guan-Mi\cite{GM_Sci}),
which deduces an optimal support function related to the strong openness property (obtained by Guan-Yuan) \cite{GY-support} and an effectiveness result of the strong openness property in $L^p$ (obtained by Guan-Yuan) \cite{GY-lp-effe}.
In \cite{GY-concavity}, Guan-Yuan obtained the concavity property with Lebesgue measurable gain on Stein manifolds (the case of weakly pseudoconvex K\"{a}hler manifolds was obtained by Guan-Mi-Yuan \cite{GMY-concavity2}),
which deduces a twisted version of the strong openness property in $L^p$ \cite{GY-twisted}.

Note that a linear function is a degenerate case of a concave function.
A natural problem was posed in \cite{GY-concavity3}:

\begin{Problem}\label{Q:chara}\cite{GY-concavity3}
How to characterize the concavity property degenerating to linearity?
\end{Problem}

For 1-dim case, Guan-Yuan \cite{GY-concavity} gave an answer to Problem \ref{Q:chara} for single point, i.e. for weights may not be subharmonic (the case of subharmonic weights was answered by Guan-Mi \cite{GM}),
and Guan-Yuan \cite{GY-concavity3} gave an answer to Problem \ref{Q:chara} for finite points.

In the present article, we give answers to Problem \ref{Q:chara} for some high dimensional cases, i.e., for the cases of products of open Riemann surfaces.

Let $\Omega_j$  be an open Riemann surface, which admits a nontrivial Green function $G_{\Omega_j}$ for any  $1\le j\le n$. Let $M=\prod_{1\le j\le n}\Omega_j$ be an $n-$dimensional complex manifold, and let $\pi_j$ be the natural projection from $M$ to $\Omega_j$. Let $K_M$ be the canonical (holomorphic) line bundle on $M$, and Let $K_{\Omega_j}$ be the canonical (holomorphic) line bundle on $\Omega_j$.

 Let $Z_j$ be a (closed) analytic subset of $\Omega_j$ for any $j\in\{1,2,...,n\}$, denote that $Z_0:=\Pi_{1\le j\le n}Z_j\subset M$. For any $j\in\{1,2,...,n\}$, let $\varphi_j$ be a subharmonic function on $\Omega_j$ such that $\varphi_j(z)>-\infty$ for any $z\in Z_j$, denote that $\varphi:=\sum_{1\le j\le n}\pi_j^*(\varphi_j)$. Let $\psi$ be a plurisubharmonic function on $M$ such that $\psi(z)=-\infty$ for any $z\in Z_0$.
Let $c$ be a positive function on $(0,+\infty)$ such that $\int_{0}^{+\infty}c(t)e^{-t}dt<+\infty$, $c(t)e^{-t}$ is decreasing on $(0,+\infty)$ and $c(-\psi)$ has a positive lower bound on any compact subset of $M\backslash Z_0$. Let $f$ be a holomorphic $(n,0)$ form on a neighborhood of $Z_0$.
Denote
\begin{equation*}
\begin{split}
\inf\{\int_{\{\psi<-t\}}|\tilde{f}|^{2}e^{-\varphi}c(-\psi):(\tilde{f}-f,z)\in&(\mathcal{O}(K_M)\otimes\mathcal{I}(\psi))_{z}\mbox{ for any $z\in Z_0$} \\&\&{\,}\tilde{f}\in H^{0}(\{\psi<-t\},\mathcal{O}(K_{M}))\},
\end{split}
\end{equation*}
by $G(t;c)$ (without misunderstanding, we denote $G(t;c)$ by $G(t)$),  where $t\in[T,+\infty)$ and
$|f|^{2}:=\sqrt{-1}^{n^{2}}f\wedge\bar{f}$ for any $(n,0)$ form $f$.

Recall that $G(h^{-1}(r))$ is concave with respect to $r$ (see \cite{GY-concavity}, see also \cite{GMY-concavity2}), where $h(t)=\int_{t}^{+\infty}c(s)e^{-s}ds$ for any $t\ge0$. In the present article, we discuss the characterization of the concavity of $G(h^{-1}(r))$ degenerating to linearity.

\subsection{Main results: characterizations of the concavity property of minimal $L^2$ integrals degenerating to linearity}
\

In this section,  we present characterizations of the concavity property of minimal $L^2$ integrals degenerating to linearity.

 We recall some notations (see \cite{OF81}, see also \cite{guan-zhou13ap,GY-concavity,GMY-concavity2}).
 Let $P_j:\Delta\rightarrow\Omega_j$ be the universal covering from unit disc $\Delta$ to $\Omega_j$.
 we call the holomorphic function $f$ (resp. holomorphic $(1,0)$ form $F$) on $\Delta$ a multiplicative function (resp. multiplicative differential (Prym differential)),
 if there is a character $\chi$, which is the representation of the fundamental group of $\Omega_j$, such that $g^{*}f=\chi(g)f$ (resp. $g^{*}(F)=\chi(g)F$),
 where $|\chi|=1$ and $g$ is an element of the fundamental group of $\Omega$. Denote the set of such kinds of $f$ (resp. $F$) by $\mathcal{O}^{\chi}(\Omega_j)$ (resp. $\Gamma^{\chi}(\Omega_j)$).

It is known that for any harmonic function $u$ on $\Omega_j$,
there exists a $\chi_{j,u}$ (called  character associate to $u$) and a multiplicative function $f_u\in\mathcal{O}^{\chi_{j,u}}(\Omega_j)$,
such that $|f_u|=P_j^{*}e^{u}$.
If $u_1-u_2=\log|f|$, then $\chi_{j,u_1}=\chi_{j,u_2}$,
where $u_1$ and $u_2$ are harmonic functions on $\Omega_j$ and $f$ is a holomorphic function on $\Omega_j$. Let $z_j\in \Omega_j$.
Recall that for the Green function $G_{\Omega_j}(z,z_j)$,
there exist a $\chi_{j,z_j}$ and a multiplicative function $f_{z_j}\in\mathcal{O}^{\chi_{j,z_j}}(\Omega_j)$, such that $|f_{z_j}(z)|=P_j^{*}e^{G_{\Omega_j}(z,z_j)}$ (see \cite{suita72}).

Let $Z_0=\{z_0\}=\{(z_1,z_2,...,z_n)\}\subset M$.
Let $\psi=\max_{1\le j\le n}\{2p_j\pi_j^{*}(G_{\Omega_j}(\cdot,z_j))\}$, where $p_j$ is positive real number for $1\le j\le n$.
Let $w_j$ be a local coordinate on a neighborhood $V_{z_j}$ of $z_j\in\Omega_j$ satisfying $w_j(z_j)=0$. Denote that $V_0:=\prod_{1\le j\le n}V_{z_j}$, and $w:=(w_1,w_2,...,w_n)$ is a local coordinate on $V_0$ of $z_0\in M$.
Let $f$ be a holomorphic $(n,0)$ form on $V_0$.

We present a characterization of the concavity of $G(h^{-1}(r))$ degenerating to linearity for the case $Z_0$ is a single point set as follows.

\begin{Theorem}
	\label{thm:linear-2d}
	Assume that $G(0)\in(0,+\infty)$.  $G(h^{-1}(r))$ is linear with respect to $r\in(0,\int_{0}^{+\infty}c(t)e^{-t}dt]$  if and only if the  following statements hold:
	
	$(1)$ $f=(\sum_{\alpha\in E}d_{\alpha}w^{\alpha}+g_0)dw_1\wedge dw_2\wedge...\wedge dw_n$ on $V_0$, where $E=\{(\alpha_1,\alpha_2,...,\alpha_n):\sum_{1\le j\le n}\frac{\alpha_j+1}{p_j}=1\,\&\,\alpha_j\in\mathbb{Z}_{\ge0}\}\not=\emptyset$, $d_{\alpha}\in\mathbb{C}$ such that $\sum_{\alpha\in E}|d_{\alpha}|\not=0$ and $g_0$ is a holomorphic function on $V_0$ such that $(g_0,z_0)\in\mathcal{I}(\psi)_{z_0}$;
	
	$(2)$ $\varphi_j=2\log|g_j|+2u_j$, where $g_j$ is a holomorphic function on $\Omega_j$ such that $g_j(z_j)\not=0$ and $u_j$ is a harmonic function on $\Omega_j$ for any $1\le j\le n$;

    $(3)$ $\chi_{j,z_j}^{\alpha_j+1}=\chi_{j,-u_j}$ for any $j\in\{1,2,...,n\}$ and $\alpha\in E$ satisfying $d_{\alpha}\not=0$.
\end{Theorem}

Let $c_j(z)$ be the logarithmic capacity (see \cite{S-O69}) on $\Omega_j$, which is locally defined by
$$c_j(z_j):=\exp\lim_{z\rightarrow z_j}(G_{\Omega_j}(z,z_j)-\log|w_j(z)|).$$
\begin{Remark}
	\label{r:1.1}When the three statements in Theorem \ref{thm:linear-2d} hold,
$$\sum_{\alpha\in E}\tilde{d}_{\alpha}\wedge_{1\le j\le n}\pi_j^*(g_j(P_j)_*(f_{u_j}f_{z_j}^{\alpha_j}df_{z_j}))$$
 is the unique holomorphic $(n,0)$ form $F$ on $M$ such that $(F-f,z_0)\in(\mathcal{O}(K_{M}))_{z_0}\otimes\mathcal{I}(\psi)_{z_0}$ and
	$$G(t)=\int_{\{\psi<-t\}}|F|^2e^{-\varphi}c(-\psi)=(\int_t^{+\infty}c(s)e^{-s}ds)\sum_{\alpha\in E}\frac{|d_{\alpha}|^2(2\pi)^ne^{-\varphi(z_{0})}}{\Pi_{1\le j\le n}(\alpha_j+1)c_{j}(z_j)^{2\alpha_{j}+2}}$$
	 for any $t\ge0$, where $f_{u_j}$ is a holomorphic function on $\Delta$ such that $|f_{u_j}|=P_j^*(e^{u_j})$ for any $j\in\{1,2,...,n\}$, $f_{z_j}$ is a holomorphic function on $\Delta$ such that $|f_{z_j}|=P_j^*(e^{G_{\Omega_j}(\cdot,z_j)})$ for any $j\in\{1,2,...,n\}$ and $\tilde{d}_{\alpha}$ is a constant such that $\tilde{d}_{\alpha}=\lim_{z\rightarrow z_0}\frac{d_{\alpha}w^{\alpha}dw_1\wedge dw_2\wedge...\wedge dw_n}{\wedge_{1\le j\le n}\pi_j^*(g_j(P_j)_*(f_{u_j}f_{z_j}^{\alpha_j}df_{z_j}))}$ for any $\alpha\in E$. We prove the remark in Section \ref{sec:pf-r1.1}.
\end{Remark}

 Let $Z_j=\{z_{j,1},z_{j,2},...,z_{j,m_j}\}\subset\Omega_j$ for any  $j\in\{1,2,...,n\}$, where $m_j$ is a positive integer.
Let $\psi=\max_{1\le j\le n}\{\pi_j^*(2\sum_{1\le k\le m_j}p_{j,k}G_{\Omega_j}(\cdot,z_{j,k}))\}$.

Let $w_{j,k}$ be a local coordinate on a neighborhood $V_{z_{j,k}}\Subset\Omega_{j}$ of $z_{j,k}\in\Omega_j$ satisfying $w_{j,k}(z_{j,k})=0$ for any $j\in\{1,2,...,n\}$ and $k\in\{1,2,...,m_j\}$, where $V_{z_{j,k}}\cap V_{z_{j,k'}}=\emptyset$ for any $j$ and $k\not=k'$. Denote that $I_1:=\{(\beta_1,\beta_2,...,\beta_n):1\le \beta_j\le m_j$ for any $j\in\{1,2,...,n\}\}$, $V_{\beta}:=\prod_{1\le j\le n}V_{z_{j,\beta_j}}$ for any $\beta=(\beta_1,\beta_2,...,\beta_n)\in I_1$ and $w_{\beta}:=(w_{1,\beta_1},w_{2,\beta_2},...,w_{n,\beta_n})$ is a local coordinate on $V_{\beta}$ of $z_{\beta}:=(z_{1,\beta_1},z_{2,\beta_2},...,z_{n,\beta_n})\in M$.
Let $f$ be a holomorphic $(n,0)$ form on $\cup_{\beta\in I_1}V_{\beta}$ such that $f=w_{\beta^*}^{\alpha_{\beta_*}}dw_{1,1}\wedge dw_{2,1}\wedge...\wedge dw_{n,1}$ on $V_{\beta^*}$, where $\beta^*=(1,1,...,1)\in I_1$.

We present a characterization of the concavity of $G(h^{-1}(r))$ degenerating to linearity for the case $Z_j$ is a  set of finite points as follows.

\begin{Theorem}
	\label{thm:prod-finite-point}Assume that $G(0)\in(0,+\infty)$.  $G(h^{-1}(r))$ is linear with respect to $r\in(0,\int_0^{+\infty} c(s)e^{-s}ds]$ if and only if the following statements hold:

	$(1)$ $\varphi_j=2\log|g_j|+2u_j$ for any $j\in\{1,2,...,n\}$, where $u_j$ is a harmonic function on $\Omega_j$ and $g_j$ is a holomorphic function on $\Omega_j$ satisfying $g_j(z_{j,k})\not=0$ for any $k\in\{1,2,...,m_j\}$;
	
	$(2)$ There exists a nonnegative integer $\gamma_{j,k}$ for any $j\in\{1,2,...,n\}$ and $k\in\{1,2,...,m_j\}$, which satisfies that $\Pi_{1\le k\leq m_j}\chi_{j,z_{j,k}}^{\gamma_{j,k}+1}=\chi_{j,-u_j}$ and $\sum_{1\le j\le n}\frac{\gamma_{j,\beta_j}+1}{p_{j,\beta_j}}=1$ for any $\beta\in I_1$;
	
	$(3)$ $f=(c_{\beta}\Pi_{1\le j\le n}w_{j,\beta_j}^{\gamma_{j,\beta_j}}+g_{\beta})dw_{1,\beta_1}\wedge dw_{2,\beta_2}\wedge...\wedge dw_{n,\beta_n}$ on $V_{\beta}$ for any $\beta\in I_1$, where $c_{\beta}$ is a constant and $g_{\beta}$ is a holomorphic function on $V_{\beta}$ such that $(g_{\beta},z_{\beta})\in\mathcal{I}(\psi)_{z_{\beta}}$;
	
	$(4)$ $\lim_{z\rightarrow z_{\beta}}\frac{c_{\beta}\Pi_{1\le j\le n}w_{j,\beta_j}^{\gamma_{j,\beta_j}}dw_{1,\beta_1}\wedge dw_{2,\beta_2}\wedge...\wedge dw_{n,\beta_n}}{\wedge_{1\le j\le n}\pi_{j}^*(g_j(P_{j})_*(f_{u_j}(\Pi_{1\le k\le m_j}f_{z_{j,k}}^{\gamma_{j,k}+1})(\sum_{1\le k\le m_j}p_{j,k}\frac{df_{z_{j,k}}}{f_{z_{j,k}}})))}=c_0$ for any $\beta\in I_1$, where $c_0\in\mathbb{C}\backslash\{0\}$ is a constant independent of $\beta$, $f_{u_j}$ is a holomorphic function $\Delta$ such that $|f_{u_j}|=P_j^*(e^{u_j})$ and $f_{z_{j,k}}$ is a holomorphic function on $\Delta$ such that $|f_{z_{j,k}}|=P_j^*(e^{G_{\Omega_j}(\cdot,z_{j,k})})$ for any $j\in\{1,2,...,n\}$ and $k\in\{1,2,...,m_j\}$.
\end{Theorem}

Denote that
\begin{equation*}
c_{j,k}:=\exp\lim_{z\rightarrow z_{j,k}}(\frac{\sum_{1\le k_1\le m_j}p_{j,k_1}G_{\Omega_j}(z,z_{j,k_1})}{p_{j,k}}-\log|w_{j,k}(z)|)
\end{equation*}
 for any $j\in\{1,2,...,n\}$ and $k\in\{1,2,...,m_j\}$.
\begin{Remark}
	\label{r:1.2}When the four statements in Theorem \ref{thm:prod-finite-point} hold,
$$c_0\wedge_{1\le j\le n}\pi_{j}^*(g_j(P_{j})_*(f_{u_j}(\Pi_{1\le k\le m_j}f_{z_{j,k}}^{\gamma_{j,k}+1})(\sum_{1\le k\le m_j}p_{j,k}\frac{df_{z_{j,k}}}{f_{z_{j,k}}})))$$
 is the unique holomorphic $(n,0)$ form $F$ on $M$ such that $(F-f,z_\beta)\in(\mathcal{O}(K_{M}))_{z_\beta}\otimes\mathcal{I}(\psi)_{z_\beta}$ for any $\beta\in I_1$ and
	$$G(t)=\int_{\{\psi<-t\}}|F|^2e^{-\varphi}c(-\psi)=(\int_{t}^{+\infty}c(s)e^{-s}ds)\sum_{\beta\in I_1}\frac{|c_{\beta}|^2(2\pi)^ne^{-\varphi(z_{\beta})}}{\Pi_{1\le j\le n}(\gamma_{j,\beta_j}+1)c_{j,\beta_j}^{2\gamma_{j,\beta_j}+2}}$$
	 for any $t\ge0$. We prove the remark in Section \ref{sec:pf-r1.2}.
\end{Remark}

 Let ${Z}_j=\{z_{j,k}:1\le k<\tilde m_j\}$ be a discrete subset of $\Omega_j$ for any  $j\in\{1,2,...,n\}$, where $\tilde{m}_j\in\mathbb{Z}_{\ge2}\cup\{+\infty\}$.
Let $p_{j,k}$ be a positive number such that $\sum_{1\le k<\tilde{m}_j}p_{j,k}G_{\Omega_j}(\cdot,z_{j,k})\not\equiv-\infty$ for any $j$.
Let $\psi=\max_{1\le j\le n}\{\pi_j^*(2\sum_{1\le k<\tilde{m}_j}p_{j,k}G_{\Omega_j}(\cdot,z_{j,k}))\}$. Assume that $\limsup_{t\rightarrow+\infty}c(t)<+\infty$.

Let $w_{j,k}$ be a local coordinate on a neighborhood $V_{z_{j,k}}\Subset\Omega_{j}$ of $z_{j,k}\in\Omega_j$ satisfying $w_{j,k}(z_{j,k})=0$ for any $j\in\{1,2,...,n\}$ and $1\le k<\tilde{m}_j$, where $V_{z_{j,k}}\cap V_{z_{j,k'}}=\emptyset$ for any $j$ and $k\not=k'$. Denote that $\tilde I_1:=\{(\beta_1,\beta_2,...,\beta_n):1\le \beta_j< \tilde m_j$ for any $j\in\{1,2,...,n\}\}$, $V_{\beta}:=\prod_{1\le j\le n}V_{z_{j,\beta_j}}$ for any $\beta=(\beta_1,\beta_2,...,\beta_n)\in\tilde I_1$ and $w_{\beta}:=(w_{1,\beta_1},w_{2,\beta_2},...,w_{n,\beta_n})$ is a local coordinate on $V_{\beta}$ of $z_{\beta}:=(z_{1,\beta_1},z_{2,\beta_2},...,z_{n,\beta_n})\in M$.
Let $f$ be a holomorphic $(n,0)$ form on $\cup_{\beta\in \tilde I_1}V_{\beta}$ such that $f=w_{\beta^*}^{\alpha_{\beta_*}}dw_{1,1}\wedge dw_{2,1}\wedge...\wedge dw_{n,1}$ on $V_{\beta^*}$, where $\beta^*=(1,1,...,1)\in \tilde I_1$.

We present that $G(h^{-1}(r))$ is not linear when there exists $j_0\in\{1,2,...,n\}$ such that $\tilde m_{j_0}=+\infty$ as follows.

\begin{Theorem}
	\label{thm:prod-infinite-point}If $G(0)\in(0,+\infty)$ and there exists $j_0\in\{1,2,...,n\}$ such that $\tilde m_{j_0}=+\infty$, then $G(h^{-1}(r))$ is not linear with respect to $r\in(0,\int_0^{+\infty} c(s)e^{-s}ds]$.
	\end{Theorem}

\subsection{Applications: characterizations of the holding of equality in optimal jets $L^2$ extension problem}
\

In this section, we give characterizations of the holding of equality in optimal jets $L^2$ extension problem.

Let $\Omega_j$  be an open Riemann surface, which admits a nontrivial Green function $G_{\Omega_j}$ for any  $1\le j\le n$. Let $M=\prod_{1\le j\le n}\Omega_j$ be an $n-$dimensional complex manifold, and let $\pi_j$ be the natural projection from $M$ to $\Omega_j$. Let $K_M$ be the canonical (holomorphic) line bundle on $M$, and Let $K_{\Omega_j}$ be the canonical (holomorphic) line bundle on $\Omega_j$.

Let $z_0=(z_1,z_2,...,z_n)\in M$.
Let $w_j$ be a local coordinate on a neighborhood $V_{z_j}$ of $z_j\in\Omega_j$ satisfying $w_j(z_j)=0$. Denote that $V_0:=\prod_{1\le j\le n}V_{z_j}$, and $w:=(w_1,w_2,...,w_n)$ is a local coordinate on $V_0$ of $z_0\in M$.
Let $f=\sum_{\alpha\in E}d_{\alpha}w^{\alpha}dw_1\wedge dw_2\wedge...\wedge dw_n$ be a holomorphic $(n,0)$ form on $V_0$,  where $E=\{(\alpha_1,\alpha_2,...,\alpha_n):\sum_{1\le j\le n}\frac{\alpha_j+1}{p_j}=1\,\&\,\alpha_j\in\mathbb{Z}_{\ge0}\}$ and $\sum_{\alpha\in E}|d_{\alpha}|\not=0$.

We obtain a characterization of the holding of equality in optimal jets $L^2$ extension problem from single points  to products of open Riemann surfaces.

\begin{Theorem}
\label{thm:2d-jet}
Let $\Psi\le0$ be a  plurisubharmonic function on $M$, and let $\varphi_j$ be a Lebesgue measurable function on $\Omega_j$ such that $\Psi+\sum_{1\le j\le n}\pi_j^*(\varphi_j)$ is plurisubharmonic on $M$ and $(\Psi+\sum_{1\le j\le n}\pi_j^*(\varphi_j))(z_0)>-\infty$. Denote that
$$\psi:=\max_{1\le j\le n}\{2p_j\pi_j^{*}(G_{\Omega_j}(\cdot,z_j))\}+\Psi$$ and $\varphi:=\sum_{1\le j\le n}\pi_j^*(\varphi_j)$ on $M$, where $p_j$ is a positive real number for $1\le j\le n$.
Let $c$ be a positive function on $(0,+\infty)$ such that $\int_{0}^{+\infty}c(t)e^{-t}dt<+\infty$ and $c(t)e^{-t}$ is decreasing on $(0,+\infty)$.

Then there exists a holomorphic $(n,0)$ form $F$ on $M$ satisfying that $(F-f,z_0)\in(\mathcal{O}(K_M)\otimes\mathcal{I}(\max_{1\le j\le n}\{2p_j\pi_j^{*}(G_{\Omega_j}(\cdot,z_j))\}))_{z_0}$ and
	$$\int_{M}|F|^2e^{-\varphi}c(-\psi)\le(\int_0^{+\infty}c(s)e^{-s}ds)\sum_{\alpha\in E}\frac{|d_{\alpha}|^2(2\pi)^ne^{-(\varphi+\Psi)(z_{0})}}{\Pi_{1\le j\le n}(\alpha_j+1)c_{j}(z_j)^{2\alpha_{j}+2}}.$$
	
	Moreover, equality $(\int_0^{+\infty}c(s)e^{-s}ds)\sum_{\alpha\in E}\frac{|d_{\alpha}|^2(2\pi)^ne^{-(\varphi+\Psi)(z_{0})}}{\Pi_{1\le j\le n}(\alpha_j+1)c_{j}(z_j)^{2\alpha_{j}+2}}=\inf\{\int_{M}|\tilde{F}|^2e^{-\varphi}c(-\psi):\tilde{F}\in H^0(M,\mathcal{O}(K_M))\,\&\, (\tilde{F}-f,z_0)\in(\mathcal{O}(K_{M})\otimes\mathcal{I}(\max_{1\le j\le n}\{2p_j\pi_j^{*}(G_{\Omega_j}(\cdot,z_j))\}))_{z_0}\}$ holds if and only if the following statements hold:

	$(1)$ $\Psi\equiv0$;
	
	$(2)$  $\varphi_j=2\log|g_j|+2u_j$, where $g_j$ is  holomorphic functions on $\Omega_j$ such that $g_j(z_j)\not=0$ and $u_j$ is a harmonic function on $\Omega_j$ for any $1\le j\le n$;

    $(3)$ $\chi_{j,z_j}^{\alpha_j+1}=\chi_{j,-u_j}$ for any $j\in\{1,2,...,n\}$ and $\alpha\in E$ satisfying $d_{\alpha}\not=0$.
\end{Theorem}

\begin{Remark}
	Let $f=\sum_{\alpha\in\mathbb{Z}_{\ge0}}d_{\alpha}w^{\alpha}dw_1\wedge dw_2\wedge...\wedge dw_n$ on $V_0$, and let $\psi=\max_{1\le j\le n}\{2p_j\pi_j^{*}(G_{\Omega_j}(\cdot,z_j))\}$. It follows from Lemma \ref{l:0} that $(f,z_0)\not\in(\mathcal{O}(K_M)\otimes\mathcal{I}(\psi))_{z_0}$ if and only if there exists $\alpha\in\mathbb{Z}_{\ge0}$ satisfying $\sum_{1\le j\le n}\frac{\alpha_j+1}{p_j}\le 1$ and $d_{\alpha}\not=0$.  It follows from Lemma \ref{l:m1} that $\lim_{t\rightarrow+\infty}\int_{\{-t-1<\psi<-t\}}|f|^2e^{-\psi}<+\infty$ if and only if $d_{\alpha}=0$ for any $\alpha$ satisfying $\sum_{1\le j\le n}\frac{\alpha_j+1}{p_j}<1$. Thus, we set $f=\sum_{\alpha\in E}d_{\alpha}w^{\alpha}dw_1\wedge dw_2\wedge...\wedge dw_n$ in the above theorem.
\end{Remark}

 Let $Z_j=\{z_{j,1},z_{j,2},...,z_{j,m_j}\}\subset\Omega_j$ for any  $j\in\{1,2,...,n\}$, where $m_j$ is a positive integer.
Let $w_{j,k}$ be a local coordinate on a neighborhood $V_{z_{j,k}}\Subset\Omega_{j}$ of $z_{j,k}\in\Omega_j$ satisfying $w_{j,k}(z_{j,k})=0$ for any $j\in\{1,2,...,n\}$ and $k\in\{1,2,...,m_j\}$, where $V_{z_{j,k}}\cap V_{z_{j,k'}}=\emptyset$ for any $j$ and $k\not=k'$. Denote that $I_1:=\{(\beta_1,\beta_2,...,\beta_n):1\le \beta_j\le m_j$ for any $j\in\{1,2,...,n\}\}$, $V_{\beta}:=\prod_{1\le j\le n}V_{z_{j,\beta_j}}$ for any $\beta=(\beta_1,\beta_2,...,\beta_n)\in I_1$ and $w_{\beta}:=(w_{1,\beta_1},w_{2,\beta_2},...,w_{n,\beta_n})$ is a local coordinate on $V_{\beta}$ of $z_{\beta}:=(z_{1,\beta_1},z_{2,\beta_2},...,z_{n,\beta_n})\in M$.

Let $f$ be a holomorphic $(n,0)$ form on $\cup_{\beta\in I_1}V_{\beta}$ such that  $f=\sum_{\alpha\in E_{\beta}}d_{\beta,\alpha}w_{\beta}^{\alpha}dw_{1,\beta_1}\wedge dw_{2,\beta_2}\wedge...\wedge dw_{n,\beta_n}$ on $V_{\beta}$, where  $E_{\beta}:=\{(\alpha_1,\alpha_2,...,\alpha_n):\sum_{1\le j\le n}\frac{\alpha_j+1}{p_{j,\beta_j}}=1\,\&\,\alpha_j\in\mathbb{Z}_{\ge0}\}$. Assume that $f=w_{\beta^*}^{\alpha_{\beta_*}}dw_{1,1}\wedge dw_{2,1}\wedge...\wedge dw_{n,1}$ on $V_{\beta^*}$, where $\beta^*=(1,1,...,1)\in I_1$. Denote that
\begin{equation*}
c_{j,k}:=\exp\lim_{z\rightarrow z_{j,k}}(\frac{\sum_{1\le k_1\le m_j}p_{j,k_1}G_{\Omega_j}(z,z_{j,k_1})}{p_{j,k}}-\log|w_{j,k}(z)|)
\end{equation*}
 for any $j\in\{1,2,...,n\}$ and $k\in\{1,2,...,m_j\}$.

We obtain a characterization of the holding of equality in optimal jets $L^2$ extension problem from products of finite subsets to products of open Riemann surfaces.

\begin{Theorem}
\label{thm:prod-finite-jet}
Let $\Psi\le0$ be a  plurisubharmonic function on $M$, and let $\varphi_j$ be a Lebesgue measurable function on $\Omega_j$ such that $\Psi+\sum_{1\le j\le n}\pi_j^*(\varphi_j)$ is plurisubharmonic on $M$ and $(\Psi+\sum_{1\le j\le n}\pi_j^*(\varphi_j))(z_\beta)>-\infty$ for any $\beta\in I_1$. Denote that
$$\psi:=\max_{1\le j\le n}\{2\sum_{1\le k\le m_j}p_{j,k}\pi_j^{*}(G_{\Omega_j}(\cdot,z_{j,k}))\}+\Psi$$
and $\varphi:=\sum_{1\le j\le n}\pi_j^*(\varphi_{j})$
 on $M$, where $p_{j,k}$ is positive real number for $1\le j\le n$ and $k\in\{1,2,...,m_j\}$.
Let $c$ be a positive function on $(0,+\infty)$ such that $\int_{0}^{+\infty}c(t)e^{-t}dt<+\infty$ and $c(t)e^{-t}$ is decreasing on $(0,+\infty)$.

Then there exists a holomorphic $(n,0)$ form $F$ on $M$ satisfying that $(F-f,z_0)\in(\mathcal{O}(K_M)\otimes\mathcal{I}(\max_{1\le j\le n}\{2\sum_{1\le k\le m_j}p_{j,k}\pi_j^{*}(G_{\Omega_j}(\cdot,z_{j,k}))\}))_{z_\beta}$ for any $\beta\in I_1$ and
	$$\int_{M}|F|^2e^{-\varphi}c(-\psi)\le(\int_{0}^{+\infty}c(s)e^{-s}ds)\sum_{\beta\in I_1}\sum_{\alpha\in E_{\beta}}\frac{|d_{\beta,\alpha}|^2(2\pi)^ne^{-(\varphi+\Psi)(z_{\beta})}}{\Pi_{1\le j\le n}(\alpha_j+1)c_{j,\beta_j}^{2\alpha_{j}+2}}.$$
	
	Moreover, equality $\inf\{\int_{M}|\tilde{F}|^2e^{-\varphi}c(-\psi):\tilde{F}\in H^0(M,\mathcal{O}(K_M))\,\&\,(\tilde{F}-f,z_\beta)\in(\mathcal{O}(K_{M})\otimes\mathcal{I}(\max_{1\le j\le n}\{2\sum_{1\le k\le m_j}p_{j,k}\pi_j^{*}(G_{\Omega_j}(\cdot,z_{j,k}))\}))_{z_\beta}$ for any $\beta\in I_1\}=(\int_{0}^{+\infty}c(s)e^{-s}ds)\sum_{\beta\in I_1}\sum_{\alpha\in E_{\beta}}\frac{|d_{\beta,\alpha}|^2(2\pi)^ne^{-(\varphi+\Psi)(z_{\beta})}}{\Pi_{1\le j\le n}(\alpha_j+1)c_{j,\beta_j}^{2\alpha_{j}+2}}$ holds if and only if the following statements hold:

	$(1)$ $\Psi\equiv0$ and $\varphi_j=2\log|g_j|+2u_j$ for any $j\in\{1,2,...,n\}$, where $u_j$ is a harmonic function on $\Omega_j$ and $g_j$ is a holomorphic function on $\Omega_j$ satisfying $g_j(z_{j,k})\not=0$ for any $k\in\{1,2,...,m_j\}$;
	
	$(2)$  there exists a nonnegative integer $\gamma_{j,k}$ for any $j\in\{1,2,...,n\}$ and $k\in\{1,2,...,m_j\}$, which satisfies that $\Pi_{1\le k\leq m_j}\chi_{j,z_{j,k}}^{\gamma_{j,k}+1}=\chi_{j,-u_j}$ and $\sum_{1\le j\le n}\frac{\gamma_{j,\beta_j}+1}{p_{j,\beta_j}}=1$ for any $\beta\in I_1$;
	
	$(3)$ $f=(c_{\beta}\Pi_{1\le j\le n}w_{j,\beta_j}^{\gamma_{j,\beta_j}}+g_{\beta})dw_{1,\beta_1}\wedge dw_{2,\beta_2}\wedge...\wedge dw_{n,\beta_n}$ on $V_{\beta}$ for any $\beta\in I_1$, where $c_{\beta}$ is a constant and $g_{\beta}$ is a holomorphic function on $V_{\beta}$ such that $(g_{\beta},z_{\beta})\in\mathcal{I}(\psi)_{z_{\beta}}$;
	
	$(4)$ $\lim_{z\rightarrow z_{\beta}}\frac{c_{\beta}\Pi_{1\le j\le n}w_{j,\beta_j}^{\gamma_{j,\beta_j}}dw_{1,\beta_1}\wedge dw_{2,\beta_2}\wedge...\wedge dw_{n,\beta_n}}{\wedge_{1\le j\le n}\pi_{j}^*(g_j(P_{j})_*(f_{u_j}(\Pi_{1\le k\le m_j}f_{z_{j,k}}^{\gamma_{j,k}+1})(\sum_{1\le k\le m_j}p_{j,k}\frac{df_{z_{j,k}}}{f_{z_{j,k}}})))}=c_0$ for any $\beta\in I_1$, where $c_0\in\mathbb{C}\backslash\{0\}$ is a constant independent of $\beta$, $f_{u_j}$ is a holomorphic function $\Delta$ such that $|f_{u_j}|=P_j^*(e^{u_j})$ and $f_{z_{j,k}}$ is a holomorphic function on $\Delta$ such that $|f_{z_{j,k}}|=P_j^*(e^{G_{\Omega_j}(\cdot,z_{j,k})})$ for any $j\in\{1,2,...,n\}$ and $k\in\{1,2,...,m_j\}$.

\end{Theorem}

 Let ${Z}_j=\{z_{j,k}:1\le k<\tilde m_j\}$ be a discrete subset of $\Omega_j$ for any  $j\in\{1,2,...,n\}$, where $\tilde{m}_j\in\mathbb{Z}_{\ge2}\cup\{+\infty\}$.

Let $w_{j,k}$ be a local coordinate on a neighborhood $V_{z_{j,k}}\Subset\Omega_{j}$ of $z_{j,k}\in\Omega_j$ satisfying $w_{j,k}(z_{j,k})=0$ for any $j\in\{1,2,...,n\}$ and $1\le k<\tilde{m}_j$, where $V_{z_{j,k}}\cap V_{z_{j,k'}}=\emptyset$ for any $j$ and $k\not=k'$. Denote that $\tilde I_1:=\{(\beta_1,\beta_2,...,\beta_n):1\le \beta_j< \tilde m_j$ for any $j\in\{1,2,...,n\}\}$, $V_{\beta}:=\prod_{1\le j\le n}V_{z_{j,\beta_j}}$ for any $\beta=(\beta_1,\beta_2,...,\beta_n)\in\tilde I_1$ and $w_{\beta}:=(w_{1,\beta_1},w_{2,\beta_2},...,w_{n,\beta_n})$ is a local coordinate on $V_{\beta}$ of $z_{\beta}:=(z_{1,\beta_1},z_{2,\beta_2},...,z_{n,\beta_n})\in M$.

Let $f$ be a holomorphic $(n,0)$ form on $\cup_{\beta\in \tilde I_1}V_{\beta}$ such that  $f=\sum_{\alpha\in E_{\beta}}d_{\beta,\alpha}w_{\beta}^{\alpha}dw_{1,\beta_1}\wedge dw_{2,\beta_2}\wedge...\wedge dw_{n,\beta_n}$ on $V_{\beta}$, where  $E_{\beta}:=\{(\alpha_1,\alpha_2,...,\alpha_n):\sum_{1\le j\le n}\frac{\alpha_j+1}{p_{j,\beta_j}}=1\,\&\,\alpha_j\in\mathbb{Z}_{\ge0}\}$. Assume that $f=w_{\beta^*}^{\alpha_{\beta_*}}dw_{1,1}\wedge dw_{2,1}\wedge...\wedge dw_{n,1}$ on $V_{\beta^*}$, where $\beta^*=(1,1,...,1)\in\tilde I_1$. Denote that
\begin{equation*}
c_{j,k}:=\exp\lim_{z\rightarrow z_{j,k}}(\frac{\sum_{1\le k_1<\tilde m_j}p_{j,k_1}G_{\Omega_j}(z,z_{j,k_1})}{p_{j,k}}-\log|w_{j,k}(z)|)
\end{equation*}
 for any $j\in\{1,2,...,n\}$ and $1\le k<\tilde m_j$ (following from Lemma \ref{l:green-sup} and Lemma \ref{l:green-sup2}, we know the above limit exists).

When the products of analytic subsets is infinite, we obtain that the equality in optimal jets $L^2$ extension problem could not hold.

\begin{Theorem}
\label{thm:prod-infinite-jet}
Let $\Psi\le0$ be a  plurisubharmonic function on $M$, and let $\varphi_j$ be a Lebesgue measurable function on $\Omega_j$ such that  $\Psi+\sum_{1\le j\le n}\pi_j^*(\varphi_j)$ is plurisubharmonic on $M$ and $(\Psi+\sum_{1\le n}\pi_j^*(\varphi))(z_{\beta})>-\infty$ for any $\beta\in\tilde I_1$. Let $p_{j,k}$ be a positive number such that $\sum_{1\le k<\tilde{m}_j}p_{j,k}G_{\Omega_j}(\cdot,z_{j,k})\not\equiv-\infty$ for any $j$. Denote that
$$\psi:=\max_{1\le j\le n}\{2\sum_{1\le k<\tilde m_j}p_{j,k}\pi_j^{*}(G_{\Omega_j}(\cdot,z_{j,k}))\}+\Psi$$
and $\varphi:=\sum_{1\le j\le n}\pi_j^*(\varphi_{j})$
 on $M$.
Let $c$ be a positive function on $(0,+\infty)$ such that $\int_{0}^{+\infty}c(t)e^{-t}dt<+\infty$, $c(t)e^{-t}$ is decreasing on $(0,+\infty)$ and $\limsup_{t\rightarrow+\infty}c(t)<+\infty$. Assume that
$$\sum_{\beta\in\tilde I_1}\sum_{\alpha\in E_{\beta}}\frac{|d_{\beta,\alpha}|^2(2\pi)^ne^{-(\varphi+\Psi)(z_{\beta})}}{\Pi_{1\le j\le n}(\alpha_j+1)c_{j,\beta_j}^{2\alpha_{j}+2}}<+\infty$$
 and there exists $j_0\in\{1,2,...,n\}$ such that $\tilde m_{j_0}=+\infty$.

Then there exists a holomorphic $(n,0)$ form $F$ on $M$ satisfying that $(F-f,z_0)\in(\mathcal{O}(K_M)\otimes\mathcal{I}(\max_{1\le j\le n}\{2\sum_{1\le k<\tilde m_j}p_{j,k}\pi_j^{*}(G_{\Omega_j}(\cdot,z_{j,k}))\}))_{z_\beta}$ for any $\beta\in \tilde I_1$ and
	$$\int_{M}|F|^2e^{-\varphi}c(-\psi)<(\int_{0}^{+\infty}c(s)e^{-s}ds)\sum_{\beta\in\tilde I_1}\sum_{\alpha\in E_{\beta}}\frac{|d_{\beta,\alpha}|^2(2\pi)^ne^{-(\varphi+\Psi)(z_{\beta})}}{\Pi_{1\le j\le n}(\alpha_j+1)c_{j,\beta_j}^{2\alpha_{j}+2}}.$$
\end{Theorem}

\subsubsection{\textbf{Suita conjecture, extended Suita conjecture, and a conjecture of Ohsawa}}\label{sec:1.2.1}

\
\

In this section, we present characterizations of the product versions of the equality parts of Suita conjecture and extended Suita conjecture, and the equality holding of a conjecture of Ohsawa for products of open Riemann surfaces.

Let $M$ be an $n-$dimensional complex manifold with a continuous volume form $dV_M$, and let $S$ be a closed complex submanifold of $M$. We consider a class of continuous plurisubharmonic function $\Psi$ from $M$ to the interval $[-\infty,0)$ such that

$(1)$ $S\subset\Psi^{-1}(-\infty)$;

$(2)$ If $S$ is $l$-dimensional around a point $x$, there exists a local coordinate $(z_1,...,z_n)$ on a neighborhood $U$ of $x$ such that $z_{l+1}=...=z_n$ on $S\cap U$ and
$$\sup_{U\backslash S}|\Psi(z)-(n-l)\log\sum_{j=l+1}^n|z_j|^2|<+\infty.$$

The set of such polar functions $\Psi$ will be denoted by $\Delta(S)$. For each $\Psi\in\Delta(S)$, one can associate a positive measure $dV_{M}[\Psi]$ on $S$ as the minimum element of the partial ordered set of positive measures $d\mu$ satisfying
$$\int_{S_l}fd\mu\ge\limsup_{t\rightarrow+\infty}\frac{
2(n-l)}{\sigma_{2n-2l-1}}\int_Mfe^{-\Psi}\mathbb{I}_{\{-1-t<\Psi<-t\}}dV_M$$
for any nonnegative continuous function $f$ with $Suppf\Subset M$. Here $S_l$ denotes the $l-$dimensional component of $S$ and $\sigma_m$ denotes the volume of the unit sphere in $\mathbb{R}^{m+1}$.

 Denote the space of $L^2$ integrable holomorphic section of $K_M$ by $A^2(M,K_M,dV_M^{-1},dV_M)$. A holomorphic section $f$ of $K_M|_S$ on $S$ is called $L^2$ integrable with respect to the measure $dV_M[\Psi]$, if $\int_{S}\frac{|f|^2}{dV_M}dV_M[\Psi]<+\infty$.  Denote the space of holomorphic section of $K_M|_S$ which is $L^2$ integrable with respect to the measure $dV_M[\Psi]$ by $A^2(S,K_M|_S,dV_M^{-1},dV_M[\Psi])$.

If $\Delta(S)$ is non-empty, we set $G(z,S):=(\sup\{u(z):u\in\Delta(S)\})^*$, which is the upper envelope of $\sup\{u(z):u\in\Delta(S)\}$. We have $G(z,S)\in\Delta(S)$ (see \cite{Ohsawa5}). Let $M$ be a Stein manifold, and let $\{\sigma_l\}_{l=1}^{+\infty}$ (resp. $\{\tau_l\}_{l=1}^{+\infty}$) be a complete orthogonal system of $A^2(M,K_M,dV_M^{-1},dV_M)$ (resp. $A^2(S,K_M|_S,dV_M^{-1},dV_M[G(\cdot,S)])$) satisfying $(\sqrt{-1})^{n^2}\int_{M}\frac{\sigma_i}{\sqrt{2^n}}\wedge\frac{\overline{\sigma}_j}{\sqrt{2^n}}=\delta_i^j$. Put $\kappa_M=\sum_{l=1}^{+\infty}\sigma_l\otimes\overline\sigma_l\in C^{\omega}(M,K_M\otimes\overline{K_M})$ (resp. $\kappa_{M/S}=\sum_{l=1}^{+\infty}\tau_l\otimes\overline\tau_l\in C^{\omega}(S,K_M\otimes K_M)$). A conjecture of Ohsawa (see \cite{Ohsawa5}) is stated below.
\begin{Conjecture}
$\frac{\pi^k}{k!}\kappa_M(x)\ge\kappa_{M/S}(x)$ for any $x\in S_{n-k}$.	
\end{Conjecture}
In \cite{GZ15}, Guan-Zhou prove the above conjecture. In the following, we give a characterization of the holding of equality in the above conjecture for the case $M$ is a product of open Riemann surfaces and $S$ is $0-$dimensional.

Let $\Omega_j$  be an open Riemann surface, which admits a nontrivial Green function $G_{\Omega_j}$ for any  $1\le j\le n$. Let $M=\prod_{1\le j\le n}\Omega_j$ be an $n-$dimensional complex manifold, and let $\pi_j$ be the natural projection from $M$ to $\Omega_j$. Let $S$ be a $0-$dimensional closed complex submanifold of $M$ (i.e. a discrete subset of $M$).
\begin{Theorem}
	\label{thm:ohsawa}$\frac{\pi^{n}}{n!}\kappa_M(x)=\kappa_{M/S}(x)$ for any $x\in S$ holds if and only if the following statements hold:
	
	$(1)$ $S$ is a single point set;
	
	$(2)$ $\Omega_j$ is conformally equivalent to the unit disc less a (possible) closed set of inner capacity zero for any $j\in\{1,2,...,n\}$.
\end{Theorem}

Let $\Omega$  be an open Riemann surface, which admits a nontrivial Green function $G_{\Omega}$. Let $w$ be a local coordinate on a neighborhood $V_{z_0}$ of $z_0\in\Omega$ satisfying $w(z_0)=0$. We define that
$$B_{\Omega}(z)dw\otimes\overline{dw}:=\kappa_{\Omega}|_{V_{z_0}}.$$Let $c_{\beta}(z)$ be the logarithmic capacity (see \cite{S-O69}) which is locally defined by
$$c_{\beta}(z_0):=\exp\lim_{z\rightarrow z_0}(G_{\Omega}(z,z_0)-\log|w(z)|)$$
on $\Omega$.
In \cite{suita72}, Suita stated a conjecture as below.
\begin{Conjecture}
	$c_{\beta}(z_0)^2\le\pi B_{\Omega}(z_0)$ holds for any $z_0\in \Omega$, and equality holds if and only if $\Omega$ is conformally equivalent to the unit disc less a (possible) closed set of inner capacity zero.
\end{Conjecture}
The inequality part of  Suita conjecture for bounded planar domain was proved by B\l ocki \cite{Blo13}, and original form of the inequality was proved by Guan-Zhou \cite{gz12}.
The equality part of Suita conjecture was proved by Guan-Zhou \cite{guan-zhou13ap}, which completed the proof of Suita conjecture.

Let $\Omega_j$  be an open Riemann surface, which admits a nontrivial Green function $G_{\Omega_j}$ for any  $1\le j\le n$. Let $M=\prod_{1\le j\le n}\Omega_j$ be an $n-$dimensional complex manifold, and let $\pi_j$ be the natural projection from $M$ to $\Omega_j$. Let $K_M$ be the canonical (holomorphic) line bundle on $M$. Let $w_j$ be a local coordinate on a neighborhood $V_{z_j}$ of $z_j\in\Omega_j$ satisfying $w_j(z_j)=0$. Denote that $V_0:=\prod_{1\le j\le n}V_{z_j}$, and $w:=(w_1,w_2,...,w_n)$ is a local coordinate on $V_0$ of $z_0=(z_1,z_2,...,z_n)\in M$. We define
$$B_{M}(z)dw_1\wedge dw_2\wedge...\wedge dw_n \otimes\overline{dw_1\wedge dw_2\wedge...\wedge dw_n}:=\kappa_{M}|_{V_{0}}.$$Let $c_{j}(z_j)$ be the logarithmic capacity which is locally defined by
$$c_{j}(z_j):=\exp\lim_{z\rightarrow z_j}(G_{\Omega_j}(z,z_j)-\log|w_j(z)|).$$

 Theorem \ref{thm:2d-jet}  gives a characterization of the holding of equality in the product version of  Suita conjecture.
\begin{Theorem}
	\label{thm:suita}
	$\Pi_{1\le j\le n}c_j(z_j)^{2}\le \pi^n B_M(z_0)$ holds for any $z_0=(z_1,z_2,...,z_n)\in M$, and equality holds if and only if $\Omega_j$ is conformally equivalent to the unit disc less a (possible) closed set of inner capacity zero for any $j\in\{1,2,...,n\}$.
\end{Theorem}

Let $\Omega$  be an open Riemann surface, which admits a nontrivial Green function $G_{\Omega}$, and let $K_{\Omega}$ be the canonical (holomorphic) line bundle on $\Omega$. Let $w$ be a local coordinate on a neighborhood $V_{z_0}$ of $z_0\in\Omega$ satisfying $w(z_0)=0$. Let $\rho=e^{-2u}$ on $\Omega$, where $u$ is a harmonic function on $\Omega$. We define that
$$B_{\Omega,\rho}dw\otimes\overline{dw}:=\sum_{l=1}^{+\infty}\sigma_l\otimes\overline{\sigma}_l|_{V_{z_0}}\in C^{\omega}(V_{z_0},K_{\Omega}\otimes\overline{K_{\Omega}}),$$
where $\{\sigma_l\}_{l=1}^{+\infty}$ are holomorphic $(1,0)$ forms on $\Omega$ satisfying $\sqrt{-1}\int_{\Omega}\rho\frac{\sigma_i}{\sqrt{2}}\wedge\frac{\overline{\sigma}_j}{\sqrt{2}}=\delta_i^j$ and $\{F\in H^0(\Omega,K_{\Omega}):\int_{\Omega}\rho|F|^2<+\infty\,\&\,\int_{\Omega}\rho\sigma_l\wedge\overline F=0$ for any $l\in\mathbb{Z}_{>0}\}=\{0\}$.

In \cite{Yamada}, Yamada  stated a conjecture as below (so-called extended Suita conjecture).
\begin{Conjecture}
	$c_{\beta}(z_0)^2\le\pi \rho(z_0) B_{\Omega,\rho}(z_0)$ holds for any $z_0\in \Omega$, and equality holds if and only $\chi_{-u}=\chi_{z_0}$, where $\chi_{-u}$ and $\chi_{z_0}$ are the characters associated to the functions $-u$ and $G_{\Omega}(\cdot,z_0)$ respectively.
\end{Conjecture}
The inequality part of extended Suita conjecture  was proved by Guan-Zhou \cite{GZ15}.
The equality part of extended Suita conjecture was proved by Guan-Zhou \cite{guan-zhou13ap}.

Let $\rho=e^{-2\sum_{1\le j\le n}\pi_j^*(u_j)}$ on $M$, where $u_j$ is a harmonic function on $\Omega_j$ for any $j\in\{1,2,...,n\}$. We define that
$$B_{M,\rho}dw_1\wedge dw_2\wedge...\wedge dw_n \otimes\overline{dw_1\wedge dw_2\wedge...\wedge dw_n }:=\sum_{l=1}^{+\infty}e_l\otimes\overline{e}_l|_{V_{z_0}}\in C^{\omega}(V_{0},K_{M}\otimes\overline{K_{M}}),$$
where $\{e_l\}_{l=1}^{+\infty}$ are holomorphic $(n,0)$ forms on $M$ satisfying $(\sqrt{-1})^{n^2}\int_{M}\rho\frac{e_i}{\sqrt{2^n}}\wedge\frac{\overline{e}_j}{\sqrt{2^n}}=\delta_i^j$ and $\{F\in H^0(M,K_{M}):\int_{M}\rho|F|^2<+\infty\,\&\,\int_{M}\rho e_l\wedge\overline F=0$ for any $l\in\mathbb{Z}_{>0}\}=\{0\}$. Theorem \ref{thm:2d-jet}  gives  a characterization of the holding of equality in the product version of the extended Suita conjecture.
 \begin{Theorem}
	\label{thm:extend}
	$\Pi_{1\le j\le n}c_j(z_j)^{2}\le \pi^n \rho(z_0) B_{M,\rho}(z_0)$ holds for any $z_0=(z_1,z_2,...,z_n)\in M$, and equality holds if and only if $\chi_{j,-u_j}=\chi_{j,z_j}$ for any $j\in\{1,2,...,n\}$, where $\chi_{j,-u_j}$ and $\chi_{j,z_j}$ are the characters associated to the functions $-u$ and $G_{\Omega}(\cdot,z_0)$ respectively.
\end{Theorem}

\section{Preparation}

\subsection{Concavity property of minimal $L^2$ integrals}
\

In this section, we recall some results about concavity property of minimal $L^2$ integrals (see \cite{GY-concavity,GMY-concavity2,GY-concavity3}).

Let $M$ be a complex manifold. We call $M$ that satisfies condition $(a)$, if there exists a closed subset $X\subset M$ satisfying the following two statements:

$(a1)$ $X$ is locally negligible with respect to $L^2$ holomorphic functions; i.e., for any local coordinate neighborhood $U\subset M$ and for any $L^2$ holomorphic function $f$ on $U\backslash X$, there exists an $L^2$ holomorphic function $\tilde f$ on $U$ such that $\tilde f|_{U\backslash X}=f$ with the same $L^2$ norm;

$(a2)$ $M\backslash X$ is a Stein manifold.

\

Let $M$ be an $n-$dimensional complex manifold satisfying condition $(a)$,  and let $K_{M}$ be the canonical (holomorphic) line bundle on $M$.
Let $\psi$ be a plurisubharmonic function on $M$,
and let  $\varphi$ be a Lebesgue measurable function on $M$,
such that $\varphi+\psi$ is a plurisubharmonic function on $M$. Take $T=-\sup_{M}\psi$ ($T$ maybe $-\infty$).

\begin{Definition}
\label{def:gain}
We call a positive measurable function $c$ on $(T,+\infty)$ in class $\mathcal{P}_T$ if the following two statements hold:

$(1)$ $c(t)e^{-t}$ is decreasing with respect to $t$;

$(2)$ there is a closed subset $E$ of $M$ such that $E\subset \{z\in Z:\psi(z)=-\infty\}$ and for any compact subset $K\subseteq M\backslash E$, $e^{-\varphi}c(-\psi)$ has a positive lower bound on $K$, where $Z$ is some analytic subset of $M$.	
\end{Definition}

Let $Z_{0}$ be a subset of $\{\psi=-\infty\}$ such that $Z_{0}\cap Supp(\{\mathcal{O}/\mathcal{I(\varphi+\psi)}\})\neq\emptyset$.
Let $U\supseteq Z_{0}$ be an open subset of $M$,
and let $f$ be a holomorphic $(n,0)$ form on $U$.
Let $\mathcal{F}\supseteq\mathcal{I}(\varphi+\psi)|_{U}$ be a  analytic subsheaf of $\mathcal{O}$ on $U$.

Denote
\begin{equation*}
\begin{split}
\inf\{\int_{\{\psi<-t\}}|\tilde{f}|^{2}e^{-\varphi}c(-\psi):(\tilde{f}-f)\in H^{0}(Z_0,&
(\mathcal{O}(K_{M})\otimes\mathcal{F})|_{Z_0})\\&\&{\,}\tilde{f}\in H^{0}(\{\psi<-t\},\mathcal{O}(K_{M}))\},
\end{split}
\end{equation*}
by $G(t;c)$ ($G(t)$ for short),  where $t\in[T,+\infty)$, $c$ is a nonnegative function on $(T,+\infty)$,
$|f|^{2}:=\sqrt{-1}^{n^{2}}f\wedge\bar{f}$ for any $(n,0)$ form $f$ and $(\tilde{f}-f)\in H^{0}(Z_0,
(\mathcal{O}(K_{M})\otimes\mathcal{F})|_{Z_0})$ means $(\tilde{f}-f,z_0)\in(\mathcal{O}(K_{M})\otimes\mathcal{F})_{z_0}$ for all $z_0\in Z_0$.

The following Theorem shows the concavity for $G(t)$.
\begin{Theorem}[see \cite{GY-concavity}, see also \cite{GMY-concavity2}]
\label{thm:general_concave}
Let $c\in\mathcal{P}_T$ satisfying $\int_T^{+\infty}c(s)e^{-s}ds<+\infty$. If there exists $t\in[T,+\infty)$ satisfying that $G(t)<+\infty$, then $G(h^{-1}(r))$ is concave with respect to $r\in(0,\int_{T}^{+\infty}c(s)e^{-s}ds)$, $\lim_{t\rightarrow T+0}G(t)=G(T)$ and $\lim_{t\rightarrow +\infty}G(t)=0$, where $h(t)=\int_{t}^{+\infty}c(s)e^{-s}ds$.
\end{Theorem}

Denote that
\begin{displaymath}
	\begin{split}
		\mathcal{H}^2(c,t):=\{\tilde{f}:\int_{\{\psi<-t\}}|\tilde{f}|^2e^{-\varphi}c(-\psi)<+\infty,(\tilde{f}-f)&\in H^0(Z_0,(\mathcal{O}(K_{M})\otimes\mathcal{F})|_{Z_0})\\
		&\&\tilde{f}\in H^0(\{\psi<-t\},\mathcal{O}(K_M))\},
	\end{split}
\end{displaymath}
where $t\in[T,+\infty)$ and $c$ is a nonnegative measurable function on $(T,+\infty)$.

\begin{Corollary}[see \cite{GY-concavity}, see also \cite{GMY-concavity2}]	\label{c:linear}
Let $c\in\mathcal{P}_T$ satisfying $\int_T^{+\infty}c(s)e^{-s}ds<+\infty$. If $G(t)\in(0,+\infty)$ for some $t\geq T$ and $G({h}^{-1}(r))$ is linear with respect to $r\in[0,\int_{T}^{+\infty}c(s)e^{-s}ds)$,
 then there is a unique holomorphic $(n,0)$ form $F$ on $M$ satisfying $(F-f)\in H^{0}(Z_0,(\mathcal{O}(K_{M})\otimes\mathcal F)|_{Z_0})$ and $G(t;c)=\int_{\{\psi<-t\}}|F|^2e^{-\varphi}c(-\psi)$ for any $t\geq T$. Furthermore,
\begin{equation}
	\label{eq:20210412b}
	\int_{\{-t_1\leq\psi<-t_2\}}|F|^2e^{-\varphi}a(-\psi)=\frac{G(T_1;c)}{\int_{T_1}^{+\infty}c(t)e^{-t}dt}\int_{t_2}^{t_1} a(t)e^{-t}dt
\end{equation}
for any nonnegative measurable function $a$ on $(T,+\infty)$, where $+\infty\geq t_1>t_2\geq T$.

Especially, if $\mathcal H^2(\tilde{c},t_0)\subset\mathcal H^2(c,t_0)$ for some $t_0\geq T$, where $\tilde{c}$ is a nonnegative measurable function on $(T,+\infty)$, we have
\begin{equation}
	\label{eq:20210412a}
	G(t_0;\tilde{c})=\int_{\{\psi<-t_0\}}|F|^2e^{-\varphi}\tilde{c}(-\psi)=\frac{G(T_1;c)}{\int_{T_1}^{+\infty}c(s)e^{-s}ds}\int_{t_0}^{+\infty} \tilde{c}(s)e^{-s}ds.\end{equation}	
\end{Corollary}

Let $M=\Omega$  be an open Riemann surface, which admits a nontrivial Green function $G_{\Omega}$.
 Let $p:\Delta\rightarrow\Omega$ be the universal covering from unit disc $\Delta$ to $\Omega$.
It is known that for any harmonic function $u$ on $\Omega$,
there exists a $\chi_{u}$ (the  character associate to $u$) and a multiplicative function $f_u\in\mathcal{O}^{\chi_{u}}(\Omega)$,
such that $|f_u|=p^{*}e^{u}$.
Let $z_0\in \Omega$.
Recall that for the Green function $G_{\Omega}(z,z_0)$,
there exist a $\chi_{z_0}$ and a multiplicative function $f_{z_0}\in\mathcal{O}^{\chi_{z_0}}(\Omega)$, such that $|f_{z_0}(z)|=p^{*}e^{G_{\Omega}(z,z_0)}$ .

Let $Z_0:=\{z_1,z_2,...,z_m\}\subset\Omega$ be a subset of $\Omega$ satisfying that $z_j\not=z_k$ for any $j\not=k$.
Let $w_j$ be a local coordinate on a neighborhood $V_{z_j}\Subset\Omega$ of $z_j$ satisfying $w_j(z_j)=0$ for $j\in\{1,2,...,m\}$, where $V_{z_j}\cap V_{z_k}=\emptyset$ for any $j\not=k$. Denote that $V_0:=\cup_{1\le j\le n}V_{z_j}$.

Let $f$ be a holomorphic $(1,0)$ form on $V_0$, and let $f=f_1dw_j$ on $V_{z_j}$, where $f_1$ is a holomorphic function on $V_0$. Let $\psi$ be a negative subharmonic function on $\Omega$, and let $\varphi$ be a Lebesgue measurable function on $\Omega$ such that $\varphi+\psi$ is subharmonic on $\Omega$.

The following Theorem gives a characterization of the concavity of $G(h^{-1}(r))$ degenerating to linearity.
\begin{Theorem}[see \cite{GY-concavity3}]
	\label{thm:m-points}Let $c\in\mathcal{P}_0.$
 Assume that $G(0)\in(0,+\infty)$ and $(\psi-2p_jG_{\Omega}(\cdot,z_j))(z_j)>-\infty$ for  $j\in\{1,2,..,m\}$, where $p_j=\frac{1}{2}v(dd^c(\psi),z_j)>0$. Then $G(h^{-1}(r))$ is linear with respect to $r$ if and only if the following statements hold:
	
	$(1)$ $\psi=2\sum_{1\le j\le m}p_jG_{\Omega}(\cdot,z_j)$;
	
	$(2)$ $\varphi+\psi=2\log|g|+2\sum_{1\le j\le m}G_{\Omega}(\cdot,z_j)+2u$ and $\mathcal{F}_{z_j}=\mathcal{I}(\varphi+\psi)_{z_j}$ for any $j\in\{1,2,...,m\}$, where $g$ is a holomorphic function on $\Omega$ such that $ord_{z_j}(g)=ord_{z_j}(f_1)$ for any $j\in\{1,2,...,m\}$ and $u$ is a harmonic function on $\Omega$;
	
	$(3)$ $\Pi_{1\le j\le m}\chi_{z_j}=\chi_{-u}$, where $\chi_{-u}$ and $\chi_{z_j}$ are the  characters associated to the functions $-u$ and $G_{\Omega}(\cdot,z_j)$ respectively;
	
	$(4)$  $\lim_{z\rightarrow z_k}\frac{f}{gp_*(f_u(\Pi_{1\le j\le m}f_{z_j})(\sum_{1\le j\le m}p_{j}\frac{d{f_{z_{j}}}}{f_{z_{j}}}))}=c_0$ for any $k\in\{1,2...,m\}$, where $c_0\in\mathbb{C}\backslash\{0\}$ is a constant independent of $k$.
\end{Theorem}

\begin{Remark}[see \cite{GY-concavity3}]\label{r:chi}
	 For any $\{z_1,z_2,..,z_m\}$, there exists a harmonic function $u$ on $\Omega$ such that $\Pi_{1\le j\le m}\chi_{z_j}=\chi_{-u}$. In fact, as $\Omega$ is an open Riemann surface, then there exists a holomorphic function $\tilde{f}$ on $\Omega$ satisfying that $u:=\log|\tilde{f}|-\sum_{1\le j\le m}G_{\Omega}(\cdot,z_j)$ is harmonic on $\Omega$, which implies that $\Pi_{1\le j\le m}\chi_{z_j}=\chi_{-u}$.
\end{Remark}

We recall a  characterization of the holding of equality in optimal jets $L^2$ extension problem from finite points to open Riemann surfaces.
\begin{Theorem}[see \cite{GY-concavity3}]\label{c:L2-1d-char}
Let $k_j$ be a nonnegative integer for any $j\in\{1,2,...,m\}$. Let $\psi$ be a negative  subharmonic function on $\Omega$ satisfying that   $\frac{1}{2}v(dd^{c}\psi,z_j)=p_j>0$ for any $j\in\{1,2,...,m\}$. Let $\varphi$ be a Lebesgue measurable function on $\Omega$  such that $\varphi+\psi$ is subharmonic on $\Omega$, $\frac{1}{2}v(dd^c(\varphi+\psi),z_j)=k_j+1$ and $\alpha_j:=(\varphi+\psi-2(k_j+1)G_{\Omega}(\cdot,z_j))(z_j)>-\infty$ for any $j$. Let $c(t)$ be a positive measurable function on $(0,+\infty)$ satisfying $c(t)e^{-t}$ is decreasing on $(0,+\infty)$ and $\int_{0}^{+\infty}c(s)e^{-s}ds<+\infty$. Let $a_j$ be a constant for any $j$.

Let $f$ be a holomorphic $(1,0)$ form on $V_0$ satisfying that $f=a_jw_j^{k_j}dw_j$ on $V_{z_j}$. Then there exists a holomorphic $(1,0)$ form $F$ on $\Omega$ such that $(F-f,z_j)\in(\mathcal{O}(K_{\Omega})\otimes\mathcal{I}(2(k_j+1)G_{\Omega}(\cdot,z_j)))_{z_j}$ and
 \begin{equation}
 	\label{eq:210902a}
 	\int_{\Omega}|F|^2e^{-\varphi}c(-\psi)\leq(\int_0^{+\infty}c(s)e^{-s}ds)\sum_{1\le j\le m}\frac{2\pi|a_j|^2e^{-\alpha_j}}{p_jc_{\beta}(z_j)^{2(k_j+1)}}.
 \end{equation}

 Moreover, equality $(\int_0^{+\infty}c(s)e^{-s}ds)\sum_{1\le j\le m}\frac{2\pi|a_j|^2e^{-\alpha_j}}{p_jc_{\beta}(z_j)^{2(k_j+1)}}=\inf\{\int_{\Omega}|\tilde{F}|^2e^{-\varphi}c(-\psi):\tilde{F}$ is a holomorphic $(1,0)$ form on $\Omega$ such that $(\tilde{F}-f,z_j)\in(\mathcal{O}(K_{\Omega})\otimes\mathcal{I}(2(k_j+1)G_{\Omega}(\cdot,z_j)))_{z_j}$ for any $j\}$ holds if and only if the following statements hold:

	$(1)$ $\psi=2\sum_{1\le j\le m}p_jG_{\Omega}(\cdot,z_j)$;
	
	$(2)$ $\varphi+\psi=2\log|g|+2\sum_{1\le j\le m}(k_j+1)G_{\Omega}(\cdot,z_j)+2u$, where $g$ is a holomorphic function on $\Omega$ such that $g(z_j)\not=0$ for any $j\in\{1,2,...,m\}$ and $u$ is a harmonic function on $\Omega$;
	
	$(3)$ $\Pi_{1\le j\le m}\chi_{z_j}^{k_j+1}=\chi_{-u}$, where $\chi_{-u}$ and $\chi_{z_j}$ are the  characters associated to the functions $-u$ and $G_{\Omega}(\cdot,z_j)$ respectively;
	
	$(4)$  $\lim_{z\rightarrow z_k}\frac{f}{gp_*(f_u(\Pi_{1\le j\le m}f_{z_j}^{k_j+1})(\sum_{1\le j\le m}p_{j}\frac{d{f_{z_{j}}}}{f_{z_{j}}}))}=c_0$ for any $k\in\{1,2...,m\}$, where $c_0\in\mathbb{C}\backslash\{0\}$ is a constant independent of $k$.
\end{Theorem}

\begin{Remark}[see \cite{GY-concavity3}]\label{rem:1.2}
When the four statements in Theorem \ref{c:L2-1d-char} hold,
$$c_0gp_*(f_u(\Pi_{1\le j\le m}f_{z_j}^{k_j+1})(\sum_{1\le j\le m}p_{j}\frac{d{f_{z_{j}}}}{f_{z_{j}}}))$$
 is the unique holomorphic $(1,0)$ form $F$ on $\Omega$ such that $(F-f,z_j)\in(\mathcal{O}(K_{\Omega})\otimes\mathcal{I}(2(k_j+1)G_{\Omega}(\cdot,z_j)))_{z_j}$ and
 \begin{equation*}
 	\int_{\Omega}|F|^2e^{-\varphi}c(-\psi)\leq(\int_0^{+\infty}c(s)e^{-s}ds)\sum_{1\le j\le m}\frac{2\pi|a_j|^2e^{-\alpha_j}}{p_jc_{\beta}(z_j)^{2(k_j+1)}}.
 \end{equation*}
\end{Remark}

Let $Z_0:=\{z_j:j\in\mathbb{Z}_{\ge1}\}\subset\Omega$ be a  discrete set of infinite points.
Let $w_j$ be a local coordinate on a neighborhood $V_{z_j}\Subset\Omega$ of $z_j$ satisfying $w_j(z_j)=0$ for $j\in\mathbb{Z}_{\ge1}$, where $V_{z_j}\cap V_{z_k}=\emptyset$ for any $j\not=k$. Denote that $V_0:=\cup_{j\in\mathbb{Z}_{\ge1}}V_{z_j}$.
Let $f$ be a holomorphic $(1,0)$ form on $V_0$, and let $f=f_1dw_j$ on $V_{z_j}$, where $f_1$ is a holomorphic function on $V_0$. Let $\psi$ be a negative subharmonic function on $\Omega$, and let $\varphi$ be a Lebesgue measurable function on $\Omega$ such that $\varphi+\psi$ is subharmonic on $\Omega$.

 The following result gives a necessary condition for $G(h^{-1}(r))$ is linear.

\begin{Proposition}[see \cite{GY-concavity3}]Let $c\in\mathcal{P}_0.$	\label{p:infinite}
 Assume that $G(0)\in(0,+\infty)$ and $(\psi-2p_jG_{\Omega}(\cdot,z_j))(z_j)>-\infty$ for  $j\in\mathbb{Z}_{\ge1}$, where $p_j=\frac{1}{2}v(dd^c(\psi),z_j)>0$. Assume that $G(h^{-1}(r))$ is linear with respect to $r$. Then the following statements hold:
	
	$(1)$ $\psi=2\sum_{j\in\mathbb{Z}_{\ge1}}p_jG_{\Omega}(\cdot,z_j)$;

	$(2)$ $\varphi+\psi=2\log|g|$ and $\mathcal{F}_{z_j}=\mathcal{I}(\varphi+\psi)_{z_j}$ for any $j\in\mathbb{Z}_{\ge1}$, where $g$ is a holomorphic function on $\Omega$ such that $ord_{z_j}(g)=ord_{z_j}(f_1)+1$ for any $j\in\mathbb{Z}_{\ge1}$;

	$(3)$  $\frac{p_j}{ord_{z_j}g}\lim_{z\rightarrow z_j}\frac{dg}{f}=c_0$ for any $j\in\mathbb{Z}_{\ge1}$, where $c_0\in\mathbb{C}\backslash\{0\}$ is a constant independent of $j$;
	
		$(4)$ $\sum_{j\in\mathbb{Z}_{\ge1}}p_j<+\infty$.
\end{Proposition}

\subsection{Some basic properties of the Green functions}
\

In this Section, we recall some basic properties of the Green functions. Let $\Omega$ be an open Riemann surface, which admits a nontrivial Green function $G_{\Omega}$, and let $z_0\in\Omega$.

\begin{Lemma}[see \cite{S-O69}, see also \cite{Tsuji}] 	\label{l:green-sup}Let $w$ be a local coordinate on a neighborhood of $z_0$ satisfying $w(z_0)=0$.  $G_{\Omega}(z,z_0)=\sup_{v\in\Delta_{\Omega}^*(z_0)}v(z)$, where $\Delta_{\Omega}^*(z_0)$ is the set of negative subharmonic function on $\Omega$ such that $v-\log|w|$ has a locally finite upper bound near $z_0$. Moreover, $G_{\Omega}(\cdot,z_0)$ is harmonic on $\Omega\backslash\{z_0\}$ and $G_{\Omega}(\cdot,z_0)-\log|w|$ is harmonic near $z_0$.
\end{Lemma}

\begin{Lemma}[see \cite{GY-concavity3}]
	\label{l:green-sup2}Let $K=\{z_j:j\in\mathbb{Z}_{\ge1}\,\&\,j<\gamma \}$ be a discrete subset of $\Omega$, where $\gamma\in\mathbb{Z}_{>1}\cup\{+\infty\}$. Let $\psi$ be a negative subharmonic function on $\Omega$ such that $\frac{1}{2}v(dd^c\psi,z_j)\ge p_j$ for any $j$, where $p_j>0$ is a constant. Then $2\sum_{1\le j< \gamma}p_jG_{\Omega}(\cdot,z_j)$ is a subharmonic function on $\Omega$ satisfying that $2\sum_{1\le j<\gamma }p_jG_{\Omega}(\cdot,z_j)\ge\psi$ and $2\sum_{1\le j<\gamma }p_jG_{\Omega}(\cdot,z_j)$ is harmonic on $\Omega\backslash K$.
\end{Lemma}

\begin{Lemma}[see \cite{GY-concavity}]\label{l:G-compact}
For any  open neighborhood $U$ of $z_0$, there exists $t>0$ such that $\{G_{\Omega}(z,z_0)<-t\}$ is a relatively compact subset of $U$.
\end{Lemma}

\begin{Lemma}[see \cite{GY-concavity3}]
	\label{l:green-approx} There exists a sequence of open Riemann surfaces $\{\Omega_l\}_{l\in\mathbb{Z}^+}$ such that $z_0\in\Omega_l\Subset\Omega_{l+1}\Subset\Omega$, $\cup_{l\in\mathbb{Z}^+}\Omega_l=\Omega$, $\Omega_l$ has a smooth boundary $\partial\Omega_l$ in $\Omega$  and $e^{G_{\Omega_l}(\cdot,z_0)}$ can be smoothly extended to a neighborhood of $\overline{\Omega_l}$ for any $l\in\mathbb{Z}^+$, where $G_{\Omega_l}$ is the Green function of $\Omega_l$. Moreover, $\{{G_{\Omega_l}}(\cdot,z_0)-G_{\Omega}(\cdot,z_0)\}$ is decreasingly convergent to $0$ on $\Omega$.
\end{Lemma}

\subsection{Some results related to $\max_{1\le j\le n}\{2p_j\log|w_j|\}$}
\

Let $f=\sum_{\alpha\in\mathbb{Z}_{\ge0}^n}b_{\alpha}w^{\alpha}$ (Taylor expansion) be a holomorphic function  on $D=\{w\in\mathbb{C}^n:|w_j|<r_0$ for any $j\in\{1,2,...,n\}\}$, where $r_0>0$. 	Let $\psi=\max_{1\le j\le n}\{2p_j\log|w_j|\}$ be a plurisubharmonic function on $\mathbb{C}^n$, where $p_j>0$ is a constant for any $j\in\{1,2,...,n\}$. We recall a characterization of  $\mathcal{I}(\psi)_o$, where $o$ is the origin in $\mathbb{C}^n$.
\begin{Lemma}[see \cite{guan-20}]\label{l:0}
$(f,o)\in\mathcal{I}(\psi)_{o}$ if and only if $\sum_{1\le j\le n}\frac{\alpha_j+1}{p_j}>1$ for any $\alpha\in\mathbb{Z}_{\ge0}^n$ satisfying $b_{\alpha}\not=0$.
\end{Lemma}
\begin{proof}
	For the convenience of the reader, we recall the proof.
	
There exists $r_1>0$ such that $\{\psi<\log r_1\}\Subset D$.	If $(f,o)\in\mathcal{I}(\psi)_{o}$,  we have
	\begin{equation}
		\label{eq:1125a}\int_{\{\psi<\log r_1\}}|f|^2e^{-\psi}d\lambda_n<+\infty,
	\end{equation}
	where $d\lambda_n$ is the Lebesgue measure on $\mathbb{C}^n$. Note that
	\begin{displaymath}
		\begin{split}
			\int_{\{\psi<\log r_1\}}|f|^2e^{-\psi}d\lambda_n=&\lim_{\epsilon\rightarrow0+0}\int_{\{\epsilon<|w_1|<r_1^{\frac{1}{2p_1}}\}\cap...\cap\{\epsilon<|w_n|<r_1^{\frac{1}{2p_n}}\}}|f|^2e^{-\psi}d\lambda_n\\
			=&\lim_{\epsilon\rightarrow0+0}(\sum_{\alpha\in\mathbb{Z}_{\ge0}^n}\int_{\{\epsilon<|w_1|<r_1^{\frac{1}{2p_1}}\}\cap...\cap\{\epsilon<|w_n|<r_1^{\frac{1}{2p_n}}\}}|b_{\alpha}w^{\alpha}|^2e^{-\psi}d\lambda_n)\\
			=&\sum_{\alpha\in\mathbb{Z}_{\ge0}^n}|b_{\alpha}|^2\int_{\{\psi<\log r_1\}}|w^{\alpha}|^2e^{-\psi}d\lambda_n.
	   \end{split}
	\end{displaymath}
 Inequality \eqref{eq:1125a} implies that
	\begin{equation}
		\label{eq:1125b}\int_{\{\psi<\log r_1\}}|w^{\alpha}|^2e^{-\psi}d\lambda_n<+\infty
	\end{equation}
	for any $\alpha\in\mathbb{Z}_{\ge0}^n$ satisfying $b_{\alpha}\not=0$. Note that
	\begin{equation}	\label{eq:1125c}\begin{split}
	\int_{\{\psi<\log r_1\}}|w^{\alpha}|^2e^{-\psi}d\lambda_n=&\int_{\{\psi<\log r_1\}}|w^{\alpha}|^2(\int_{0}^{+\infty}\mathbb{I}_{\{l<e^{-\psi}\}}dl)d\lambda_n\\
		=&\int_{0}^{r_1}(\int_{\{\psi<\log r\}}|w^{\alpha}|^2d\lambda_n)r^{-2}dr\\
		&+\frac{1}{r_1}\int_{\{\psi<\log r_1\}}|w^{\alpha}|^2d\lambda_n
		\end{split}
	\end{equation}
	and
	\begin{equation}
		\label{eq:1125d}\begin{split}
			\int_{\{\psi<\log r\}}|w^{\alpha}|^2d\lambda_n=&\int_{\{|w_1|<r^{\frac{1}{2p_1}}\}\cap...\cap\{|w_n|<r^{\frac{1}{2p_n}}\}}|\Pi_{1\le j\le n}w_j^{\alpha_j}|^2d\lambda_n\\
			=&\pi^n\frac{r^{\sum_{1\le j\le n}\frac{\alpha_j+1}{p_j}}}{\Pi_{1\le j\le n}(\alpha_j+1)}.
		\end{split}
	\end{equation}
	It follows from inequality \eqref{eq:1125b}, equality \eqref{eq:1125c} and equality \eqref{eq:1125d} that $\sum_{1\le j\le n}\frac{\alpha_j+1}{p_j}>1$ for any $\alpha\in\mathbb{Z}_{\ge0}^n$ satisfying $b_{\alpha}\not=0$.
	
	If $\sum_{1\le j\le n}\frac{\alpha_j+1}{p_j}>1$ for any $\alpha\in\mathbb{Z}_{\ge0}^n$ satisfying $b_{\alpha}\not=0$, it follows from equality \eqref{eq:1125c} and equality \eqref{eq:1125d} that
	\begin{equation}
		\label{eq:1125f}\int_{\{\psi<\log r_1\}}|w^{\alpha}|^2e^{-\psi}d\lambda_n<+\infty
	\end{equation}
	holds for any $\alpha\in\mathbb{Z}_{\ge0}^n$ satisfying $b_{\alpha}\not=0$. Note that there exists $N>0$ and $r_2\in(0,r_1)$ such that
	$$|f-\sum_{0\le\alpha_j\le N}b_{\alpha}w^{\alpha}|^2\le e^{\psi}$$	
on $\{\psi<\log r_2\}$.	By inequality \eqref{eq:1125f}, we have
\begin{displaymath}
	\begin{split}
		\int_{\{\psi<\log r_2\}}|f|^2e^{-\psi}d\lambda_n\le& 2\int_{\{\psi<\log r_2\}}|f-\sum_{0\le\alpha_j\le N}b_{\alpha}w^{\alpha}|^2e^{-\psi}d\lambda_n\\
		&+\int_{\{\psi<\log r_2\}}|\sum_{0\le\alpha_j\le N}b_{\alpha}w^{\alpha}|^2e^{-\psi}d\lambda_n\\
		<&+\infty,
	\end{split}
\end{displaymath}
i.e. $(f,o)\in\mathcal{I}(\psi)_{o}$.	
	\end{proof}

In the following two lemmas, we discuss integrals $\int_{\{-t-1<\psi<-t\}}|f|^2e^{-\psi}d\lambda_n$ and $\int_{\{\psi<-t\}}|f|^2d\lambda_n$.
\begin{Lemma}\label{l:m1}
Let $\psi=\max_{1\le j\le n}\{2p_j\log|w_j|\}$ be a plurisubharmonic function on $\mathbb{C}^n$, where $p_j>0$.
	Let $f=\sum_{\alpha\in\mathbb{Z}_{\ge0}^n}b_{\alpha}w^{\alpha}$ (Taylor expansion) be a holomorphic function on $\{\psi<-t_0\}$, where $t_0>0$. Denote that $q_{\alpha}:=\sum_{1\le j\le n}\frac{\alpha_j+1}{p_j}-1$ for any $\alpha\in\mathbb{Z}_{\ge0}^n$ and  $E_1:=\{\alpha\in\mathbb{Z}_{\ge0}^n:q_{\alpha}=0\}$. Then
\begin{displaymath}\begin{split}
	\int_{\{-t-1<\psi<-t\}}|f|^2e^{-\psi}d\lambda_n=&\sum_{\alpha\in E_1}\frac{|b_{\alpha}|^2\pi^n}{\Pi_{1\le j\le n}(\alpha_j+1)}\\
	&+\sum_{\alpha\not\in E_1}\frac{|b_{\alpha}|^2\pi^n(q_{\alpha}+1)(e^{-q_{\alpha}t}-e^{-q_{\alpha}(t+1)})}{q_{\alpha}\Pi_{1\le j\le n}(\alpha_j+1)}	
\end{split}\end{displaymath}
	 for any $t>t_0$.
\end{Lemma}
\begin{proof}By direct calculations, we obtain that
	\begin{equation}\label{eq:211125c}\begin{split}
		&\int_{\{-t-1<\psi<-t\}}|w^{\alpha}|^2e^{-\psi}d\lambda_n\\
		=&(2\pi)^n\int_{\{e^{-\frac{t+1}{2}}<\max_{1\le j\le n}\{s_j^{p_j}\}<e^{-\frac{t}{2}}\,\&\,s_j>0\}}\frac{\Pi_{1\le j\le n}s_j^{2\alpha_j+1}}{\max_{1\le j\le n}\{s_j^{2p_j}\}}ds_1ds_2...ds_n\\
		=&(2\pi)^n\frac{1}{\Pi_{1\le j\le n}p_j}\int_{\{e^{-\frac{t+1}{2}}<\max_{1\le j\le n}\{r_j\}<e^{-\frac{t}{2}}\,\&\,r_j>0\}}\frac{\Pi_{1\le j\le n}r_j^{\frac{2\alpha_j+2}{p_j}-1}}{\max_{1\le j\le n}\{r_j^2\}}dr_1dr_2...dr_n.
		\end{split}
	\end{equation}
 By the Fubini's theorem, we have
\begin{equation}
	\label{eq:211125d}\begin{split}
		&\int_{\{e^{-\frac{t+1}{2}}<\max_{1\le j\le n}\{r_j\}<e^{-\frac{t}{2}}\,\&\,r_j>0\}}\frac{\Pi_{1\le j\le n}r_j^{\frac{2\alpha_j+2}{p_j}-1}}{\max_{1\le j\le n}\{r_j^2\}}dr_1dr_2...dr_n\\
		=&\sum_{1\le j'\le n}\int_{e^{-\frac{t+1}{2}}}^{e^{-\frac{t}{2}}}(\int_{\{0\le r_j<r_{j'},j\not=j'\}}\Pi_{j\not=j'}r_j^{\frac{2\alpha_j+2}{p_j}-1}\wedge_{j\not=j'}dr_j)r_{j'}^{\frac{2\alpha_{j'}+2}{p_{j'}}-3}dr_{j'}\\
		=&\sum_{1\le j'\le n}(\Pi_{j\not=j'}\frac{p_j}{2\alpha_j+2})\int_{e^{-\frac{t+1}{2}}}^{e^{-\frac{t}{2}}}r_{j'}^{2\sum_{1\le j\le n}\frac{\alpha_j+1}{p_j}-3}dr_{j'}\\
		=&2^{-n}(2q_{\alpha}+2)\Pi_{1\le j\le n}\frac{p_j}{\alpha_j+1}\int_{e^{-\frac{t+1}{2}}}^{e^{-\frac{t}{2}}}r^{2q_{\alpha}-1}dr
			\end{split}
\end{equation}	
It is clear that $\int_{\{-t-1<\psi<-t\}}|f|^2e^{-\psi}d\lambda_n=\sum_{\alpha\in\mathbb{Z}_{\ge0}^n}|b_{\alpha}|^2\int_{\{-t-1<\psi<-t\}}|w^{\alpha}|^2e^{-\psi}d\lambda_n$, then equality \eqref{eq:211125c} and equality \eqref{eq:211125d} implies that
\begin{displaymath}\begin{split}
	\int_{\{-t-1<\psi<-t\}}|f|^2e^{-\psi}d\lambda_n=&\sum_{\alpha\in E_1}\frac{|b_{\alpha}|^2\pi^n}{\Pi_{1\le j\le n}(\alpha_j+1)}\\
	&+\sum_{\alpha\not\in E_1}\frac{|b_{\alpha}|^2\pi^n(q_{\alpha}+1)(e^{-q_{\alpha}t}-e^{-q_{\alpha}(t+1)})}{q_{\alpha}\Pi_{1\le j\le n}(\alpha_j+1)}	
\end{split}\end{displaymath}
\end{proof}

\begin{Lemma}\label{l:m2}
Let $\psi=\max_{1\le j\le n}\{2p_j\log|w_j|\}$ be a plurisubharmonic function on $\mathbb{C}^n$, where $p_j>0$.
	Let $f=\sum_{\alpha\in \mathbb{Z}_{\ge0}^n}b_{\alpha}w^{\alpha}$ (Taylor expansion) be a holomorphic function on $\{\psi<-t_0\}$, where $t_0>0$. Then
	$$\int_{\{\psi<-t\}}|f|^2d\lambda_n=\sum_{\alpha\in\mathbb{Z}_{\ge0}^n}e^{-\sum_{1\le j\le n}\frac{\alpha_j+1}{p_j}t}\frac{|b_{\alpha}|^2\pi^n}{\Pi_{1\le j\le n}(\alpha_j+1)}$$
	holds for any $t\ge t_0$. Moreover, if $f=\sum_{\alpha\in E_1}b_{\alpha}w^{\alpha}+g_0$, where $E_1=\{(\alpha_1,\alpha_2,...,\alpha_n):\sum_{1\le j\le n}\frac{\alpha_j+1}{p_{j}}=1\,\&\,\alpha_j\in\mathbb{Z}_{\ge0}\}$, $(g_0,o)\in\mathcal{I}(\psi)_o$ and $o$ is the origin in $\mathbb{C}^n$, we have
	$$\lim_{t\rightarrow+\infty}e^{t}\int_{\{\psi<-t\}}|f|^2d\lambda_n=\sum_{\alpha\in E_1}\frac{|b_{\alpha}|^2\pi^n}{\Pi_{1\le j\le n}(\alpha_j+1)}.$$
\end{Lemma}
\begin{proof}
	By direct calculations, we obtain that
	\begin{equation}\label{eq:1127b}\begin{split}
		&\int_{\{\psi<-t\}}|w^{\alpha}|^2d\lambda_n\\
		=&(2\pi)^n\int_{\{\max_{1\le j\le n}\{s_j^{p_j}\}<e^{-\frac{t}{2}}\,\&\,s_j>0\}}\Pi_{1\le j\le n}s_j^{2\alpha_j+1}ds_1ds_2...ds_n\\
		=&(2\pi)^n\frac{1}{\Pi_{1\le j\le n}p_j}\\
		&\times\int_{\{\max_{1\le j\le n}\{r_j\}<e^{-\frac{t}{2}}\,\&\,r_j>0\}}\Pi_{1\le j\le n}r_j^{\frac{2\alpha_j+2}{p_j}-1}dr_1dr_2...dr_n.
		\end{split}
	\end{equation}
 By the Fubini's theorem, we have
\begin{equation}
	\label{eq:211125f}\begin{split}
		&\int_{\{\max_{1\le j\le n}\{r_j\}<e^{-\frac{t}{2}}\,\&\,r_j>0\}}\Pi_{1\le j\le n}r_j^{\frac{2\alpha_j+2}{p_j}-1}dr_1dr_2...dr_n\\
		=&\sum_{j'=1}^n\int_{0}^{e^{-\frac{t}{2}}}(\int_{\{0\le r_j<r_{j'},j\not=j'\}}\Pi_{j\not=j'}r_j^{\frac{2\alpha_j+2}{p_j}-1}\wedge_{j\not=j'}dr_j)r_{j'}^{\frac{2\alpha_{j'}+2}{p_{j'}}-1}dr_{j'}\\
		=&\sum_{j'=1}^n(\Pi_{j\not=j'}\frac{p_j}{2\alpha_j+2})\int_{0}^{e^{-\frac{t}{2}}}r_{j'}^{\sum_{1\le k\le n}\frac{2\alpha_k+2}{p_k}-1}dr_{j'}\\
		=&e^{-\sum_{1\le j\le n}\frac{\alpha_j+1}{p_j}t}\Pi_{1\le j\le n}\frac{p_j}{2\alpha_j+2}.
			\end{split}
\end{equation}		
Following from $\int_{\{\psi<-t\}}|f|^2d\lambda_n=\sum_{\alpha\in\mathbb{Z}_{\ge0}^n}|b_{\alpha}|^2\int_{\{\psi<-t\}}|w^{\alpha}|^2d\lambda_n$, equality \eqref{eq:1127b} and equality \eqref{eq:211125f} that
\begin{equation}
	\label{eq:1127f}\int_{\{\psi<-t\}}|f|^2d\lambda_n=\sum_{\alpha\in\mathbb{Z}_{\ge0}^n}e^{-\sum_{1\le j\le n}\frac{\alpha_j+1}{p_j}t}\frac{|b_{\alpha}|^2\pi^n}{\Pi_{1\le j\le n}(\alpha_j+1)}.
\end{equation}

Now, we consider the case $f=\sum_{\alpha\in E_1}b_{\alpha}w^{\alpha}+g_0$. It follows from Lemma \ref{l:0} that $g_0=\sum_{\alpha\in E_2}\tilde{b}_{\alpha}w^{\alpha}$ (Taylor expansion), where $E_2=\{(\alpha_1,\alpha_2,...,\alpha_n):\sum_{1\le j\le n}\frac{\alpha_j+1}{p_{j}}>1\,\&\,\alpha_j\in\mathbb{Z}_{\ge0}\}$.
 It follows from equality \eqref{eq:1127f} that
\begin{equation}\begin{split}
	\label{eq:1127g}\lim_{t\rightarrow+\infty}e^{t}\int_{\{\psi<-t\}}|g_0|^2d\lambda_n&=\lim_{t\rightarrow+\infty}\sum_{\alpha\in E_2}e^{(1-\sum_{1\le j\le n}\frac{\alpha_j+1}{p_j})t}\frac{|\tilde{b}_{\alpha}|^2\pi^n}{\Pi_{1\le j\le n}(\alpha_j+1)}\\
	&=0.
	\end{split}
\end{equation}
As $\int_{\{\psi<-t\}}|f|^2d\lambda_n=\int_{\{\psi<-t\}}|\sum_{\alpha\in E_1}b_{\alpha}w^{\alpha}|^2d\lambda_n+\int_{\{\psi<-t\}}|g_0|^2d\lambda_n$, it follows from equality \eqref{eq:1127g} and equality \eqref{eq:1127f}, we obtain that
\begin{displaymath}
	\begin{split}
		\lim_{t\rightarrow+\infty}e^{t}\int_{\{\psi<-t\}}|f|^2d\lambda_n=\sum_{\alpha\in E_1}\frac{|b_{\alpha}|^2\pi^n}{\Pi_{1\le j\le n}(\alpha_j+1)}.
	\end{split}
\end{displaymath}
\end{proof}

The following lemma  will be used in the proof of Lemma \ref{l:orth2}.
\begin{Lemma}\label{l:1}Let $\psi=\max_{1\le j\le n}\{2p_j\log|w_j|\}$ be a plurisubharmonic function on $\mathbb{C}^n$.
Let $\gamma$ be a nonnegative integer such that $\frac{\gamma+1}{p_n}<1$ and let $\tilde{\psi}=\max_{1\le j\le n-1}\{2p_j(1-\frac{\gamma+1}{p_n})\log|w_j|\}$ be a plurisubharmonic function on $\mathbb{C}^{n-1}$. If $(f,o)\in\mathcal{I}(\psi)_o$, then $((\frac{\partial}{\partial w_n})^{\sigma}f(\cdot,0),o')\in\mathcal{I}(\tilde{\psi})_{o'}$ for any integer $\sigma$ satisfying $0\le \sigma\le \gamma$, where  $o'$ is the origin in $\mathbb{C}^{n-1}$.
	\end{Lemma}
\begin{proof}
	Let $f=\sum_{\alpha\in\mathbb{Z}_{\ge0}^n}b_{\alpha}w^{\alpha}$ (Taylor expansion)   on $D=\{w\in\mathbb{C}^n:|w_j|<r_0$ for any $j\in\{1,2,...,n\}\}$, where $r_0>0$. It follows from Lemma \ref{l:0} and $(f,o)\in\mathcal{I}(\psi)_o$ that
	\begin{equation}
		\label{eq:1125g}\sum_{1\le j\le n}\frac{\alpha_j+1}{p_j}>1
	\end{equation}
 for any $\alpha\in\mathbb{Z}_{\ge0}^n$ satisfying $b_{\alpha}\not=0$. For any integer $\sigma\le\gamma$, inequality \eqref{eq:1125g} shows that
	\begin{equation}\label{eq:1125h}
		\begin{split}
			\sum_{1\le j\le n-1}\frac{\alpha_j+1}{p_j(1-\frac{\gamma+1}{p_n})}\ge 	\sum_{1\le j\le n-1}\frac{\alpha_j+1}{p_j(1-\frac{\sigma+1}{p_n})}>1,
		\end{split}
	\end{equation}
where $\alpha\in\mathbb{Z}_{\ge0}^n$ satisfies $\alpha_n=\sigma$ and $b_{\alpha}\not=0$.
	As $(\frac{\partial}{\partial w_n})^{\sigma}f(w_1,...,w_{n-1},0)=\sum_{\alpha_{n}=\sigma}\sigma!b_{\alpha}\Pi_{1\le j\le n-1}w_j^{\alpha_j}$ on $\tilde{D}:=\{(w_1,...,w_{n-1})\in\mathbb{C}^{n-1}:|w_j|<r_0$ for any $j\in\{1,2,...,n-1\}\}$, it follows from inequality \eqref{eq:1125h} and Lemma \ref{l:0} that
	$$((\frac{\partial}{\partial w_n})^{\sigma}f(\cdot,0),o')\in\mathcal{I}(\tilde{\psi})_{o'}$$
	 for any integer $\sigma$ satisfying $0\le \sigma\le \gamma$.
	\end{proof}

\subsection{Some other required results}
\

We recall an $L^2$ extension Theorem, which will be used in the proof of Theorem \ref{thm:ohsawa}.
\begin{Theorem}[see \cite{guan-zhou13ap}]
	\label{gz:L2}
	Let $M$ be an $n-$dimensional complex manifold with a continuous volume form $dV_M$, and let $S$ be a closed complex submanifold of $M$. Let $\psi\in \Delta(S)$. Then for any  holomorphic section $f$ of $K_M|_S$ on $S$, such that
$$\sum_{k=1}^{n}\frac{\pi^k}{k!}\int_{S_{n-k}}\frac{|f|^2}{dV_{M}}dV_{M}[\psi]<+\infty,$$
there exists a holomorphic $(n,0)$ form $F$ on $M$ such that $F|_S=f$ and
$$\int_{M}|F|^2\leq\sum_{k=1}^{n}\frac{\pi^k}{k!}\int_{S_{n-k}}\frac{|f|^2}{dV_{M}}dV_{M}[\psi].$$
\end{Theorem}

The following three lemmas will be used in the proof of Proposition \ref{p:exten-pro-finite}.
\begin{Lemma}[see \cite{GY-concavity}, see also \cite{GMY-concavity2}] \label{lem:L2} Let $c$ be a positive function on $(0,+\infty)$, such that $\int_{0}^{+\infty}c(t)e^{-t}dt<+\infty$ and $c(t)e^{-t}$ is decreasing on $(0,+\infty)$.
Let $B\in(0,+\infty)$ and $t_{0}\geq 0$ be arbitrarily given.
Let $M$ be an $n-$dimensional Stein manifold.
Let $\psi<0$ be a plurisubharmonic function
on $M$.
Let $\varphi$ be a plurisubharmonic function on $M$.
Let $F$ be a holomorphic $(n,0)$ form on $\{\psi<-t_{0}\}$,
such that
\begin{equation}
\label{equ:20171124a}
\int_{K\cap\{\psi<-t_{0}\}}|F|^{2}<+\infty
\end{equation}
for any compact subset $K$ of $M$,
and
\begin{equation}
\label{equ:20171122a}
\int_{M}\frac{1}{B}\mathbb{I}_{\{-t_{0}-B<\psi<-t_{0}\}}|F|^{2}e^{-\varphi}\leq C<+\infty.
\end{equation}
Then there exists a
holomorphic $(n,0)$ form $\tilde{F}$ on $M$, such that
\begin{equation}
\label{equ:3.4}
\begin{split}
\int_{M}&|\tilde{F}-(1-b_{t_0,B}(\psi))F|^{2}e^{-\varphi+v_{t_0,B}(\psi)}c(-v_{t_0,B}(\psi))\leq C\int_{0}^{t_{0}+B}c(t)e^{-t}dt
\end{split}
\end{equation}
where
$b_{t_0,B}(t)=\int_{-\infty}^{t}\frac{1}{B}\mathbb{I}_{\{-t_{0}-B< s<-t_{0}\}}ds$ and
$v_{t_0,B}(t)=\int_{-t_0}^{t}b_{t_0,B}(s)ds-t_0$.
\end{Lemma}

It is clear that $\mathbb{I}_{(-t_{0},+\infty)}\leq b_{t_0,B}(t)\leq\mathbb{I}_{(-t_{0}-B,+\infty)}$ and $\max\{t,-t_{0}-B\}\leq v_{t_0,B}(t) \leq\max\{t,-t_{0}\}$.

\begin{Lemma}[see \cite{GY-concavity}]
	\label{l:converge}
	Let $M$ be a complex manifold. Let $S$ be an analytic subset of $M$.  	
	Let $\{g_j\}_{j=1,2,...}$ be a sequence of nonnegative Lebesgue measurable functions on $M$, which satisfies that $g_j$ are almost everywhere convergent to $g$ on  $M$ when $j\rightarrow+\infty$,  where $g$ is a nonnegative Lebesgue measurable function on $M$. Assume that for any compact subset $K$ of $M\backslash S$, there exist $s_K\in(0,+\infty)$ and $C_K\in(0,+\infty)$ such that
	$$\int_{K}{g_j}^{-s_K}dV_M\leq C_K$$
	 for any $j$, where $dV_M$ is a continuous volume form on $M$.
	
 Let $\{F_j\}_{j=1,2,...}$ be a sequence of holomorphic $(n,0)$ form on $M$. Assume that $\liminf_{j\rightarrow+\infty}\int_{M}|F_j|^2g_j\leq C$, where $C$ is a positive constant. Then there exists a subsequence $\{F_{j_l}\}_{l=1,2,...}$, which satisfies that $\{F_{j_l}\}$ is uniformly convergent to a holomorphic $(n,0)$ form $F$ on $M$ on any compact subset of $M$ when $l\rightarrow+\infty$, such that
 $$\int_{M}|F|^2g\leq C.$$
\end{Lemma}

\begin{Lemma}[see \cite{G-R}]
\label{closedness}
Let $N$ be a submodule of $\mathcal O_{\mathbb C^n,o}^q$, $1\leq q<+\infty$, let $f_j\in\mathcal O_{\mathbb C^n}(U)^q$ be a sequence of $q-$tuples holomorphic in an open neighborhood $U$ of the origin $o$. Assume that the $f_j$ converge uniformly in $U$ towards  a $q-$tuples $f\in\mathcal O_{\mathbb C^n}(U)^q$, assume furthermore that all germs $(f_{j},o)$ belong to $N$. Then $(f,o)\in N$.	
\end{Lemma}

Let $\Omega_j$  be an open Riemann surface, which admits a nontrivial Green function $G_{\Omega_j}$ for any  $1\le j\le n$. Let $M=\prod_{1\le j\le n}\Omega_j$ be an $n-$dimensional complex manifold, and let $\pi_j$ be the natural projection from $M$ to $\Omega_j$. Let $\tilde{Z}_j:=\{z_{j,k}:1\le k<\tilde m_j\}$ be a discrete subset of $\Omega_j$ for any  $j\in\{1,2,...,n\}$, where $\tilde{m}_j\in\mathbb{Z}_{\ge2}\cup\{+\infty\}$.

Let $\varphi_j$ be a subharmonic function on $\Omega_j$ such that $\varphi(z_{j,k})>-\infty$ for any $j\in\{1,2,...,n\}$ and $1\leq k<\tilde{m}_j$.
Let $p_{j,k}$ be a positive number such that $\sum_{1\le k<\tilde{m}_j}p_{j,k}G_{\Omega_j}(\cdot,z_{j,k})\not\equiv-\infty$ for any $j$. Denote that $\varphi:=\sum_{1\le j\le n}\pi_j^*(\varphi_j)$ and  $\psi:=\max_{1\le j\le n}\{\pi_j^*(2\sum_{1\le k<\tilde{m}_j}p_{j,k}G_{\Omega_j}(\cdot,z_{j,k}))\}$. Let $c$ be a positive function on $(0,+\infty)$ such that $\int_{0}^{+\infty}c(t)e^{-t}dt<+\infty$ and $c(t)e^{-t}$ is decreasing on $(0,+\infty)$.

Let $w_{j,k}$ be a local coordinate on a neighborhood $V_{z_{j,k}}\Subset\Omega_{j}$ of $z_{j,k}\in\Omega_j$ satisfying $w_{j,k}(z_{j,k})=0$ for any $j\in\{1,2,...,n\}$ and $1\le k<\tilde{m}_j$, where $V_{z_{j,k}}\cap V_{z_{j,k'}}=\emptyset$ for any $j$ and $k\not=k'$. Denote that $\tilde I_1:=\{(\beta_1,\beta_2,...,\beta_n):1\le \beta_j< \tilde m_j$ for any $j\in\{1,2,...,n\}\}$, $V_{\beta}:=\prod_{1\le j\le n}V_{z_{j,\beta_j}}$ and $w_{\beta}:=(w_{1,\beta_1},w_{2,\beta_2},...,w_{n,\beta_n})$ is a local coordinate on $V_{\beta}$ of $z_{\beta}:=(z_{1,\beta_1},z_{2,\beta_2},...,z_{n,\beta_n})\in M$ for any $\beta=(\beta_1,\beta_2,...,\beta_n)\in\tilde I_1$.
 Denote that
\begin{equation}
	\label{eq:1126a}c_{j,k}=\exp\lim_{z\rightarrow z_{j,k}}(\frac{\sum_{1\le k_1<\tilde m_j}p_{j,k_1}G_{\Omega_j}(z,z_{j,k_1})}{p_{j,k}}-\log|w_{j,k}(z)|)\end{equation}
 for any $j\in\{1,2,...,n\}$ and $1\le k<\tilde m_j$.
Denote that $E_{\beta}:=\{(\alpha_1,\alpha_2,...,\alpha_n):\sum_{1\le j\le n}\frac{\alpha_j+1}{p_{j,\beta_j}}=1\,\&\,\alpha_j\in\mathbb{Z}_{\ge0}\}$.
 Let $f$ be a holomorphic $(n,0)$ form on $\cup_{\beta\in\tilde I_1}V_{\beta}$ such that $f=\sum_{\alpha\in E_{\beta}}d_{\beta,\alpha}w_{\beta}^{\alpha}dw_{1,\beta_1}\wedge dw_{2,\beta_2}\wedge...\wedge dw_{n,\beta_n}$ on $V_{\beta}$ for any $\beta\in\tilde I_1$.
 \begin{Proposition}
 	\label{p:exten-pro-finite}If $\sum_{\beta\in\tilde I_1}\sum_{\alpha\in E_{\beta}}\frac{|d_{\beta,\alpha}|^2(2\pi)^ne^{-\varphi(z_{\beta})}}{\Pi_{1\le j\le n}(\alpha_j+1)c_{j,\beta_j}^{2\alpha_{j}+2}}<+\infty$, there exists a holomorphic $(n,0)$ form $F$ on $M$, which satisfies that $(F-f,z_{\beta})\in(\mathcal{O}(K_{M})\otimes\mathcal{I}(\psi))_{z_{\beta}}$ for any $\beta\in\tilde I_1$ and
 	\begin{displaymath}
 		\int_{M}|F|^2e^{-\varphi}c(-\psi)\le(\int_{0}^{+\infty}c(s)e^{-s}ds)\sum_{\beta\in\tilde I_1}\sum_{\alpha\in E_{\beta}}\frac{|d_{\beta,\alpha}|^2(2\pi)^ne^{-\varphi(z_{\beta})}}{\Pi_{1\le j\le n}(\alpha_j+1)c_{j,\beta_j}^{2\alpha_{j}+2}}.
 	\end{displaymath}
 \end{Proposition}
\begin{proof}
The following Remark shows that it suffices to prove Proposition \ref{p:exten-pro-finite} for the case $\tilde m_j<+\infty$ for any $j\in\{1,2,...,n\}$.
\begin{Remark} Assume that Proposition \ref{p:exten-pro-finite} holds for the case $\tilde m_j<+\infty$ for any $j\in\{1,2,...,n\}$.
For any $j\in\{1,2,...,n\}$, it follows from Lemma \ref{l:green-approx} that there exists a sequence of Riemann surfaces $\{\Omega_{j,l}\}_{l\in\mathbb{Z}_{\ge1}}$, which satisfies that $\Omega_{j,l}\Subset\Omega_{j,l+1}\Subset\Omega_{j}$ for any $l$, $\cup_{l\in\mathbb{Z}_{\ge1}}\Omega_{j,l}=\Omega_j$ and $\{G_{\Omega_{j,l}}(\cdot,z)-G_{\Omega_j}(\cdot,z)\}_{l\in\mathbb{Z}_{\ge1}}$ is decreasingly convergent to $0$ with respect to $l$ for any $z\in\Omega_j$. As $\tilde{Z}_j$ is a discrete subset of $\Omega_j$, $\tilde Z_{j,l}:=\Omega_{j,l}\cap\tilde{Z}_{j}$ is a set of finite points. Denote that $M_l:=\Pi_{1\le j\le n}\Omega_{j,l}$ and $\psi_l:=\max_{1\le j\le n}\{\pi_{j}^*(\sum_{z_{j,k}\in \tilde{Z}_{j,l}}2p_{j,k}G_{\Omega_{j,l}}(\cdot,z_{j,k}))\}$  on $M_l$.
Denote that $$c_{j,k,l}=\exp\lim_{z\rightarrow z_{j,k}}(\frac{\sum_{z_{j,k_1}\in\Omega_{j,l}}p_{j,k_1}G_{\Omega_{j,l}}(z,z_{j,k_1})}{p_{j,k}}-\log|w_{j,k}(z)|)$$
for $z_{j,k}\in\Omega_{j,l}$.
Note that $\{c_{j,k,l}\}$ is decreasingly convergent to $c_{j,k}$ and $\{\psi_l\}$ is decreasingly convergent to $\psi$ with respect to $l$.

Then there exists a holomorphic $(n,0)$ form $F_l$ on $M_l$ such that $(F_l-f,z_{\beta})\in(\mathcal{O}(K_{M})\otimes\mathcal{I}(\psi))_{z_{\beta}}$ for any $\beta\in\{\tilde\beta\in\tilde{I}_1:z_{\tilde\beta}\in M_l\}$ and
\begin{displaymath}
	\begin{split}
		\int_{M_l}|F_l|^2e^{-\varphi}c(-\psi_l)&\leq(\int_0^{+\infty}c(s)e^{-s}ds)\sum_{\beta\in\{\tilde\beta\in\tilde{I}_1:z_{\tilde\beta}\in M_l\}}\sum_{\alpha\in E_{\beta}}\frac{|d_{\beta,\alpha}|^2(2\pi)^ne^{-\varphi(z_{\beta})}}{\Pi_{1\le j\le n}(\alpha_j+1)c_{j,\beta_j,l}^{2\alpha_{j}+2}}\\
		&\le (\int_{0}^{+\infty}c(s)e^{-s}ds)\sum_{\beta\in\tilde I_1}\sum_{\alpha\in E_{\beta}}\frac{|d_{\beta,\alpha}|^2(2\pi)^ne^{-\varphi(z_{\beta})}}{\Pi_{1\le j\le n}(\alpha_j+1)c_{j,\beta_j}^{2\alpha_{j}+2}}.
	\end{split}
\end{displaymath}
As $\psi\le\psi_l$ and $c(t)e^{-t}$ is decreasing on $(0,+\infty)$, we have
\begin{equation}
	\label{eq:211130a}\begin{split}
		&\int_{M_l}|F_l|^2e^{-\varphi-\psi_l+\psi}c(-\psi)\\
		\le&\int_{M_l}|F_l|^2e^{-\varphi}c(-\psi_l)\\
		\leq&(\int_0^{+\infty}c(s)e^{-s}ds)\sum_{\beta\in\tilde I_1}\sum_{\alpha\in E_{\beta}}\frac{|d_{\beta,\alpha}|^2(2\pi)^ne^{-\varphi(z_{\beta})}}{\Pi_{1\le j\le n}(\alpha_j+1)c_{j,\beta_j}^{2\alpha_{j}+2}}.
	\end{split}
\end{equation}
Note that $\psi$ is continuous on $M\backslash \{z_{\beta}:\beta\in \tilde I_1\}$, $\psi_l$ is continuous on $M_l\backslash \{z_{\beta}:\beta\in \tilde I_1\}$ and $\{z_{\beta}:\beta\in \tilde I_1\}$ is a discrete subset of $M$.
 For any compact subset $K$ of $M\backslash \{z_{\beta}:\beta\in \tilde I_1\}$, there exist $l_K>0$ such that $K\Subset\Pi_{1\le j\le n}\Omega_{j,l_K}$ and $C_K>0$  such that $\frac{e^{\varphi+\psi_l-\psi}}{c(-\psi)}\le C_K$ for any $l\ge l_K$.
 It follows from Lemma \ref{l:converge} and the diagonal method that there exists a subsequence of $\{F_l\}$, denoted still by $\{F_l\}$, which is uniformly convergent to a holomorphic $(n,0)$ form $F$ on $M$ on any compact subset of $M$. It follows from the Fatou's Lemma and inequality \eqref{eq:211130a} that
\begin{displaymath}
	\begin{split}
		\int_{M}|F|^2e^{-\varphi}c(-\psi)&=\int_{M}\lim_{l\rightarrow+\infty}|F_l|^2e^{-\varphi-\psi_l+\psi}c(-\psi)\\
		&\leq\liminf_{l\rightarrow+\infty}\int_{M_l}|F_l|^2e^{-\varphi-\psi_l+\psi}c(-\psi)\\
		&\le(\int_0^{+\infty}c(s)e^{-s}ds)\sum_{\beta\in\tilde I_1}\sum_{\alpha\in E_{\beta}}\frac{|d_{\beta,\alpha}|^2(2\pi)^ne^{-\varphi(z_{\beta})}}{\Pi_{1\le j\le n}(\alpha_j+1)c_{j,\beta_j}^{2\alpha_{j}+2}}.	\end{split}
\end{displaymath}
 Since $\{F_l\}$ is uniformly convergent to   $F$ on any compact subset of $M$ and $(F_l-f,z_{\beta})\in(\mathcal{O}(K_{M})\otimes\mathcal{I}(\psi))_{z_{\beta}}$ for any $\beta\in\{\tilde\beta\in\tilde{I}_1:z_{\tilde\beta}\in M_l\}$, following from Lemma \ref{closedness}, we have  $(F-f,z_{\beta})\in(\mathcal{O}(K_{M})\otimes\mathcal{I}(\psi))_{z_{\beta}}$ for any $\beta\in\tilde I_1$.
\end{Remark}

Denote that $m_j=\tilde{m}_j-1$. As $M$ is a Stein manifold, then there exist smooth plurisubharmonic functions $\Phi_l$ on $M$, which are decreasingly convergent to $\varphi$ with respect to $l$.
 It follows from Lemma \ref{l:green-sup} and Lemma \ref{l:green-sup2} that there exists a local coordinate $\tilde{w}_{j,k}$ on a neighborhood $\tilde{V}_{z_{j,k}}\Subset V_{z_{j,k}}$ of $z_{j,k}$ satisfying $\tilde{w}_{j,k}(z_{j,k})=0$ and $|\tilde{w}_{j,k}|=e^{\frac{\sum_{1\le k_1\le m_j}p_{j,k_1}G_{\Omega_j}(\cdot,z_{j,k_1})}{p_{j,k}}}$ on $\tilde{V}_{z_{j,k}}$. Denote that $\tilde{V}_{\beta}:=\Pi_{1\le j\le n}\tilde V_{j,\beta_j}$ for any $\beta\in I_1$. Let $\tilde{f}$ be a holomorphic $(n,0)$ form on $\cup_{\beta\in I_1}\tilde{V}_{\beta}$ satisfying
$$\tilde{f}=\sum_{\alpha\in E_{\beta}}\tilde{d}_{\beta,\alpha}\tilde{w}_{\beta}^{\alpha}d\tilde{w}_{1,\beta_1}\wedge d\tilde{w}_{2,\beta_2}\wedge...\wedge d\tilde{w}_{n,\beta_n}$$
on $V_{\beta}$,  where $\tilde{d}_{\beta,\alpha}=d_{\beta,\alpha}(\lim_{z\rightarrow z_{j,\beta_j}}\frac{w_{j,\beta_j}(z)}{\tilde{w}_{j,\beta_j}(z)})^{\alpha_j+1}$.
It follows from Lemma \ref{l:0} that $(f-\tilde{f},z_{\beta})\in(\mathcal{O}(K_{M})\otimes\mathcal{I}(\psi))_{z_{\beta}}$ for any $\beta\in I_1$. Following from equality \eqref{eq:1126a}, we have $|\tilde{d}_{\beta,\alpha}|=|\frac{d_{\beta,\alpha}}{\Pi_{1\le j\le n}c_{j,\beta_j}^{\alpha_j+1}}|$.
It follows from Lemma \ref{l:G-compact} and Lemma \ref{l:green-sup2} that there exists $t_0>0$ such that $\{\psi<-t_0\}\Subset \cup_{\beta\in I_1}\tilde V_{\beta}$, which implies that $\int_{\{\psi<-t\}}|\tilde f|^2<+\infty$.

Using Lemma \ref{lem:L2}, there exists a holomorphic $(n,0)$ form $F_{l,t}$ on $M$ such that
\begin{equation}
	\label{eq:1125m}
	\begin{split}
&\int_{M}|F_{l,t}-(1-b_{t,1}(\psi))\tilde{f}|^{2}e^{-\Phi_l-\psi+v_{t,1}(\psi)}c(-v_{t,1}(\psi))\\
\leq& (\int_{0}^{t+1}c(s)e^{-s}ds) \int_{M}\mathbb{I}_{\{-t-1<\psi<-t\}}|\tilde f|^2e^{-\Phi_l-\psi},
\end{split}
\end{equation}
where $t\ge t_0$. Note that $b_{t,1}(s)=0$ for large enough $s$, then $(F_{l,t}-\tilde f,z_{\beta})\in(\mathcal{O}(K_{M})\otimes\mathcal{I}(\psi))_{z_{\beta}}$ for any $\beta\in I_1$.

For any $\epsilon>0$, there exists $t_1>t_0$, such that
 $\sup_{z\in\{\psi<-t_1\}\cap \tilde{V}_{\beta}}|\Phi_l(z)-\Phi_l(z_{\beta})|<\epsilon$ for any $\beta\in I_1$. As $\{\psi<-t_1\}\Subset \cup_{\beta\in I_1}\tilde{V}_{\beta}$, it follows from Lemma \ref{l:m1} that
\begin{equation}
	\label{eq:1203a}\int_{M}\mathbb{I}_{\{-t-1<\psi<-t\}}|\tilde{f}|^2e^{-\Phi_l-\psi}\le \sum_{\beta\in\tilde I_1}\sum_{\alpha\in E_{\beta}}\frac{|d_{\beta,\alpha}|^2(2\pi)^ne^{-\Phi_l(z_{\beta})+\epsilon}}{\Pi_{1\le j\le n}(\alpha_j+1)c_{j,\beta_j}^{2\alpha_{j}+2}}.
\end{equation}
Letting $t\rightarrow+\infty$ and $\epsilon\rightarrow0$, inequality \eqref{eq:1203a} implies that
\begin{equation}
	\label{eq:1126c}\limsup_{t\rightarrow+\infty}\int_{M}\mathbb{I}_{\{-t-1<\psi<-t\}}|\tilde{f}|^2e^{-\Phi_l-\psi}\le \sum_{\beta\in\tilde I_1}\sum_{\alpha\in E_{\beta}}\frac{|d_{\beta,\alpha}|^2(2\pi)^ne^{-\Phi_l(z_{\beta})}}{\Pi_{1\le j\le n}(\alpha_j+1)c_{j,\beta_j}^{2\alpha_{j}+2}}.
\end{equation}
As $v_{t,1}(\psi)\ge\psi$ and $c(t)e^{-t}$ is decreasing, Combining inequality \eqref{eq:1125m} and \eqref{eq:1126c}, then we have
\begin{equation}
	\label{eq:1126d}
	\begin{split}
&\limsup_{t\rightarrow+\infty}\int_{M}|F_{l,t}-(1-b_{t,1}(\psi))\tilde{f}|^{2}e^{-\Phi_l}c(-\psi)\\
\le&\limsup_{t\rightarrow+\infty}\int_{M}|F_{l,t}-(1-b_{t,1}(\psi))\tilde{f}|^{2}e^{-\Phi_l-\psi+v_{t,1}(\psi)}c(-v_{t,1}(\psi))\\
\leq& \limsup_{t\rightarrow+\infty}(\int_{0}^{t+1}c(s)e^{-s}ds) \int_{M}\mathbb{I}_{\{-t-1<\psi<-t\}}|\tilde f|^2e^{-\Phi_l-\psi}\\
\leq&(\int_0^{+\infty}c(s)e^{-s}ds)\sum_{\beta\in\tilde I_1}\sum_{\alpha\in E_{\beta}}\frac{|d_{\beta,\alpha}|^2(2\pi)^ne^{-\Phi_l(z_{\beta})}}{\Pi_{1\le j\le n}(\alpha_j+1)c_{j,\beta_j}^{2\alpha_{j}+2}}\\
	<&+\infty.
\end{split}
\end{equation}
For any open set $K\Subset M\backslash \{z_{\beta}:\beta\in\tilde I_1\}$, it follows from $b_{t,1}(s)=1$ for any $s\ge -t$ and $c(s)e^{-s}$ is decreasing with respect to $s$ that there exists a constant $C_K>0$ such that
$$\int_{K}|(1-b_{t,1}(\psi))\tilde{f}|^2e^{-\Phi_l}c(-\psi)\le C_K\int_{\{\psi<-t_1\}}|\tilde{f}|^2<+\infty$$ for any $t>t_1$,
which implies that
$$\limsup_{t\rightarrow+\infty}\int_{K}|F_{l,t}|^2e^{-\Phi_l}c(-\psi)<+\infty.$$
Using Lemma \ref{l:converge} and the diagonal method,
we obtain that
there exists a subsequence of $\{F_{l,t}\}_{t\rightarrow+\infty}$ denoted by $\{F_{l,t_m}\}_{m\rightarrow+\infty}$
uniformly convergent on any compact subset of $M\backslash \{z_{\beta}:\beta\in\tilde I_1\}$. As $\{z_{\beta}:\beta\in\tilde I_1\}$ is discrete subset of $M$, we obtain that $\{F_{l,t_m}\}_{m\rightarrow+\infty}$ is uniformly convergent to a holomorphic $(n,0)$ form $F_{l}$ on $M$ on any compact subset of $M$. Then it follows from inequality \eqref{eq:1126d} and the Fatou's Lemma that
\begin{displaymath}
\begin{split}
&\int_{M}|F_{l}|^{2}e^{-\Phi_l}c(-\psi)
\\=&\int_{\Omega}\liminf_{m\rightarrow+\infty}|F_{l,t_m}-(1-b_{t_m,1}(\psi))\tilde{f}|^{2}e^{-\Phi_l}c(-\psi)\\
\leq&\liminf_{m\rightarrow+\infty}\int_{M}|F_{l,t_m}-(1-b_{t_m,1}(\psi))\tilde{f}|^{2}e^{-\Phi_l}c(-\psi)\\
\leq&(\int_0^{+\infty}c(s)e^{-s}ds)\sum_{\beta\in \tilde I_1}\sum_{\alpha\in E_{\beta}}\frac{|d_{\beta,\alpha}|^2(2\pi)^ne^{-\Phi_l(z_{\beta})}}{\Pi_{1\le j\le n}(\alpha_j+1)c_{j,\beta_j}^{2\alpha_{j}+2}}\\
	<&+\infty.
\end{split}	
\end{displaymath}
Note that $\lim_{l\rightarrow+\infty}\Phi_l(z_{\beta})=\varphi(z_{\beta})>-\infty$ for any $\beta\in I_1$, then we have
\begin{equation}
	\label{eq:1126e}\begin{split}
		&\limsup_{l\rightarrow+\infty}\int_{M}|F_{l}|^{2}e^{-\Phi_l}c(-\psi)\\
		\leq&(\int_0^{+\infty}c(s)e^{-s}ds)\sum_{\beta\in\tilde I_1}\sum_{\alpha\in E_{\beta}}\frac{|d_{\beta,\alpha}|^2(2\pi)^ne^{-\varphi(z_{\beta})}}{\Pi_{1\le j\le n}(\alpha_j+1)c_{j,\beta_j}^{2\alpha_{j}+2}}\\
	<&+\infty.
	\end{split}
\end{equation}
Using Lemma \ref{l:converge},
we obtain that
there exists a subsequence of $\{F_{l}\}$ (also denoted by $\{F_{l}\}$)
uniformly convergent to a holomorphic $(n,0)$ form $F$ on $M$ on any compact subset of $M$, which satisfies that
\begin{equation*}
\int_{M}|F|^2e^{-\varphi}c(-\psi)\le (\int_0^{+\infty}c(s)e^{-s}ds)\sum_{\beta\in\tilde I_1}\sum_{\alpha\in E_{\beta}}\frac{|d_{\beta,\alpha}|^2(2\pi)^ne^{-\varphi(z_{\beta})}}{\Pi_{1\le j\le n}(\alpha_j+1)c_{j,\beta_j}^{2\alpha_{j}+2}}.
\end{equation*}
It follows from Lemma \ref{closedness}, we have $(F-f,z_{\beta})\in(\mathcal{O}(K_{M})\otimes\mathcal{I}(\psi))_{z_{\beta}}$ for any $\beta\in I_1$.

Thus, Proposition \ref{p:exten-pro-finite} holds.
\end{proof}

\begin{Lemma}
	\label{l:psi=G}Let $\psi=\max_{1\le j\le n}\{\pi_j^*(2\sum_{1\le k<\tilde m_j}p_{j,k}G_{\Omega_j}(\cdot,z_{j,k}))\}$ be a plurisubharmonic function on $M$, where $\sum_{1\le k<\tilde m_j}p_{j,k}G_{\Omega_j}(\cdot,z_{j,k})\not\equiv-\infty$ for any $j\in\{1,2,..,n\}$.	 Let $\Psi\le0$ be a plurisubharmonic function on $M$, and denote that $\tilde\psi:=\psi+\Psi$. Let $l(t)$ be a positive Lebesgue measurable function on $(0,+\infty)$ satisfying $l$ is decreasing on $(0,+\infty)$ and $\int_0^{+\infty}l(t)dt<+\infty$. If $\Psi\not\equiv 0$ on $M$, there exists a Lebesgue measurable subset $V$ of $M$ such that  $l(-\tilde\psi(z))<l(-\psi(z))$ for any  $z\in V$ and $\mu(V)>0$, where $\mu$ is the Lebesgue measure on $M$.
\end{Lemma}
\begin{proof} Let $U_0\Subset M\backslash\{z_\beta:\exists j\in\{1,2,...,n\}$ s.t. $2\le \beta_j<\tilde m_j \}$ be a neighborhood of $z_{\beta^*}$, where $\beta^*=(1,1,...,1)\in\tilde I_1$.
	It follows from Lemma \ref{l:green-sup2} and Lemma \ref{l:G-compact} that there exists $t_0>0$ such that $\{z\in U_0:\psi<-t_0\}\Subset U_0$. As $l$ is decreasing and $\int_0^{+\infty}l(t)dt<+\infty$, then there exists $t_1>t_0$ such that $l(t)<l(t_1)$ holds for any $t>t_1.$
	
	 As $\tilde\psi$ and $\Psi\not\equiv0$ are upper semicontinuous, we have $\sup_{z\in\{\psi\leq-t_1\}\cap U_0}\tilde\psi(z)<-t_1$, which implies that there exists  $t_2\in(t_0,t_1)$ such that $\sup_{z\in\{\psi\leq-t_2\}\cap U_0}\tilde\psi(z)<-t_1$. Denote that $t_3:=-\sup_{z\in\{\psi\leq-t_2\}\cap U_0}\tilde\psi(z)$. Let $V=\{z\in\Omega:-t_1<\psi<-t_2\}\cap U_0$, then $\mu(V)>0$. As $l(t)$ is decreasing on $(0,+\infty)$, for any $z\in V$, we have
	$$l(-\tilde\psi(z))\leq l(t_3)<l(t_1)\leq l(-\psi).$$
	Thus, Lemma \ref{l:psi=G} holds.
\end{proof}

 Let $Z_j:=\{z_{j,1},z_{j,2},...,z_{j,m_j}\}\subset\Omega_j$ for any  $j\in\{1,2,...,n\}$, where $m_j$ is a positive integer.
For any $j\in\{1,2,...,n\}$, let $\varphi_j$ be a subharmonic function on $\Omega_j$ such that  $\varphi_j(z_{j,k})>-\infty$  for any $k\in\{1,2,...,m_j\}$.
Let $\varphi=\sum_{1\le j\le n}\pi_j^*(\varphi_j)$ and $\psi=\max_{1\le j\le n}\{\pi_j^*(2\sum_{1\le k\le m_j}p_{j,k}G_{\Omega_j}(\cdot,z_{j,k}))\}$ on $M=\Pi_{1\le j\le n}\Omega_j$, where $p_{j,k}>0$ for any $1\le j\le n$ and $1\le k\le m_j$.

Let $w_{j,k}$ be a local coordinate on a neighborhood $V_{z_{j,k}}\Subset\Omega_{j}$ of $z_{j,k}\in\Omega_j$ satisfying $w_{j,k}(z_{j,k})=0$ for any $j\in\{1,2,...,n\}$ and $k\in\{1,2,...,m_j\}$, where $V_{z_{j,k}}\cap V_{z_{j,k'}}=\emptyset$ for any $j$ and $k\not=k'$. Denote that $I_1:=\{(\beta_1,\beta_2,...,\beta_n):1\le \beta_j\le m_j$ for any $j\in\{1,2,...,n\}\}$, $V_{\beta}:=\prod_{1\le j\le n}V_{z_{j,\beta_j}}$ for any $\beta=(\beta_1,\beta_2,...,\beta_n)\in I_1$ and $w_{\beta}:=(w_{1,\beta_1},w_{2,\beta_2},...,w_{n,\beta_n})$ is a local coordinate on $V_{\beta}$ of $z_{\beta}:=(z_{1,\beta_1},z_{2,\beta_2},...,z_{n,\beta_n})\in M$. Denote that $\Psi_j:=2\sum_{1\le k\le m_j}p_{j,k}G_{\Omega_j}(\cdot,z_{j,k})$,
\begin{displaymath}
	\begin{split}
		H_j:=\{f\in H^0(\Omega_j,&\mathcal{O}(K_{\Omega_j})):\int_{\Omega_j}|f|^2e^{-\varphi_j}<+\infty\\
		&\&(f,z_{j,k})\in(\mathcal{O}(K_{\Omega_j})\otimes\mathcal{I}(\Psi_j))_{z_{j,k}}\mbox{for any $k\in\{1,2,...,m_j\}$}\}
	\end{split}
\end{displaymath}
and
\begin{displaymath}
	\begin{split}
	H_0:=\{f\in H^0(M,&\mathcal{O}(K_{M})):\int_{M}|f|^2e^{-\varphi}<+\infty\\
		&\&(f,z_{\beta})\in(\mathcal{O}(K_{M})\otimes\mathcal{I}(\psi))_{z_{\beta}}\mbox{for any $\beta \in I_1$}\}.	
	\end{split}
\end{displaymath}
We give an orthogonal property related to $H_j$ and $H_0$, which will be used in the proofs of Lemma \ref{l:orth1} and Theorem \ref{thm:prod-finite-point}.
\begin{Lemma}
	\label{l:orth2}Let $f_j$ be a holomorphic $(1,0)$ form on $\Omega_j$ such that $\int_{\Omega_j}|f_j|^2e^{-\varphi_j}<+\infty$ and $\int_{\Omega_j}f_j\wedge\overline fe^{-\varphi_j}=0$ for any $f\in H_j$. Denote that $\gamma_{j,k}:=ord_{z_{j,k}}f_j$ for any $j\in\{1,2,...,n\}$ and $k\in\{1,2,...,m_j\}$. If $\sum_{1\le j\le n}\frac{\gamma_{j,\beta_j}+1}{p_{j,\beta_j}}=1$ for any $\beta\in I_1$, then $\int_{M}f\wedge\overline{\wedge_{1\le j\le n} \pi_j^*(f_j)}e^{-\varphi}=0$ for any $f\in H_0$.
 \end{Lemma}
\begin{proof}
	We will prove Lemma \ref{l:orth2} by induction on $n$. For the case $n=1$, Lemma \ref{l:orth2} is trivial. We assume that Lemma \ref{l:orth2} holds for $n-1$.
	
As $\Omega_j$ is an open Riemann surface, there exists a holomorphic $(1,0)$ form $h_j$ on $\Omega_j$ such that $h_j(z)\not=0$ for any $z\in\Omega_j$ (see \cite{OF81}). There exist a holomorphic function $\tilde{f}_j$ on $\Omega_j$ such that $f_j=\tilde{f}_jh_j$ and a holomorphic function $\tilde{f}$ on $M$ such that $f=\tilde{f}\wedge_{1\le j\le n}\pi_j^*(h_j)$. Let $\tilde{M}=\prod_{1\le j\le n-1}\Omega_j$ be an $(n-1)-$dimensional complex manifold, and let $\tilde{\pi}_j$ be the natural projection from $\tilde{M}$ to $\Omega_j$. Denote that
	$\tilde{\varphi}:=\sum_{1\le j\le n-1}\tilde{\pi}_j^*(\varphi_j)$.
	As $|\tilde f(x,y_n)|^2\le C_{y_n}\int_{\Omega_n}|\tilde f(x,\cdot)|^2|h_n|^2e^{-\varphi_n}$ for any $x\in\tilde M$ and $y_n\in\Omega_n$, where $C_{y_n}$ is a constant independent of $x$, we know
$$\int_{\tilde M}|\tilde f(\cdot,y_n)|^2e^{-\tilde{\varphi}}|\wedge_{1\le j\le n-1}\tilde{\pi}_j^*(h_j)|^2\le C_{y_n}\int_{M}|f|^2e^{-\varphi}<+\infty.$$ Denote that
	\begin{equation}
		\label{eq:1124a}\tilde{F}(y_n):=\int_{\tilde{M}}\tilde f(\cdot,y_n)\overline{\Pi_{1\le j\le n-1}\tilde{\pi}^*_j(\tilde{f}_j)}e^{-\tilde{\varphi}}|\wedge_{1\le j\le n-1}\tilde{\pi}_j^*(h_j)|^2,
	\end{equation}
where $y_n\in\Omega_n$, then $\tilde{F}$ is a function on $\Omega_n$. Take $F=\tilde{F}h_n$.

In the following, we prove $F\in H_n$.

Let $y_n\in\Omega_n$, and let $w_n$ be a local coordinate on a neighborhood $U_n\Subset\Omega_n$ of $y_n$ satisfying $w_n(y_n)=0$. For any nonnegative $\sigma$, $(\frac{\partial}{\partial w_n})^{\sigma}\tilde{f}$ is a holomorphic function on $\tilde{M}\times U_n$. Without misunderstanding, we see $U_n$ and the unit disc $\Delta$  the same ($y_n=0$). There exists a constant $C_{r,\sigma}$ such that
\begin{equation}
	\label{eq:1124b}\sup_{|w_n|<r}|(\frac{\partial}{\partial w_n})^{\sigma}\tilde{f}(x,w_n)|^2\le C_{r,\sigma}\int_{\Omega_n}|\tilde{f}(x,\cdot)|^2|h_n|^2e^{-\varphi_n}
\end{equation}
for any $x\in\tilde M$, where $r\in(0,1)$. It follows from inequality \eqref{eq:1124b}  and $\int_{M}|f|^2e^{-\varphi}<+\infty$ that
\begin{equation}
	\label{eq:1124d}\int_{\tilde M}|(\frac{\partial}{\partial w_n})^{\sigma}\tilde{f}(\cdot,w_n)|^2e^{-\tilde \varphi}|\wedge_{1\le j\le n-1}\tilde{\pi}_j^*(h_j)|^2<+\infty.\end{equation}	
Denote that
$$\tilde{F}_{\sigma}(w_n):=\int_{\tilde{M}}(\frac{\partial}{\partial w_n})^{\sigma}\tilde{f}(\cdot,w_n)\overline{\Pi_{1\le j\le n-1}\tilde{\pi}^*_j(\tilde{f}_j)}e^{-\tilde{\varphi}}|\wedge_{1\le j\le n-1}\tilde{\pi}_j^*(h_j)|^2,$$	
	where $w_n\in \Delta$, then $\tilde {F}_{\sigma}$ is a function on $\Delta$. It is clear the $\tilde{F}=\tilde{F}_0$ on $U_n=\Delta$. It follows from inequality \eqref{eq:1124b} that
	\begin{equation}
		\label{eq:1124c}
		\begin{split}
			&|(\frac{\partial}{\partial w_n})^{\sigma}\tilde{f}(x,w_n)-(\frac{\partial}{\partial w_n})^{\sigma}\tilde{f}(x,\tilde w_n)-(w_n-\tilde w_n)(\frac{\partial}{\partial w_n})^{\sigma+1}f(x,\tilde w_n)|^2\\
			\le&|w_n-\tilde w_n|^4\sup_{|z|<r}|(\frac{\partial}{\partial w_n})^{\sigma+2}\tilde{f}(x,z)|\\
			\le &C_{r,\sigma}\int_{\Omega_n}|\tilde{f}(x,\cdot)|^2|h_n|^2e^{-\varphi_n} |w_n-\tilde w_n|^4
		\end{split}
	\end{equation}
	for any $x\in \tilde{M}$, where $|w_n|<r$ and $|\tilde w_n|<r$.
	Following from inequality \eqref{eq:1124c} and Cauchy-Schwarz inequality, we obtain that
	\begin{displaymath}
		\begin{split}
			&|\tilde F_{\sigma}(w_n)-\tilde{F}_{\sigma}(\tilde w_n)-(w_n-\tilde w_n)\tilde F_{\sigma+1}(\tilde w_n)|^2\\
			\le &|\int_{\tilde{M}}((\frac{\partial}{\partial w_n})^{\sigma}\tilde{f}(\cdot,w_n)-(\frac{\partial}{\partial w_n})^{\sigma}\tilde{f}(\cdot,\tilde w_n)-(w_n-\tilde w_n)(\frac{\partial}{\partial w_n})^{\sigma+1}f(x,\tilde w_n))\\
			&\times\overline{\Pi_{1\le j\le n-1}\tilde{\pi}^*_j(\tilde{f}_j)}e^{-\tilde{\varphi}}|\wedge_{1\le j\le n-1}\tilde{\pi}_j^*(h_j)|^2|^2\\
			\le & C_{r,\sigma}|w_n-\tilde w_n|^4\int_{M}|f|^2e^{-\varphi}\times\Pi_{1\le j\le n-1}\int_{\Omega_j}|f_j|^2e^{-\varphi_j},
					\end{split}
	\end{displaymath}
	which implies that
	$$\frac{\partial}{\partial w_n}\tilde{F}_{\sigma}=\tilde{F}_{\sigma+1}$$
	 on $\Delta$ for any $\sigma\in\mathbb{Z}_{\ge0}$. Thus, we obtain that $F$ is a holomorphic $(1,0)$ form on $\Omega_n$.
	Let
	$$\tilde \psi=\max_{1\le j\le n-1}\{\tilde{\pi}_j^
	*(\sum_{1\le k\le m_j}p_{j,k}(1-\frac{\gamma_{n,k}+1}{p_{n,k}})G_{\Omega_j}(\cdot,z_{j,k}))\}$$
	 be a plurisubharmonic function on $\tilde M$.
 Denote that $I_2:=\{(\tilde\beta_1,\tilde\beta_2,...\tilde\beta_{n-1}):1\le\tilde\beta_j\le m_j\,\&\,\tilde\beta_j\in\mathbb{Z}_{>0}$ for any $j\in\{1,2,...,n-1\}\}$ and $x_{\tilde{\beta}}:=(z_{1,\tilde\beta_1},z_{2,\tilde\beta_2},...,z_{n-1,\tilde\beta_{n-1}})\in\tilde M$ for any $\tilde\beta\in I_2$. As $(f,z_{\beta})\in(\mathcal{O}(K_{M})\otimes\mathcal{I}(\psi))_{z_{\beta}}$ for any $\beta \in I_1$, we have $(\tilde f,z_{\beta})\in\mathcal{I}(\psi)_{z_{\beta}}$ for any $\beta\in I_1$.
	 Taking $y_n=z_{n,k}\in\Omega_n$ ($k\in\{1,2,...,m_n\}$), it follows from Lemma \ref{l:1} that $((\frac{\partial}{\partial w_n})^{\sigma}\tilde f(\cdot,z_{n,k}),x_{\tilde\beta})\in\mathcal{I}(\tilde\psi)_{x_{\tilde\beta}}$ for any $\sigma\in\{0,1,...,\gamma_{n,k}\}$ and $\tilde\beta\in I_2$, which implies that
	 $$((\frac{\partial}{\partial w_n})^{\sigma}\tilde f(\cdot,z_{n,k})\wedge_{1\le j\le n-1}\tilde\pi_j^*(h_j),x_{\tilde\beta})\in(\mathcal{O}(K_{\tilde{M}})\otimes\mathcal{I}(\tilde\psi))_{x_{\tilde\beta}}$$
	  for any $\sigma\in\{0,1,...,\gamma_{n,k}\}$ and $\tilde\beta\in I_2$. As $\sum_{1\le j\le n}\frac{\gamma_{j,\beta_j}+1}{p_{j,\beta_j}}$ for any $\beta\in I_1$, we have
	  $$\sum_{1\le j\le n-1}\frac{\gamma_{j,\tilde\beta_j}+1}{p_{j,\tilde\beta_j}(1-\frac{\gamma_{n,k}+1}{p_{n,k}})}=1$$
	 for any $\tilde\beta\in{I}_2$. As Lemma \ref{l:orth2} holds for $n-1$ and $(\frac{\partial}{\partial w_n})^{\sigma}\tilde f(\cdot,z_{n,k})\wedge_{1\le j\le n-1}\tilde\pi_j^*(h_j)$ is a holomorphic $(n-1,0)$ form on $\tilde M$, it follows from inequality \eqref{eq:1124d} and $((\frac{\partial}{\partial w_n})^{\sigma}\tilde f(\cdot,z_{n,k})\wedge_{1\le j\le n-1}\tilde\pi_j^*(h_j),x_{\tilde\beta})\in(\mathcal{O}(K_{\tilde{M}})\otimes\mathcal{I}(\tilde\psi))_{x_{\tilde\beta}}$ for any $\sigma\in\{0,1,...,\gamma_{n,k}\}$ and $\tilde\beta\in I_2$ that
	 \begin{equation}
	 	\label{eq:1124e}\tilde{F}_{\sigma}(z_{n,k})=0
	 \end{equation}
	 for any $\sigma\in\{0,1,...,\gamma_{n,k}\}$. As $\frac{\partial}{\partial w_n}\tilde{F}_{\sigma}=\tilde{F}_{\sigma+1}$ and $\tilde{F}_0=\tilde{F}$, we obtain that
	 $$(\tilde F,z_{n,k})\in\mathcal{I}(2(\gamma_{n,k}+1)G_{\Omega_n}(\cdot,z_{n,k}))$$
	  for any $k\in\{1,2,...,m_n\}$. Note that $F=\tilde Fh_n$ and $f=\tilde f\wedge_{1\le j\le n}\pi_j^*(h_j)$.
	 It follows from the Fubini's theorem that
	 \begin{displaymath}\begin{split}
	 	&\int_{\Omega_n}|F|^2e^{-\varphi_n}\\
	 	=&\int_{\Omega_n}(\int_{\tilde{M}}\tilde f(\cdot,y_n)\overline{\Pi_{1\le j\le n-1}\tilde{\pi}^*_j(\tilde{f}_j)}e^{-\tilde{\varphi}}|\wedge_{1\le j\le n-1}\tilde{\pi}_j^*(h_j)|^2	)|h_n(y_n)|^2\\
	 	=&\int_{M}|f|^2e^{-\varphi}\\
	 	<&+\infty.
	 \end{split}	 \end{displaymath} Thus, we have $F\in H_n$.
	
	It follows from the Fubini's theorem that
	\begin{displaymath}
		\begin{split}
			&\int_{M}f\wedge\overline{\wedge_{1\le j\le n} \pi_j^*(f_j)}e^{-\varphi}\\
			=&\int_{\Omega_n}(h_n(y_n)\int_{\tilde{M}}\tilde f(\cdot,y_n)\overline{\Pi_{1\le j\le n-1}\tilde{\pi}^*_j(\tilde{f}_j)}e^{-\tilde{\varphi}}|\wedge_{1\le j\le n-1}\tilde{\pi}_j^*(h_j)|^2)\wedge \overline{f_n}(y_n) e^{-\varphi_n}\\
			=&\int_{\Omega_n}F\wedge\overline{f_n}e^{-\varphi_n}\\
			=&0.
		\end{split}
	\end{displaymath}
	We have thus prove Lemma \ref{l:orth2} for $n$. The proof of Lemma \ref{l:orth2} is now complete.
	\end{proof}

Let $z_j\in\Omega_j$ and $z_0\in M$ such that $\pi_j(z_0)=z_j$ for any $1\le j\le n$. Let $\varphi_j$ be  subharmonic functions on $\Omega_j$ such that $\varphi_j(z_j)>-\infty$. Denote that $\psi:=\max_{1\le j\le n}\{2p_j\pi_j^{*}(G_{\Omega_j}(\cdot,z_j))\}$ and $\varphi:=\sum_{1\le j\le n}\pi_j^*(\varphi_j)$ on $M$, where $p_j$ is positive real number for $1\le j\le n$.

Let $w_j$ be a local coordinate on a neighborhood $V_{z_j}$ of $z_j\in\Omega_j$ satisfying $w_j(z_j)=0$. Denote that $V_0:=\prod_{1\le j\le n}V_{z_j}$, and $w:=(w_1,w_2,...,w_n)$ is a local coordinate on $V_0$ of $z_0\in M$.
Take $E=\{(\alpha_1,\alpha_2,...,\alpha_n):\sum_{1\le j\le n}\frac{\alpha_j+1}{p_{j}}=1\,\&\,\alpha_j\in\mathbb{Z}_{\ge0}\}$.

It follows from Proposition \ref{p:exten-pro-finite} and Lemma \ref{l:converge} that there exists a holomorphic  $(1,0)$ form $f_{j,\alpha_j}$ on $\Omega_j$ such that $(f_{j,\alpha_j}-w_j^{\alpha_j}dw_j,z_j)\in(\mathcal{O}(K_{\Omega_j})\otimes \mathcal{I}(2(\alpha_j+1)G_{\Omega_j}(\cdot,z_j)))_{z_j}$ and $\int_{\Omega_j}|f_{j,\alpha_j}|^2e^{-\varphi_j}=\inf\{\int_{\Omega_j}|\tilde{f}|^2e^{-\varphi_j}:\tilde{f}\in H^0(\Omega_j,\mathcal{O}(K_{\Omega_j}))\,\&\,(\tilde{f}-w_j^{\alpha_j}dw_j,z_j)\in(\mathcal{O}(K_{\Omega_j})\otimes \mathcal{I}(2(\alpha_j+1)G_{\Omega_j}(\cdot,z_j)))_{z_j}\}<+\infty$ for any $\alpha\in E$ and $j\in\{1,2,...,n\}$.

The following Lemma gives a property of $\sum_{\alpha\in E}d_{\alpha}\Pi_{1\le j\le n}\pi_{j}^*(f_{j,\alpha_j})$, which will be used in the proofs of Theorem \ref{thm:linear-2d} and Remark \ref{r:1.1}.

\begin{Lemma}
	\label{l:orth1}$F=\sum_{\alpha\in E}d_{\alpha}\Pi_{1\le j\le n}\pi_{j}^*(f_{j,\alpha_j})$ is a holomorphic $(n,0)$ form on $M$ such that $({F}-\sum_{\alpha\in E}d_{\alpha}w^{\alpha}dw_1\wedge dw_2\wedge...\wedge dw_n ,z_0)\in\mathcal{O}(K_{M})\otimes\mathcal{I}(\psi))_{z_0}$,
	$$\int_{M}|F|^2e^{-\varphi}=\sum_{\alpha\in E}|d_{\alpha}|^2\int_{M}|\Pi_{1\le j\le n}\pi_{j}^*(f_{j,\alpha_j})|^2e^{-\varphi}$$
	 and $\int_{M}|F|^2e^{-\varphi}=\inf\{\int_{M}|\tilde{F}|^2e^{-\varphi}:\tilde{F}$ is a holomorphic $(n,0)$ form on $M$ such that $(\tilde{F}-\sum_{\alpha\in E}d_{\alpha}w^{\alpha}dw_1\wedge dw_2\wedge...\wedge dw_n ,z_0)\in\mathcal{O}(K_{M})\otimes\mathcal{I}(\psi))_{z_0}\}$, where $d_{\alpha}$ is a constant for any $\alpha\in E$.
\end{Lemma}
\begin{proof}
As $(f_{j,\alpha_j}-w_j^{\alpha_j}dw_j,z_j)\in(\mathcal{O}(K_{\Omega_j})\otimes \mathcal{I}(2(\alpha_j+1)G_{\Omega_j}(\cdot,z_j)))_{z_j}$ and $\sum_{1\le j\le n}\frac{\alpha_j+1}{p_j}=1$ for any $\alpha\in E$, it follows from Lemma \ref{l:0} that $({F}-\sum_{\alpha\in E}d_{\alpha}w^{\alpha}dw_1\wedge dw_2\wedge...\wedge dw_n ,z_0)\in\mathcal{O}(K_{M})\otimes\mathcal{I}(\psi))_{z_0}$.

	As $f_{j,\alpha_j}$ is a  holomorphic $(1,0)$ form  on $\Omega_j$ such that $(f_{j,\alpha_j}-w_j^{\alpha_j}dw_j,z_j)\in(\mathcal{O}(K_{\Omega_j})\otimes \mathcal{I}(2(\alpha_j+1)G_{\Omega_j}(\cdot,z_j)))_{z_j}$ and $\int_{\Omega_j}|f_{\alpha,j}|^2e^{-\varphi_j}=\inf\{\int_{\Omega_j}|\tilde{f}|^2e^{-\varphi_j}:\tilde{f}\in H^0(\Omega_j,\mathcal{O}(K_{\Omega_j}))\,\&\,(\tilde{f}-w_j^{\alpha_j}dw_j,z_j)\in(\mathcal{O}(K_{\Omega_j})\otimes \mathcal{I}(2(\alpha_j+1)G_{\Omega_j}(\cdot,z_j)))_{z_j}\}$, we obtain that $\int_{\Omega_j}f_{j,\alpha_j}\wedge\overline{f}e^{-\varphi_j}=0$ for any holomorphic $(1,0)$ form $f$ on $\Omega_j$ satisfying $\int_{\Omega_j}|f|^2e^{-\varphi_j}<+\infty$ and $(f,z_j)\in(\mathcal{O}(K_{\Omega_j})\otimes\mathcal{I}(2(\alpha_j+1)G_{\Omega_j}(\cdot,z_j)))_{z_j}$. It follows from Lemma \ref{l:orth2} that
$$\int_{M}f\wedge\overline{\wedge_{1\le j\le n}\pi_j^*(f_{j,\alpha_j})}e^{-\varphi}=0$$
 for any holomorphic $(n,0)$ form $f$ on $M$ satisfying $\int_{M}|f|^2e^{-\varphi}<+\infty$ and $(f,z_0)\in\mathcal{O}(K_{M})\otimes\mathcal{I}(\psi))_{z_0}$ for any $\alpha\in E$. Thus, we have $$\int_{M}f\wedge\overline{F}e^{-\varphi}=0$$
 for any holomorphic $(n,0)$ form $f$ on $M$ satisfying $\int_{M}|f|^2e^{-\varphi}<+\infty$ and $(f,z_0)\in\mathcal{O}(K_{M})\otimes\mathcal{I}(\psi))_{z_0}$ for any $\alpha\in E$, which implies that $\int_{M}|F|^2e^{-\varphi}=\inf\{\int_{M}|\tilde{F}|^2e^{-\varphi}:\tilde{F}$ is a holomorphic $(n,0)$ form on $M$ such that $(\tilde{F}-\sum_{\alpha\in E}d_{\alpha}w^{\alpha}dw_1\wedge dw_2\wedge...\wedge dw_n ,z_0)\in\mathcal{O}(K_{M})\otimes\mathcal{I}(\psi))_{z_0}\}$.

 As $ord_{z_j}f_{j,\alpha_j}=\alpha_j$ and $\int_{\Omega_j}f_{j,\alpha_j}\wedge\overline{f}e^{-\varphi_j}=0$ for any holomorphic $(1,0)$ form $f$ on $\Omega_j$ satisfying $\int_{\Omega_j}|f|^2e^{-\varphi_j}<+\infty$ and $(f,z_j)\in(\mathcal{O}(K_{\Omega_j})\otimes\mathcal{I}(2(\alpha_j+1)G_{\Omega_j}(\cdot,z_j)))_{z_j}$, we have
 $$\int_{M}\wedge_{1\le j\le n}\pi_j^*(f_{j,\alpha_j})\wedge\overline{\wedge_{1\le j\le n}\pi_j^*(f_{j,\tilde{\alpha}_j})}e^{-\varphi}=0$$
 for any $\alpha\in E$ and $\tilde{\alpha}\in E$ satisfying $\alpha\not=\tilde{\alpha}$, which implies that
 	$$\int_{M}|F|^2e^{-\varphi}=\sum_{\alpha\in E}|d_{\alpha}|^2\int_{M}|\Pi_{1\le j\le n}\pi_{j}^*(f_{j,\alpha_j})|^2e^{-\varphi}.$$
 Hence we prove Lemma \ref{l:orth1}.
\end{proof}

The following lemma will be used in the proofs of Theorem \ref{thm:prod-finite-point} and Theorem \ref{thm:prod-infinite-point}.
\begin{Lemma}
	\label{l:decom}Let $U_j$ be an open subset of $\Omega_j$, and let $f_j\not\equiv0$ be a holomorphic $(1,0)$ form on $U_j$ for any $j\in\{1,2,...,n\}$. Let $F$ be a holomorphic $(n,0)$ form on $M$. If $F=\wedge_{1\le j\le n}\pi_j^*(f_j)$ on $\Pi_{1\le j\le n}U_j$, then there exists a holomorphic $(1,0)$ form $\tilde{f}_j$ on $\Omega_j$ for any $j\in\{1,2,...,n\}$ such that $F=\wedge_{1\le j\le n}\pi_j^*(\tilde{f}_j)$ on $M$ and $f_j=\tilde{f}_{j}$ on $U_j$ for any $j\in\{1,2,...,n\}$.
\end{Lemma}
\begin{proof}
As $\Omega_j$ is an open Riemann suface, there exists a holomorphic $(1,0)$ form $h_j$ on $\Omega_j$ such that $h_j(z)\not=0$ for any $z\in \Omega_j$ (see \cite{OF81}). As $\frac{F}{\Pi_{1\le j\le n}\pi_j^*(h_j)}$ is a holomorphic function on $M$ and $\frac{F}{\Pi_{1\le j\le n}\pi_j^*(h_j)}=\Pi_{1\le j\le n}\pi_j^*(\frac{f_j}{h_j})$ on $\Pi_{1\le j\le n}U_j$, there exists a holomorphic function $\tilde h_j$ on $\Omega_j$ such that $\tilde h_j=\frac{f_j}{h_j}$ on $U_j$ for any $j\in\{1,2,...,n\}$. Let $\tilde{f}_j=\tilde{h}_jh_j$ on $\Omega_j$, which satisfies that $F=\wedge_{1\le j\le n}\pi_j^*(\tilde{f}_j)$ on $M$.
\end{proof}

\section{Proofs of Theorem \ref{thm:linear-2d} and Remark \ref{r:1.1}}
In this section, we prove Theorem \ref{thm:linear-2d} and Remark \ref{r:1.1}.

\subsection{Proof of the sufficiency part of Theorem \ref{thm:linear-2d}}
\label{sec:s-1}
\

Denote that $\tilde{E}=\{\alpha\in E:d_{\alpha}\not=0\}$. For any $\alpha=(\alpha_1,\alpha_2,...,\alpha_n)\in\tilde{E}$ and $j\in\{1,2,...,n\}$, it follows from Proposition \ref{p:exten-pro-finite} and Lemma \ref{l:converge} that there exists a holomorphic $(1,0)$ form $f_{j,\alpha_j}$ on $\Omega_j$ such that $(f_{j,\alpha_j}-w_j^{\alpha_j}dw_j,z_j)\in(\mathcal{O}(K_{\Omega_j})\otimes\mathcal{I}(2(\alpha_j+1)G_{\Omega_j}(\cdot,z_j)))_{z_j}$ and $\int_{\Omega_j}|f_{j,\alpha_j}|^2e^{-\varphi_j}=\inf\{\int_{\Omega_j}|\tilde{f}|^2e^{-\varphi_j}:\tilde{f}\in H(\Omega_j,\mathcal{O}(K_{\Omega_j}))\,\&\,(\tilde{f}-w_j^{\alpha_j}dw_j,z_j)\in(\mathcal{O}(K_{\Omega_j})\otimes\mathcal{I}(2(\alpha_j+1)G_{\Omega_j}(\cdot,z_j)))_{z_j}\}<+\infty$.

Note that $f=(\sum_{\alpha\in E}d_{\alpha}w^{\alpha}+g_0)dw_{1}\wedge dw_{2}\wedge...\wedge dw_{n}$, where  $g_0$ is a holomorphic function on $V_0$ such that $(g_0,z_0)\in\mathcal{I}(\psi)_{z_0}$. As $\psi=\max_{1\le j\le n}\{2p_j\pi_j^*(G_{\Omega_j}(\cdot,z_j))\}$, following Lemma \ref{l:orth1} and $\sum_{1\le j\le n}\frac{\alpha_j+1}{p_j}=1$ for any $\alpha\in E$, we get that
\begin{displaymath}
	\begin{split}
		&\sum_{\alpha\in E}(|d_{\alpha}|^2\Pi_{1\le j\le n}\int_{\Omega_j}|f_{j,\alpha_j}|^2e^{-\varphi_j})\\
		=&\int_{M}|\sum_{\alpha\in E}d_{\alpha}\Pi_{1\le j\le n}\pi_j^*(f_{j,\alpha_j})|^2e^{-\varphi}\\
		=&\inf\{\int_{M}|\tilde{f}|^2e^{-\varphi}:\tilde{f}\in H^0(M,\mathcal{O}(K_M))\,\&\,(\tilde{f}-f,z_0)\in(\mathcal{O}(K_{\Omega})\otimes\mathcal{I}(\psi))_{z_0}\}.
			\end{split}
\end{displaymath}
Denote that $G_{j,\alpha_j}(t)=\inf\{\int_{\{2(\alpha_j+1)G_{\Omega_j}(\cdot,z_j)<-t\}}|\tilde{f}|^2e^{-\varphi_j}:\tilde{f}$ is a holomorphic $(1,0)$ form on $\{2(\alpha_j+1)G_{\Omega_j}(\cdot,z_j)<-t\}$ such that $(\tilde{f}-w_j^{\alpha_j},z_j)\in(\mathcal{O}(K_{\Omega_j})\otimes\mathcal{I}(2(\alpha_j+1)G_{\Omega_j}(\cdot,z_j)))_{z_j}\}$. Note that $\psi=\max_{1\le j\le n}\{2p_j\pi_j^*(G_{\Omega_j}(\cdot,z_j))\}$. Thus, we have
\begin{equation}
	\label{eq:211003c}
	G(t;\tilde{c}\equiv1)=\sum_{\alpha\in E}(|d_{\alpha}|^2\Pi_{1\le j\le n}G_{j,\alpha_j}(\frac{(\alpha_j+1)t}{p_j})).
\end{equation}

 Note that $G_{j,\alpha_j}(0)\in(0,+\infty)$. It follows from Theorem \ref{thm:m-points} that there exists a positive constant $k_{j,\alpha_j}$ such that $G_{j,\alpha_j}(t)=k_{j,\alpha_j}e^{-t}\in(0,+\infty)$ for any $t\ge0.$ As $\sum_{1\le j\le n}\frac{\alpha_j+1}{p_j}=1$ for any $\alpha\in E$, then equality \eqref{eq:211003c} tells us that
 \begin{displaymath}
 	G(t;\tilde{c})=\sum_{\alpha\in E}(|d_{\alpha}|^2\Pi_{1\le j\le n}k_{j,\alpha_j}e^{-\frac{(\alpha_j+1)t}{p_j}})=e^{-t}\sum_{\alpha\in E}(|d_{\alpha}|^2\Pi_{1\le j\le n}k_{j,\alpha_j}),
 \end{displaymath}
which implies that $G(-\log r;\tilde{c})$ is linear with respect to $r.$

It follows from Corollary \ref{c:linear} that there exists a holomorphic $(n,0)$ form $F$ on $M$, such that $(F-f,z_0)\in(\mathcal{O}(K_M)\otimes\mathcal{I}(\psi))_{z_0}$ and
\begin{equation}
	\label{eq:1127e}G(t;\tilde{c})=\int_{\{\psi<-t\}}|F|^2e^{-\varphi}=e^{-t}G(0;\tilde{c}).
\end{equation}
Let $t\in[0,+\infty)$. Let $\tilde{F}\in H^0(\{\psi<-t\},\mathcal{O}(K_M))$ satisfying that $\int_{\{\psi<-t\}}|\tilde F|^2e^{-\varphi}c(-\psi)<+\infty$ and $(\tilde{F}-F,z_0)\in(\mathcal{O}(K_{M})\otimes\mathcal{I}(\psi))_{z_0}$. Following from $\varphi(z_0)>-\infty$, there exists  $t_1>t$ such that
 \begin{equation}
 	\label{eq:211003d}
 	\int_{\{\psi<-t_1\}}|\tilde{F}|^{2}e^{-\varphi}<+\infty.
 \end{equation}
 As $c(s)e^{-s}$ is decreasing with respect to $s$, it follows from inequality \eqref{eq:211003d} that
 \begin{displaymath}
 	\begin{split}
 		&\int_{\{\psi<-t\}}|\tilde{F}|^2e^{-\varphi}\\
 		\le&\int_{\{\psi<-t_1\}}|\tilde{F}|^2e^{-\varphi}+\int_{\{-t_1\le\psi<-t\}}|\tilde{F}|^2e^{-\varphi}\\
 		\le&\int_{\{\psi<-t_1\}}|\tilde{F}|^2e^{-\varphi}+\frac{1}{\inf_{s\in(t,t_1]}c(s)}\int_{\{-t_1\le\psi<-t\}}|\tilde{F}|^2e^{-\varphi}c(-\psi)\\
 		<&+\infty.
 	\end{split}
 \end{displaymath}
 Thus,  Corollary \ref{c:linear} shows that
 \begin{equation}
 	\label{eq:1127d}G(t;c)=\int_{\{\psi<-t\}}|F|^2e^{-\varphi}c(-\psi)=\frac{G(0;c)}{\int_{0}^{+\infty}c(s)e^{-s}ds}\int_{t}^{+\infty}c(s)e^{-s}ds,
 \end{equation}
 i.e. $G(h^{-1}(r);c)$ is linear with respect to $r$.

\subsection{Proof of the necessity part of Theorem \ref{thm:linear-2d}}\label{sec:n-1}
\

It follows from the linearity of $G(h^{-1}(r))$ and Corollary \ref{c:linear}, there exists a holomorphic $(n,0)$ form $F$ on $M$, such that $(F-f,z_0)\in(\mathcal{O}(K_M)\otimes\mathcal{I}(\psi))_{z_0}$ and
\begin{equation}
	\label{eq:211126a}G(t)=\int_{\{\psi<-t\}}|F|^2e^{-\varphi}c(-\psi).
\end{equation}

For any $t>0$, let $\tilde{F}\in H^0(\{\psi<-t\},\mathcal{O}(K_M))$ satisfying $\int_{\{\psi<-t\}}|\tilde F|^2e^{-\varphi}<+\infty$ and  $(F-\tilde{F},z_0)\in(\mathcal{O}(K_{M})\otimes\mathcal{I}(\psi))_{z_0}$. Following from the Strong openness property (see \cite{GZSOC}) and $\varphi(z_0)>-\infty$,   there exists $r>1$ and $t_1>t$ such that
 \begin{equation}
 	\label{eq:211003a}
 	\int_{\{\psi<-t_1\}}|F-\tilde{F}|^{2r}e^{-r\psi}<+\infty
 \end{equation}
 and
 \begin{equation}
 	\label{eq:211003b}
 	\int_{\{\psi<-t_1\}}e^{-\frac{r}{r-1}\varphi}<+\infty.
 \end{equation}
 As $c(s)e^{-s}$ is decreasing with respect to $s$, it follows from inequality \eqref{eq:211003a} and inequality \eqref{eq:211003b} that
 \begin{displaymath}
 	\begin{split}
 		&\int_{\{\psi<-t\}}|\tilde{F}|^2e^{-\varphi}c(-\psi)\\
 		\le&\int_{\{\psi<-t_1\}}|\tilde{F}|^2e^{-\varphi}c(-\psi)+\int_{\{-t_1\le\psi<-t\}}|\tilde{F}|^2e^{-\varphi}c(-\psi)\\
 		\le&C\int_{\{\psi<-t_1\}}|F-\tilde{F}|^2e^{-\varphi-\psi}+2\int_{\{\psi<-t_1\}}|F|^2e^{-\varphi}c(-\psi)\\
 		&+(\sup_{s\in(t,t_1]}c(s))\int_{\{-t_1\le\psi<-t\}}|\tilde{F}|^2e^{-\varphi}\\
 		\le&C(\int_{\{\psi<-t_1\}}|F-\tilde{F}|^{2r}e^{-r\psi})^{\frac{1}{r}}(\int_{\{\psi<-t_1\}}e^{-\frac{r}{r-1}\varphi})^{1-\frac{1}{r}}\\
 		&+2\int_{\{\psi<-t_1\}}|F|^2e^{-\varphi}c(-\psi)+(\sup_{s\in(t,t_1]}c(s))\int_{\{-t_1\le\psi<-t\}}|\tilde{F}|^2e^{-\varphi}\\
 		<&+\infty,
 	\end{split}
 \end{displaymath}
 where $C$ is a constant.  Thus, it follows from Corollary \ref{c:linear} and equality \eqref{eq:211126a} that
 \begin{equation}
 	\label{eq:1128a}G(t;\tilde{c}\equiv1)=\int_{\{\psi<-t\}}|F|^2e^{-\varphi}=e^{-t}\frac{G(0;c)}{\int_0^{+\infty}c(s)e^{-s}ds}
 \end{equation}
 for any $t>0$.
Theorem \ref{thm:general_concave} shows that $G(0,\tilde{c})=\lim_{t\rightarrow0+0}G(t;\tilde{c})$, then we have $G(-\log r,\tilde{c})$ is linear with respect to $r\in(0,1].$

It follows from Lemma \ref{l:0} that  $f=(\sum_{\alpha\in E_1}d_{\alpha}w^{\alpha}+g_0)dw_{1}\wedge dw_{2}\wedge...\wedge dw_{n}$,  where $E_1=\{\alpha\in\mathbb{Z}_{\ge0}^n:\sum_{1\le j\le n}\frac{\alpha_j+1}{p_j}\le1\}$, $d_{\alpha}\in\mathbb{C}$ and $g_0$ is a holomorphic function on $V_0$ such that $(g_0,z_0)\in\mathcal{I}(\psi)_{z_0}$.
Denote that $G_{\alpha}(t):=\inf\{\int_{\{\psi<-t\}}|\tilde{f}|^2e^{-\varphi}:\tilde{f}$ is a holomorphic $(n,0)$ form on $M$ such that$(\tilde{f}-w^{\alpha}dw_{1}\wedge dw_{2}\wedge...\wedge dw_{n},z_0)\in(\mathcal{O}(K_M)\otimes\mathcal{I}(\psi))_{z_0}\}$, where $\alpha\in E_1$. Denote that $\tilde{E}_1=\{\alpha\in E_1:d_{\alpha}\not=0\}$. As $G(0)>0$, we have $\tilde E_1\not=\emptyset$.

Firstly, we prove  $\tilde{E}_1\subset E$ by contradiction: if not, then $s_0=\inf\{\sum_{1\le j\le n}\frac{\alpha_j+1}{p_j}:\alpha\in \tilde{E}_1\}<1$. Denote that $E_2:=\{\alpha\in\tilde{E}_1:\sum_{1\le j\le n}\frac{\alpha_j+1}{p_j}=s_0\}$  and
\begin{displaymath}
	\begin{split}
		G_2(t):=\inf\{&\int_{\{\psi<-t\}}|\tilde{f}|^2e^{-\varphi}:\tilde{f}\in H^0(\{\psi<-t\},\mathcal{O}(K_M))\\
		&\&\,(\tilde{f}-\sum_{\alpha\in E_2}d_{\alpha}w^{\alpha}dw_{1}\wedge dw_{2}\wedge...\wedge dw_{n},z_0)\in(\mathcal{O}(K_M)\otimes\mathcal{I}(s_0\psi))_{z_0}\}.
	\end{split}
\end{displaymath}
It follows from Lemma \ref{l:0} that
$$(\sum_{\alpha\in E_1\backslash E_2}w^{\alpha}dw_{1}\wedge dw_{2}\wedge...\wedge dw_{n},z_0)\in(\mathcal{O}(K_M)\otimes\mathcal{I}(s_0\psi))_{z_0}.$$
If $\tilde{f}$ is a holomorphic $(n,0)$ form on $M$ such that$(\tilde{f}-f,z_0)\in(\mathcal{O}(K_M)\otimes\mathcal{I}(\psi))_{z_0}$,  we have $(\tilde{f}-\sum_{\alpha\in E_2}d_{\alpha}w^{\alpha}dw_{1}\wedge dw_{2}\wedge...\wedge dw_{n},z_0)=(\tilde{f}-f+\sum_{\alpha\in E_1\backslash E_2}d_{\alpha}w^{\alpha}dw_{1}\wedge dw_{2}\wedge...\wedge dw_{n}+g_0,z_0)\in(\mathcal{O}(K_M)\otimes\mathcal{I}(s_0\psi))_{z_0}.$ Thus, we get
$$G(t;\tilde{c}\equiv1)\ge G_2(t)$$
 for any $t\ge0$.
As $(\sum_{\alpha\in E_2}w^{\alpha}dw_{1}\wedge dw_{2}\wedge...\wedge dw_{n},z_0)\not\in(\mathcal{O}(K_M)\otimes\mathcal{I}(s_0\psi))_{z_0}$, then $G_{2}(t)>0$ for any $t\ge0.$ Theorem \ref{thm:general_concave} shows that $G_{2}(-\frac{\log r}{s_0})$ is concave with respect to $r\in(0,1]$,  which implies that $\lim_{r\rightarrow0+0}\frac{G_{2}(-\frac{\log r}{s_0})}{r}>0$. Thus, we have
\begin{displaymath}
	\begin{split}
		\lim_{r\rightarrow0+0}\frac{G(-\log r;\tilde{c})}{r}&\ge\lim_{r\rightarrow0+0}\frac{G_2(-\log r)}{r}\\
		&=\lim_{r\rightarrow0+0}\frac{G_2(-\frac{\log r^{s_0}}{s_0})}{r^{s_0}}r^{s_0-1}\\
		&=+\infty,
	\end{split}
\end{displaymath}
which contradicts to the linearity of $G(-\log r;\tilde{c})$. Thus, we get $\tilde{E}_1\subset{E}\not=\emptyset$ and $f=(\sum_{\alpha\in E}d_{\alpha}w^{\alpha}+g_0)dw_{1}\wedge dw_{2}\wedge...\wedge dw_{n}$.

Now, we prove the statements $(2)$ and $(3)$. Denote that
\begin{displaymath}
	\begin{split}
	G_{j,\alpha_j}(t):=\inf\{\int_{\{2(\alpha_j+1)G_{\Omega_j}(\cdot,z_j)<-t\}}&|\tilde{f}|^2e^{-\varphi_j}:\tilde{f}\in H^0(\{2(\alpha_j+1)G_{\Omega_j}(\cdot,z_j)<-t\},\mathcal{O}(K_{\Omega_j}))
	\\ & \&\,(\tilde{f}-w_j^{\alpha_j},z_j)\in(\mathcal{O}(K_{\Omega_j})\otimes\mathcal{I}(2(\alpha_j+1)G_{\Omega_j}(\cdot,z_j)))_{z_j}\}.	
	\end{split}
\end{displaymath}
 Note that $\psi=\max_{1\le j\le n}\{2p_j\pi_j^*(G_{\Omega_j}(\cdot,z_j))\}$ and $\sum_{1\le j\le n}\frac{\alpha_j+1}{p_j}=1$ for any $\alpha\in E$. Following from Lemma \ref{l:orth1} that
\begin{equation}
	\label{eq:211003e}
	G(t;\tilde{c}\equiv1)=\sum_{\alpha\in E}(|d_{\alpha}|^2\Pi_{1\le j\le n}G_{j,\alpha_j}(\frac{(\alpha_j+1)t}{p_j})).
\end{equation}
Theorem \ref{c:L2-1d-char} shows that
\begin{equation}\label{eq:211003f}
	G_{j,\alpha_j}(t)\le e^{-t}\frac{2\pi e^{-\varphi_j(z_j)}}{(\alpha_j+1)c_j(z_j)^{2\alpha_j+2}}.
\end{equation}
Combining equality \eqref{eq:211003e} and inequality \eqref{eq:211003f} and $\sum_{1\le j\le n}\frac{\alpha_j+1}{p_j}=1$ for any $\alpha\in E$, we have
\begin{equation}
	\label{eq:211003g}
	G(t;\tilde{c})\leq e^{-t}\sum_{\alpha\in E}(|d_{\alpha}|^2(2\pi)^ne^{-\varphi(z_0)}\Pi_{1\le j\le n}\frac{1}{(\alpha_j+1)c_j(z_j)^{2\alpha_j+2}}).
\end{equation}

As $M$ is a Stein manifold and $\varphi$ is plurisubharmonic function on $M$, there exist smooth plurisubharmonic functions $\Phi_l$ on $M$, which are decreasingly convergent to $\varphi$ with respect to $l$.  It follows from Lemma \ref{l:green-sup} and Lemma \ref{l:green-sup2} that there exists a local coordinate $\tilde{w}_{j}$ on a neighborhood $\tilde{V}_{z_{j}}\Subset V_{z_{j}}$ of $z_{j}$ satisfying $\tilde{w}_{j}(z_{j})=0$ and $|\tilde{w}_{j}|=e^{G_{\Omega_j}(\cdot,z_{j})}$ on $\tilde{V}_{z_{j}}$. Denote that $\tilde{V}_{0}:=\Pi_{1\le j\le n}\tilde V_{z_j}$. It follows from Lemma \ref{l:G-compact} that there exists $t_2>0$ such that $\{\psi<-t_2\}\Subset\tilde{V}_0$.  Note that
\begin{equation}
	\label{eq:1128b}\begin{split}G(t;\tilde{c})&=\int_{\{\psi<-t\}}|F|^2e^{-\varphi}	\\
	&\ge \int_{\{\psi<-t\}}|F|^2e^{-\Phi_l}.
\end{split}\end{equation}
As $(F-f,z_0)\in (\mathcal{O}(K_{M})\otimes\mathcal{I}(\psi))_{z_{0}}$, it follows from Lemma \ref{l:0} that $$F=(\sum_{\alpha\in E}\tilde{d}_{\alpha}\tilde w^{\alpha}+\tilde{g}_0)d\tilde{w}_1\wedge d\tilde{w}_2\wedge...\wedge d\tilde{w}_n,$$
where $|\frac{d_{\alpha}}{\tilde{d}_{\alpha}}|=c_j(z_j)^{\alpha_j+1}$ for any $\alpha\in E$ and $(\tilde{g}_0,z_0)\in\mathcal{I}(\psi)_{z_0}$. It follows from inequality \eqref{eq:1128b} and Lemma \ref{l:m2} that
\begin{equation}
	\label{eq:1128c}
	\liminf_{t\rightarrow+\infty}e^tG(t;\tilde{c})\ge \sum_{\alpha\in E}\frac{|d_{\alpha}|^2(2\pi)^ne^{-\Phi_l(z_0)}}{\Pi_{1\le j\le n}(\alpha_j+1)c_j(z_j)^{2\alpha_j+2}}.
\end{equation}
Letting $l\rightarrow+\infty$, inequality \eqref{eq:1128c} shows that
\begin{equation}
	\label{eq:211003h}
	\liminf_{t\rightarrow+\infty}e^tG(t;\tilde{c})\ge\sum_{\alpha\in E}\frac{|d_{\alpha}|^2(2\pi)^ne^{-\varphi(z_0)}}{\Pi_{1\le j\le n}(\alpha_j+1)c_j(z_j)^{2\alpha_j+2}}.
\end{equation}
As $G(-\log r;\tilde{c})$ is linear with respect to $r$, it follows from inequality \eqref{eq:211003g} and equality \eqref{eq:211003h} that
\begin{equation}
	\label{eq:211110a}G(t;\tilde{c})= e^{-t}\sum_{\alpha\in E}\frac{|d_{\alpha}|^2(2\pi)^ne^{-\varphi(z_0)}}{\Pi_{1\le j\le n}(\alpha_j+1)c_j(z_j)^{2\alpha_j+2}}.
\end{equation}
 Combining equality \eqref{eq:211003e}, inequality \eqref{eq:211003f} and equality \eqref{eq:211110a}, we have
 \begin{equation}
 	\label{eq:211110b}G_{j,\alpha_j}(t)= e^{-t}\frac{2\pi e^{-\varphi_j(z_j)}}{(\alpha_j+1)c_j(z_j)^{2\alpha_j+2}}
 \end{equation}
 holds for any $j$ and any $\alpha$ satisfying $d_{\alpha}\not=0$. It follows from Theorem \ref{thm:m-points} that $\varphi_j=2\log|g_j|+2u_j$ and $\chi_{j,z_j}^{\alpha_j+1}=\chi_{j,-u_j}$ for any $j\in\{1,2,...,n\}$ and $\alpha\in E$ satisfying $d_{\alpha}\not=0$, where $g_j$ is a holomorphic function on $\Omega_j$ such that $g_j(z_j)\not=0$ and $u_j$ is a harmonic function on $\Omega_j$. 
 
  Thus, we complete the proof of the necessity.

 \subsection{Proof of Remark \ref{r:1.1}}
 \

 \label{sec:pf-r1.1}Following from Corollary \ref{c:linear}, we have the uniqueness, thus it suffices to prove the existence.
 It follows from Section \ref{sec:n-1} that
 \begin{equation*}
G_{j,\alpha_j}(t)= e^{-t}\frac{2\pi e^{-\varphi_j(z_j)}}{(\alpha_j+1)c_j(z_j)^{2\alpha_j+2}}
 \end{equation*}
 holds for any $t\ge0$, any $j\in\{1,2,...,n\}$ and any $\alpha\in E$ satisfying $d_{\alpha}\not=0$. Remark \ref{rem:1.2} shows that
 $$G_{j,\alpha_j}(0)=\int_{\Omega_j}|a_{j,\alpha_j}g_j(P_j)_*(f_{u_j}(f_{z_j})^{\alpha_j}df_{z_j})|^2e^{-\varphi_j},$$
  where $a_{j,\alpha_j}=\lim_{z\rightarrow z_j}\frac{w_j^{\alpha_j}dw_j}{g_j(P_j)_*(f_{u_j}(f_{z_j})^{\alpha_j}df_{z_j})}$. It follows from Lemma \ref{l:orth1} that
 \begin{displaymath}\begin{split}
 	G(0;\tilde{c}\equiv1 )&=\int_{M}|\sum_{\alpha\in E}d_{\alpha}\wedge_{1\le j\le n}\pi_j^*(a_{j,\alpha_j}g_j(P_j)_*(f_{u_j}(f_{z_j})^{\alpha_j}df_{z_j}))|^2e^{-\varphi}\\
 	&=\sum_{\alpha\in E}(|d_{\alpha}|^2\Pi_{1\le j\le n}G_{j,\alpha_j}(0))\\
 	&=\sum_{\alpha\in E}(|d_{\alpha}|^2(2\pi)^ne^{-\varphi(z_0)}\Pi_{1\le j\le n}\frac{1}{(\alpha_j+1)c_j(z_j)^{2\alpha_j+2}}). 	
 \end{split} \end{displaymath}
Note that $\tilde{d}_{\alpha}=\lim_{z\rightarrow z_0}\frac{d_{\alpha}w^{\alpha}dw_1\wedge dw_2\wedge...\wedge dw_n}{\wedge_{1\le j\le n}\pi_j^*(g_j(P_j)_*(f_{u_j}(f_{z_j})^{\alpha_j}df_{z_j}))}=d_{\alpha}\Pi_{1\le j\le n}a_{j,\alpha_j}$. As $G(-\log r;\tilde{c})$ is linear with respect to $r$, Corollary \ref{c:linear} tells us that
\begin{displaymath}\begin{split}
	 	G(t;\tilde{c}\equiv1 )&=\int_{\{\psi<-t\}}|\sum_{\alpha\in E}d_{\alpha}\wedge_{1\le j\le n}\pi_j^*(a_{j,\alpha_j}g_j(P_j)_*(f_{u_j}(f_{z_j})^{\alpha_j}df_{z_j}))|^2e^{-\varphi}\\
 	&=e^{-t}G(0;\tilde{c})\\
 	&=e^{-t}\sum_{\alpha\in E}(|d_{\alpha}|^2(2\pi)^ne^{-\varphi(z_0)}\Pi_{1\le j\le n}\frac{1}{(\alpha_j+1)c_j(z_j)^{2\alpha_j+2}}).
\end{split}\end{displaymath}
It follows from equality \eqref{eq:1127e} and equality \eqref{eq:1127d}, we get that $\sum_{\alpha\in E}\tilde{d}_{\alpha}\wedge_{1\le j\le n}\pi_j^*(g_j(P_j)_*(f_{u_j}(f_{z_j})^{\alpha_j}df_{z_j}))$ is the $(n,0)$ form satisfying the conditions in Remark \ref{r:1.1}.

\section{Proofs of Theorem \ref{thm:prod-finite-point} and Remark \ref{r:1.2}}
In this section, we prove Theorem \ref{thm:prod-finite-point} and Remark \ref{r:1.2}.

\subsection{Proof of the sufficiency part of Theorem \ref{thm:prod-finite-point}}\label{sec:s-2}
\

Denote that
 $$F=c_0\wedge_{1\le j\le n}\pi_{j}^*(g_j(P_{j})_*(f_{u_j}(\Pi_{1\le k\le m_j}f_{z_{j,k}}^{\gamma_{j,k}+1})(\sum_{1\le k\le m_j}p_{j,k}\frac{df_{z_{j,k}}}{f_{z_{j,k}}})))$$
 is a holomorphic $(n,0)$ form on $M$. As
$$\lim_{z\rightarrow z_{\beta}}\frac{c_{\beta}\Pi_{1\le j\le n}w_{j,\beta_j}^{\gamma_{j,\beta_j}}dw_{1,\beta_1}\wedge dw_{2,\beta_2}\wedge...\wedge dw_{n,\beta_n}}{\wedge_{1\le j\le n}\pi_{j}^*(g_j(P_{j})_*(f_{u_j}(\Pi_{1\le k\le m_j}f_{z_{j,k}}^{\gamma_{j,k}+1})(\sum_{1\le k\le m_j}p_{j,k}\frac{df_{z_{j,k}}}{f_{z_{j,k}}})))}=c_0,$$
 $f=(c_{\beta}\Pi_{1\le j\le n}w_{j,\beta_j}^{\gamma_{j,\beta_j}}+g_{\beta})dw_{1,\beta_1}\wedge dw_{2,\beta_2}\wedge...\wedge dw_{n,\beta_n}$ on $V_{\beta}$ and $\sum_{1\le j\le n}\frac{\gamma_{j,\beta_j}+1}{p_{j,\beta_j}}=1$ for any $\beta\in I_1$, it follows from Lemma \ref{l:0} that  $(F-f,z_{\beta})\in(\mathcal{O}(K_{M})\otimes\mathcal{I}(\psi))_{z_{\beta}}$ for any $\beta\in I_1$.
Denote that
$$f_j:=g_j(P_{j})_*(f_{u_j}(\Pi_{1\le k\le m_j}f_{z_{j,k}}^{\gamma_{j,k}+1})(\sum_{1\le k\le m_j}p_{j,k}\frac{df_{z_{j,k}}}{f_{z_{j,k}}})),$$
$$r_{j,k}:=\varphi_j(z_{j,k})+\lim_{z\rightarrow z_{j,k}}(2\sum_{1\le k'\le m_j}(\gamma_{j,k'}+1)G_{\Omega_j}(z,z_{j,k'})-2(\gamma_{j,k}+1)\log|w_{j,k}(z)|)$$
and
$$a_{j,k}:=\lim_{z\rightarrow z_{j,k}}\frac{f_j}{w_{j,k}^{\gamma_{j,k}}dw_{j,k}}.$$
 Note that
 \begin{displaymath}
 	\begin{split}
 		&\varphi_j+2\sum_{1\le k'\le m_j}(\gamma_{j,k'}+1)G_{\Omega_j}(\cdot,z_{j,k'})-2(\gamma_{j,k}+1)\log|w_{j,k}|\\
 		=&\varphi_j+2\sum_{1\le k'\le m_j}(\gamma_{j,k'}+1)G_{\Omega_j}(\cdot,z_{j,k'})-2(\gamma_{j,k}+1)G_{\Omega_j}(\cdot,z_{j,k})\\
 		&+(2(\gamma_{j,k}+1)G_{\Omega_j}(\cdot,z_{j,k})-2(\gamma_{j,k}+1)\log|w_{j,k}|).
 	\end{split}
 \end{displaymath}
It follows from Theorem \ref{c:L2-1d-char} and Remark \ref{rem:1.2} that
$\int_{\Omega_j}|f_j|^2e^{-\varphi_j}=\inf\{\int_{\Omega_j}|\tilde f|^2e^{-\varphi_j}:\tilde f$ is a holomorphic $(1,0)$ form on $\Omega_j$ such that $(\tilde f-f_j,z_{j,k})\in(\mathcal{O}(K_{\Omega_j})\otimes\mathcal{I}(2(\gamma_{j,k}+1)G_{\Omega_j}(\cdot,z_{j,k})))_{z,k}$ for any $k\in\{1,2,...,m_j\}$$\}$ and
\begin{equation}\label{eq:1125i}
	\int_{\Omega_j}|f_j|^2e^{-\varphi_j}=\sum_{1\le k\le m_j}\frac{2\pi|a_{j,k}|^2e^{-r_{j,k}}}{\gamma_{j,k}+1},
\end{equation}
which implies that
\begin{equation}
	\label{eq:1125j}\int_{\Omega_j}f_j\wedge\overline{F_1}e^{-\varphi_j}=0
\end{equation}
for any $F_1\in\{\tilde f\in H^0(\Omega_j,\mathcal{O}(K_{\Omega_j})):\int_{\Omega_j}|\tilde f|^2e^{-\varphi_j}<+\infty\,
	\&\,(\tilde f,z_{j,k})\in(\mathcal{O}(K_{\Omega_j})\otimes\mathcal{I}(2(\gamma_{j,k}+1)G_{\Omega_j}(\cdot,z_{j,k})))_{z_{j,k}}$ for any $k\in\{1,2,...,m_j\}\}$. It follows from Lemma \ref{l:orth2} and equality \eqref{eq:1125j} that
	\begin{equation}
		\label{eq:1125k}\int_{M}F\wedge\overline{F_1}e^{-\varphi}=0
	\end{equation}
	for any $F_1\in\{\tilde f\in H^0(M,\mathcal{O}(K_{M})):\int_{M}|\tilde f|^2e^{-\varphi}<+\infty
		\,\&\,(\tilde f,z_{\beta})\in(\mathcal{O}(K_{M})\otimes\mathcal{I}(\psi))_{z_{\beta}}$ for any $\beta \in I_1\}$.
	Equality \eqref{eq:1125i} shows that
\begin{equation}
	\label{eq:1122h}\int_{M}|F|^2e^{-\varphi}=|c_0|^2\Pi_{1\le j\le n}(\sum_{1\le k\le m_j}\frac{2\pi|a_{j,k}|^2e^{-r_{j,k}}}{\gamma_{j,k}+1}).
\end{equation}
As $(F-f,z_{\beta})\in(\mathcal{O}(K_{M})\otimes\mathcal{I}(\psi))_{z_{\beta}}$ for any $\beta\in I_1$, it follows from equality \eqref{eq:1125k} and equality \eqref{eq:1122h} that
\begin{equation}
	\label{eq:1125l}\begin{split}
	G(0;\tilde{c}\equiv1)&=\int_{M}|F|^2e^{-\varphi}\\
	&=|c_0|^2\Pi_{1\le j\le n}(\sum_{1\le k\le m_j}\frac{2\pi|a_{j,k}|^2e^{-r_{j,k}}}{\gamma_{j,k}+1})\\
	&=|c_0|^2(2\pi)^n \sum_{\beta\in I_1} \Pi_{1\le j\le n}\frac{|a_{j,\beta_j}|^2e^{-r_{j,\beta_j}}}{\gamma_{j,\beta_j}+1}.
	\end{split}
\end{equation}
As $\sum_{1\le j\le n}\frac{\gamma_{j,\beta_j}+1}{p_{j,\beta_j}}=1$ for any $\beta\in I_1$, we have
$$\frac{\gamma_{j,1}+1}{p_{j,1}}=\frac{\gamma_{j,k}+1}{p_{j,k}}$$
for any $j\in\{1,2,...,n\}$ and $k\in\{1,2,...,m_j\}$.
 Note that
 $$(F-(c_0\Pi_{1\le j\le n}a_{j,\beta_j}w_{j,\beta_j}^{\gamma_{j,\beta_j}})dw_{1,\beta_1}\wedge dw_{2,\beta_2}\wedge...\wedge dw_{n,\beta_n},z_{\beta})\in(\mathcal{O}(K_M)\otimes\mathcal{I}(\psi))_{z_{\beta}}$$ for any $\beta\in I_1$ and
\begin{displaymath}
	\begin{split}
		&e^{\varphi(z_{\beta})}\Pi_{1\le j\le n}(\exp\lim_{z\rightarrow z_{j,\beta_j}}(\frac{\sum_{1\le k\le m_j}p_{j,k}G_{\Omega_j}(z,z_{j,k})+\frac{t}{2}}{p_{j,\beta_j}}-\log|w_{j,\beta_j}|))^{2\gamma_{j,\beta_j}+2}\\
		=&e^{t}\Pi_{1\le j\le n}\exp(\varphi(z_{j,\beta_j})\\
		&+\lim_{z\rightarrow z_{j,\beta_j}}(\sum_{1\le k\le m_j}2(\gamma_{j,k}+1)G_{\Omega_j}(z,z_{j,k})-2(\gamma_{j,\beta_j}+1)\log|w_{j,\beta_j}|))\\
		=&e^{t}\Pi_{1\le j\le n}e^{r_{j,\beta_j}}
	\end{split}
\end{displaymath}for any $\beta\in I_1$ and $t\ge0$.
 It follows from Proposition \ref{p:exten-pro-finite} ($M\sim \{\psi<-t\}$ and $\psi\sim\max_{1\le j\le n}\{\pi_j^*(2\sum_{1\le k\le m_j}p_{j,k}G_{\Omega_j}(z,z_{j,k})+t)\}$) that
 \begin{equation}
 	\label{eq:1126f}
 	G(t;\tilde{c})\le e^{-t}|c_0|^2(2\pi)^n \sum_{\beta\in I_1} \Pi_{1\le j\le n}\frac{|a_{j,\beta_j}|^2e^{-r_{j,\beta_j}}}{\gamma_{j,\beta_j}+1}
 \end{equation}
holds for any $t\ge0$. Theorem \ref{thm:general_concave} shows that $G(-\log r;\tilde{c})$ is concave with respect to $r$. Combining equality \eqref{eq:1125l} and inequality \eqref{eq:1126f}, we obtain that
\begin{equation*}
	G(t;\tilde{c})= e^{-t}|c_0|^2(2\pi)^n \sum_{\beta\in I_1} \Pi_{1\le j\le n}\frac{|a_{j,\beta_j}|^2e^{-r_{j,\beta_j}}}{\gamma_{j,\beta_j}+1}
\end{equation*}
for any $t\ge 0$, i.e. $G(-\log r;\tilde{c})$ is linear with respect to $r$.

It follows from Corollary \ref{c:linear} that
$$G(t;\tilde{c})=e^{-t}\int_{\{\psi<-t\}}|F|^2e^{-\varphi}.$$
 For any $t\ge0$. For any $t\ge0$ and any holomorphic $(n,0)$ form $\tilde{F}$ on  $\{\psi<-t\}$ satisfying $\int_{\{\psi<-t\}}|\tilde{F}|^2e^{-\varphi}c(-\psi)<+\infty$, following from $\varphi(z_{\beta})>-\infty$ for any $\beta\in I_1$,   there exists  $t_1>t$ such that
 \begin{equation}
 	\label{eq:1126g}
 	\int_{\{\psi<-t_1\}}|\tilde{F}|^{2}e^{-\varphi}<+\infty.
 \end{equation}
 As $c(s)e^{-s}$ is decreasing with respect to $s$, it follows from inequality \eqref{eq:1126g} that
 \begin{displaymath}
 	\begin{split}
 		&\int_{\{\psi<-t\}}|\tilde{F}|^2e^{-\varphi}\\
 		\le&\int_{\{\psi<-t_1\}}|\tilde{F}|^2e^{-\varphi}+\int_{\{-t_1\le\psi<-t\}}|\tilde{F}|^2e^{-\varphi}\\
 		\le&\int_{\{\psi<-t_1\}}|\tilde{F}|^2e^{-\varphi}+\frac{1}{\inf_{s\in(t,t_1]}c(s)}\int_{\{-t_1\le\psi<-t\}}|\tilde{F}|^2e^{-\varphi}c(-\psi)\\
 		<&+\infty.
 	\end{split}
 \end{displaymath}
 Thus,  Corollary \ref{c:linear} shows that $G(h^{-1}(r),c)$ is linear with respect to $r$.

Thus, the sufficiency part of Theorem \ref{thm:prod-finite-point} holds.

\subsection{Proof of the necessity part of Theorem \ref{thm:prod-finite-point}}\label{sec:n-2}
\

We prove the necessity part of Theorem \ref{thm:prod-finite-point} in four steps.

\

\emph{Step 1. $G(-\log r;\tilde{c}\equiv1 )$ is linear.}

\

Following from the linearity of $G(h^{-1}(r))$ and Corollary \ref{c:linear}, there exists a holomorphic $(n,0)$ form $F$ on $M$, such that $(F-f,z_{\beta})\in(\mathcal{O}(K_M)\otimes\mathcal{I}(\psi))_{z_{\beta}}$ for any $\beta\in I_1$ and
\begin{equation}
	\label{eq:1121a}G(t)=\int_{\{\psi<-t\}}|F|^2e^{-\varphi}c(-\psi).
\end{equation}
Let $t>0$, and let $\tilde{F}$ be a holomorphic $(n,0)$ form on $\{\psi<-t\}$ satisfying $\int_{\{\psi<-t\}}|\tilde{F}|^2e^{-\varphi}<+\infty$ and $(\tilde{F}-f,z_{\beta})\in(\mathcal{O}(K_{M})\otimes\mathcal{I}(\psi))_{z_{\beta}}$ for any $\beta\in I_1$. Note that $(F-\tilde{F},z_{\beta})\in(\mathcal{O}(K_{M})\otimes\mathcal{I}(\psi))_{z_{\beta}}$. Following from the Strong openness property (see \cite{GZSOC}) and $\varphi(p_0)>-\infty$,   there exists $r>1$ and $t_1>t$ such that
 \begin{equation}
 	\label{eq:211120a}
 	\int_{\{\psi<-t_1\}}|F-\tilde{F}|^{2r}e^{-r\psi}<+\infty
 \end{equation}
 and
 \begin{equation}
 	\label{eq:211120b}
 	\int_{\{\psi<-t_1\}}e^{-\frac{r}{r-1}\varphi}<+\infty.
 \end{equation}
 As $c(s)e^{-s}$ is decreasing with respect to $s$, it follows from inequality \eqref{eq:211120a} and inequality \eqref{eq:211120b} that
 \begin{displaymath}
 	\begin{split}
 		&\int_{\{\psi<-t\}}|\tilde{F}|^2e^{-\varphi}c(-\psi)\\
 		\le&\int_{\{\psi<-t_1\}}|\tilde{F}|^2e^{-\varphi}c(-\psi)+\int_{\{-t_1\le\psi<-t\}}|\tilde{F}|^2e^{-\varphi}c(-\psi)\\
 		\le&C\int_{\{\psi<-t_1\}}|F-\tilde{F}|^2e^{-\varphi-\psi}+2\int_{\{\psi<-t_1\}}|F|^2e^{-\varphi}c(-\psi)\\
 		&+(\sup_{s\in(t,t_1]}c(s))\int_{\{-t_1\le\psi<-t\}}|\tilde{F}|^2e^{-\varphi}\\
 		\le&C(\int_{\{\psi<-t_1\}}|F-\tilde{F}|^{2r}e^{-r\psi})^{\frac{1}{r}}(\int_{\{\psi<-t_1\}}e^{-\frac{r}{r-1}\varphi})^{1-\frac{1}{r}}\\
 		&+2\int_{\{\psi<-t_1\}}|F|^2e^{-\varphi}c(-\psi)+(\sup_{s\in(t,t_1]}c(s))\int_{\{-t_1\le\psi<-t\}}|\tilde{F}|^2e^{-\varphi}\\
 		<&+\infty,
 	\end{split}
 \end{displaymath}
 where $C$ is a constant.
 Thus, it follows from Corollary \ref{c:linear} and equality \eqref{eq:1121a} that
 $$G(t;\tilde{c}\equiv1)=\int_{\{\psi<-t\}}|F|^2e^{-\varphi}=e^{-t}\frac{G(0)}{\int_0^{+\infty}c(s)e^{-s}ds}.$$
Theorem \ref{thm:general_concave} shows that $G(0,\tilde{c})=\lim_{t\rightarrow0+0}G(t;\tilde{c})$. Thus we have $G(-\log r,\tilde{c})$ is linear with respect to $r\in(0,1].$

\

\emph{Step 2. $\sum_{1\le j\le n}\frac{\gamma_{j,\beta_j}+1}{p_{j,\beta_j}}=1$ and  $f=(c_{\beta}\Pi_{1\le j\le n}w_{j,\beta_j}^{\gamma_{j,\beta_j}}+g_{\beta})dw_{1,\beta_1}\wedge dw_{2,\beta_2}\wedge...\wedge dw_{n,\beta_n}$.}

\

It follows from Lemma \ref{l:green-sup2} and Lemma \ref{l:G-compact} that there exists $t_0>0$ such that $\{\psi<-t_0\}\Subset\cup_{\beta\in I_1}V_{\beta}$ and $\{z\in\Omega_j:2\sum_{1\le k\le m_j}p_{j,k}G_{\Omega_j}(\cdot,z_{j,k})<-t_0\}\cap V_{z_{j,k}}$ is simply connected for any $j\in\{1,2,...,n\}$ and $k\in\{1,2,...,m_j\}$. For any $\beta\in I_1$,
 denote
\begin{equation*}
\begin{split}
\inf\{\int_{\{\psi<-t\}\cap V_{\beta}}|\tilde{f}|^{2}e^{-\varphi}:(\tilde{f}-f,z_{\beta})\in&(\mathcal{O}(K_M)\otimes\mathcal{I}(\psi))_{z_{\beta}}\\&\&{\,}\tilde{f}\in H^{0}(\{\psi<-t\}\cap V_{\beta},\mathcal{O}(K_{M}))\},
\end{split}
\end{equation*}
by $G_{\beta}(t)$,  where $t\in[t_0,+\infty)$. Note that $\{\psi<-t\}=\cup_{\beta\in I_1}(\{\psi<-t\}\cap V_{\beta})$ for any $t\ge t_0$ and $\{\psi<-t\}\cap V_{\beta}=\Pi_{1\le j\le n}(\{z\in\Omega_j:2\sum_{1\le k\le m_j}p_{j,k}G_{\Omega_j}(\cdot,z_{j,k})<-t\}\cap V_{z_{j,\beta_j}})$ for any $t\ge t_0$ and $\beta\in I_1$.  By the definition of $G(t;\tilde{c})$ and $G_{\beta}(t)$, we have $G(t;\tilde{c})=\sum_{\beta\in I_1}G_{\beta}(t)$ for $t\ge t_0$. Thus, we have
$$G_{\beta}(t)=\int_{\{\psi<-t\}\cap V_{\beta}}|F|^2e^{-\varphi}$$
for any $t\ge t_0$.
 Theorem \ref{thm:general_concave} tells us that $G_{\beta}(-\log r)$ is concave with respect to $r\in(0,e^{-t_0})$. As $G(-\log r;\tilde{c})$ is linear with respect to $r$, we have $G_{\beta}(-\log r)$ linear with respect to $r\in(0,e^{-t_0}]$.

Let $\beta^*=(1,1,...,1)\in I_1$. As $f=w_{\beta^*}^{\alpha_{\beta_*}}dw_{1,1}\wedge dw_{2,1}\wedge...\wedge dw_{n,1}$ on $V_{\beta^*}$ and $\frac{1}{2p_{j,1}}(2\sum_{1\le k\le m_j}p_{j,k}G_{\Omega_j}(\cdot,z_{j,k})+t_0)$ is the Green function on $\{z\in\Omega_j:2\sum_{1\le k\le m_j}p_{j,k}G_{\Omega_j}(\cdot,z_{j,k})<-t_0\}\cap V_{z_{j,1}}$. Using Theorem \ref{thm:linear-2d} and Remark \ref{r:1.1}, we obtain that there exists a holomorphic $(1,0)$ form $h_{j,0}$ on $\{z\in\Omega_j:2\sum_{1\le k\le m_j}p_{j,k}G_{\Omega_j}(\cdot,z_{j,k})<-t_0\}\cap V_{z_{j,1}}$ for any $j\in\{1,2,...,n\}$ such that
$$F=\wedge_{1\le j\le n}\pi_{j}^*(h_{j,0})$$
 on $\{\psi<-t_0\}\cap V_{\beta^*}$.  It follows from Lemma \ref{l:decom} that there exists a holomorphic $(1,0)$ form $h_{j,1}$ on $\Omega_j$ such that
$$F=\wedge_{1\le j\le n}\pi_j^*(h_{j,1})$$
on $M$ and $h_{j,0}=h_{j,1}$ on $\{z\in\Omega_j:2\sum_{1\le k\le m_j}p_{j,k}G_{\Omega_j}(\cdot,z_{j,k})<-t_0\}\cap V_{z_{j,1}}$.

Denote that $\gamma_{j,k}:=ord_{z_{j,k}}h_{j,1}$ for any $j\in\{1,2,...,n\}$ and $k\in\{1,2,...,m_j\}$. As $G_{\beta}(-\log r)$ is linear with respect to $r\in(0,e^{-t_0}]$ and $G_{\beta}(t)=\int_{\{\psi<-t\}\cap V_{\beta}}|F|^2e^{-\varphi}>0$ for any $t\ge t_0$, it follows from Theorem \ref{thm:linear-2d} and Lemma \ref{l:0} that
$$\sum_{1\le j\le n}\frac{\gamma_{j,\beta_j}+1}{p_{j,\beta_j}}=1$$ and
$$f=(c_{\beta}\Pi_{1\le j\le n}w_{j,\beta_j}^{\gamma_{j,\beta_j}}+g_{\beta})dw_{1,\beta_1}\wedge dw_{2,\beta_2}\wedge...\wedge dw_{n,\beta_n}$$
 on $V_{\beta}$ for any $\beta\in I_1$, where $c_{\beta}$ is a constant and $g_{\beta}$ is a holomorphic function on $V_{\beta}$ such that $(g_{\beta},z_{\beta})\in\mathcal{I}(\psi)_{z_{\beta}}$.

 \

 \emph{Step 3. $\varphi_j=2\log|g_j|+2u_j$ and $\chi_{j,-u_j}=\Pi_{1\le k\le m_j}\chi_{j,z_{j,k}}^{\gamma_{j,k}+1}.$}

 \

As  $\sum_{1\le j\le n}\frac{\gamma_{j,\beta_j}+1}{p_{j,\beta_j}}=1$ for any $\beta\in I_1$, we have
$$\frac{\gamma_{j,1}+1}{p_{j,1}}=\frac{\gamma_{j,k}+1}{p_{j,k}}$$
 for any $j\in\{1,2,...,n\}$ and $k\in\{1,2,...,m_j\}$. Denote that $\Psi_j:=2\sum_{1\le k\le m_j}(\gamma_{j,k}+1)G_{\Omega_j}(\cdot,z_{j,k})$. For any $j\in\{1,2,...,n\}$, denote
\begin{equation*}
\begin{split}
\inf\{\int_{\{\Psi_j<-t\}}|\tilde{f}|^{2}e^{-\varphi_j}:(\tilde{f}-h_{j,1},z_{j,k})\in(\mathcal{O}(K_{\Omega_j})\otimes&\mathcal{I}(\Psi_j))_{z_{j,k}}\,\mbox{for any $k\in\{1,2,...,m_j\}$}\\&\&{\,}\tilde{f}\in H^{0}(\{\Psi_j<-t\},\mathcal{O}(K_{\Omega_j}))\},
\end{split}
\end{equation*}
by $G_{j}(t)$,  where $t\in[0,+\infty)$. Note that $\psi=\max_{1\le j\le n}\{\pi_j^*(\frac{p_{j,1}}{\gamma_{j,1}+1}\Psi_j)\}$ and $f=(c_{\beta}\Pi_{1\le j\le n}w_{j,\beta_j}^{\gamma_{j,\beta_j}}+g_{\beta})dw_{1,\beta_1}\wedge dw_{2,\beta_2}\wedge...\wedge dw_{n,\beta_n}$ on $V_{\beta}$ for any $\beta\in I_1$, where $c_{\beta}$ is a constant and $g_{\beta}$ is a holomorphic function on $V_{\beta}$ such that $(g_{\beta},z_{\beta})\in\mathcal{I}(\psi)_{z_{\beta}}$. For any $j\in\{1,2,...,n\}$ and $t\ge 0$, taking any $\tilde{f}_j\in H^{0}(\{\frac{p_{j,1}}{\gamma_{j,1}+1}\Psi_j<-t\},\mathcal{O}(K_{\Omega_j}))$ satisfying $(\tilde{f}_j-h_{j,1},z_{j,k})\in(\mathcal{O}(K_{\Omega_j})\otimes\mathcal{I}(\Psi_j))_{z_{j,k}}$ for any $k\in\{1,2,...,m_j\}$, it follows from Lemma \ref{l:0} and $\sum_{1\le j\le n}\frac{\gamma_{j,\beta_j}+1}{p_{j,\beta_j}}=1$ that  $\wedge_{1\le j\le n}\pi_j^*(\tilde{f}_j)\in H^0(\{\psi<-t\},\mathcal{O}(K_M))$ and $(\wedge_{1\le j\le n}\pi_j^*(\tilde{f}_j)-f,z_{\beta})\in(\mathcal{O}(K_M)\otimes\mathcal{I}(\psi))_{z_{\beta}}$ for any $\beta\in I_1$, which implies that
\begin{equation}
	\label{eq:1121b}G(t;\tilde{c}\equiv1)=\Pi_{1\le j\le n}G_{j}(\frac{\gamma_{j,1}+1}{p_{j,1}}t)
\end{equation}
holds for any $t\ge0$. Denote that
$$r_{j,k} :=\varphi_j(z_{j,k})+\lim_{z\rightarrow z_{j,k}}(\Psi_j-2(\gamma_{j,k}+1)\log|w_{j,k}|)$$
and
$$b_{j,k}:=\lim_{z\rightarrow z_{j,k}}\frac{h_{j,1}}{w_{j,k}^{\gamma_{j,k}}dw_{j,k}}$$
for any $j\in\{1,2,...,n\}$ and $k\in\{1,2,...,m_j\}$. It follows from Proposition \ref{p:exten-pro-finite} that
\begin{equation}
\label{eq:1121c}
	G_j(t)\le e^{-t}\sum_{1\le k\le m_j}\frac{2\pi|b_{j,k}|^2e^{-r_{j,k}}}{\gamma_{j,k}+1}
\end{equation}
for any $j\in\{1,2,...,n\}$. As $\sum_{1\le j\le n}\frac{\gamma_{j,1}+1}{p_{j,1}}=1$, following from equality \eqref{eq:1121b} and inequality \eqref{eq:1121c}, we obtain that
\begin{equation}
	\label{eq:1121d}
	G(t;\tilde{c}\equiv1)\le e^{-t}\Pi_{1\le j\le n}(\sum_{1\le k\le m_j}\frac{2\pi|b_{j,k}|^2e^{-r_{j,k}}}{\gamma_{j,k}+1}).
\end{equation}

As $M$ is a Stein manifold and $\varphi$ is plurisubharmonic function on $M$, there exist smooth plurisubharmonic functions $\Phi_l$ on $M$, which are decreasingly convergent to $\varphi$ with respect to $l$.  It follows from Lemma \ref{l:green-sup} and Lemma \ref{l:green-sup2} that there exists a local coordinate $\tilde{w}_{j,k}$ on a neighborhood $\tilde{V}_{z_{j,k}}\Subset V_{z_{j,k}}$ of $z_{j,k}$ satisfying $\tilde{w}_{j,k}(z_{j,k})=0$ and $|\tilde{w}_{j,k}|=e^{\frac{1}{p_{j,k}}\sum_{1\le k_1\le m_j}p_{j,k_1}G_{\Omega_j}(\cdot,z_{j,k_1})}$ on $\tilde{V}_{z_{j,k}}$ for any $j\in\{1,2,...,n\}$ and $k\in\{1,2,...,m_j\}$. Denote that $\tilde{V}_{\beta}:=\Pi_{1\le j\le n}V_{z_{j,\beta_j}}$ for any $\beta\in I_1$. It follows from Lemma \ref{l:G-compact} that there exists $t_2>0$ such that $\{\psi<-t_2\}\Subset\cup_{\beta\in I_1}\tilde{V}_\beta$.  Note that
\begin{equation}
	\label{eq:1128d}\begin{split}G(t;\tilde{c}\equiv1 )&=\int_{\{\psi<-t\}}|F|^2e^{-\varphi}	\\
	&\ge \int_{\{\psi<-t\}}|F|^2e^{-\Phi_l}
	\\&=\sum_{\beta\in I_1}\int_{\{\psi<-t\}\cap \tilde{V}_{\beta}}|F|^2e^{-\Phi_l}
\end{split}\end{equation}
for any $t\ge t_2$.
It follows from Lemma \ref{l:0} that
\begin{equation}
	\label{eq:1128f}(F-(\Pi_{1\le j\le n}b_{j,\beta_j}w_{j,\beta_j}^{\gamma_{j,\beta_j}})d{w}_{1,\beta_1}\wedge d{w}_{2,\beta_2}\wedge...\wedge d{w}_{n,\beta_n},z_{\beta})\in (\mathcal{O}(K_{M})\otimes\mathcal{I}(\psi))_{z_{\beta}}
\end{equation}
 for any $\beta\in I_1$, which implies that
$$F=(\Pi_{1\le j\le n}\tilde{b}_{j,\beta_j}\tilde{w}_{j,\beta_j}^{\gamma_{j,\beta_j}}+\tilde{g}_{\beta})d\tilde{w}_{1,\beta_1}\wedge d\tilde{w}_{2,\beta_2}\wedge...\wedge d\tilde{w}_{n,\beta_n},$$
where $\tilde{b}_{j,\beta_j}=b_{j,\beta_j}(\lim_{z\rightarrow z_{j,\beta}}\frac{w_{j,\beta_j}}{\tilde{w}_{j,\beta_j}})^{\gamma_{j,\beta_j}+1}$ and $(\tilde{g}_{\beta},z_{\beta})\in\mathcal{I}(\psi)_{z_\beta}$ for any $\beta\in I_1$. It follows from inequality \eqref{eq:1128d} and Lemma \ref{l:m2} that
\begin{equation}
	\label{eq:1128e}\begin{split}
	\liminf_{t\rightarrow+\infty}e^tG(t;\tilde{c}\equiv1 )\ge \sum_{\beta\in I_1}(2\pi)^ne^{-\Phi_l(z_{\beta})}\Pi_{1\le j\le n}\frac{|\tilde{b}_{j,\beta_j}|^2}{\gamma_{j,\beta_j}+1}.		
	\end{split}\end{equation}
	Note that
	\begin{displaymath}
		\begin{split}
			e^{-\varphi(z_\beta)}\Pi_{1\le j\le n}|\tilde{b}_{j,\beta_j}|^2&=			e^{-\varphi(z_\beta)}\Pi_{1\le j\le n}|{b}_{j,\beta_j}|^2\lim_{z\rightarrow z_{j,\beta}}|\frac{w_{j,\beta_j}}{\tilde{w}_{j,\beta_j}}|^{2\gamma_{j,\beta_j}+2}\\
			&=\Pi_{1\le j\le n}|{b}_{j,\beta_j}|^2e^{-(\varphi_j(z_{j,\beta_j})+\lim_{z\rightarrow z_{j,\beta_j}}(\Psi_j-(2\gamma_{j,\beta_j}+2)\log|w_{j,\beta_j}|))}\\
			&=\Pi_{1\le j\le n}|{b}_{j,\beta_j}|^2e^{-r_{j,\beta_j}}.
		\end{split}
	\end{displaymath}
Letting $l\rightarrow+\infty$, inequality \eqref{eq:1128e} shows that
\begin{equation}
	\label{eq:1121e}\begin{split}
	\liminf_{t\rightarrow+\infty}e^{t}G(t;\tilde{c}\equiv1)&\ge\sum_{\beta\in I_1}(2\pi)^n\Pi_{1\le j\le n}\frac{|{b}_{j,\beta_j}|^2e^{-r_{j,\beta_j}}}{\gamma_{j,\beta_j}+1}\\
	&=\Pi_{1\le j\le n}(\sum_{1\le k\le m_j}\frac{2\pi|b_{j,k}|^2e^{-r_{j,k}}}{\gamma_{j,k}+1}).
	\end{split}
\end{equation}
As $G(-\log r;\tilde{c}\equiv1)$ is linear with respect to $r$, it follows from inequality \eqref{eq:1121d} and inequality \eqref{eq:1121e} that
\begin{equation}
	\label{eq:1121f}G(t;\tilde{c}\equiv1)= e^{-t}\Pi_{1\le j\le n}(\sum_{1\le k\le m_j}\frac{2\pi|b_{j,k}|^2e^{-r_{j,k}}}{\gamma_{j,k}+1}),
\end{equation}
which implies that
\begin{equation}\label{eq:1122d}
		G_j(t)= e^{-t}\sum_{1\le k\le m_j}\frac{2\pi|b_{j,k}|^2e^{-r_{j,k}}}{\gamma_{j,k}+1},
\end{equation}
i.e. $G_j(-\log r)$ is linear with respect to $r\in(0,1]$.
Note that $r_{j,k}=\varphi_j(z_{j,k})+\lim_{z\rightarrow z_{j,k}}(\Psi_j-2(\gamma_{j,k}+1)\log|w_{j,k}|)$. It follows from Theorem
\ref{c:L2-1d-char} that
$$\varphi_j=2\log|g_j|+2u_j$$
and
$$\Pi_{1\le j\le m_j}\chi_{j,z_{j,k}}^{\gamma_{j,k}+1}=\chi_{j,-u_j}$$
for any $j\in\{1,2,...,n\}$ and $k\in\{1,2,...,m_j\}$,
where $g_j$ is a holomorphic function on $\Omega_j$ such that $g_j(z_{j,k})\not=0$ and $u_j$ is a harmonic function on $\Omega_j$.

 \

 \emph{Step 4. $\lim_{z\rightarrow z_{\beta}}\frac{c_{\beta}\Pi_{1\le j\le n}w_{j,\beta_j}^{\gamma_{j,\beta_j}}dw_{1,\beta_1}\wedge dw_{2,\beta_2}\wedge...\wedge dw_{n,\beta_n}}{\wedge_{1\le j\le n}\pi_{j}^*(g_j(P_{j})_*(f_{u_j}(\Pi_{1\le k\le m_j}f_{z_{j,k}}^{\gamma_{j,k}+1})(\sum_{1\le k\le m_j}p_{j,k}\frac{df_{z_{j,k}}}{f_{z_{j,k}}})))}=c_0$.}

 \

By the definition of $G_{j}(t)$, we have
\begin{equation}
	\label{eq:1122e}\int_{\{\Psi_j<-t\}}|h_{j,1}|^2e^{-\varphi_j}\geq e^{-t}\sum_{1\le k\le m_j}\frac{2\pi|b_{j,k}|^2e^{-r_{j,k}}}{\gamma_{j,k}+1}.
\end{equation}
 As $G(t;\tilde{c}\equiv1)=\int_{\{\psi<-t\}}|F|^2e^{-\varphi}=e^{-t}\Pi_{1\le j\le n}(\sum_{1\le k\le m_j}\frac{2\pi|b_{j,k}|^2e^{-r_{j,k}}}{\gamma_{j,k}+1})$ and $F=\wedge_{1\le k\le m_j}\pi_j^{*}(h_{j,1})$, it follows from inequality \eqref{eq:1122e} that
 \begin{equation}
 	\label{eq:1122f}\int_{\{\Psi_j<-t\}}|h_{j,1}|^2e^{-\varphi_j}= e^{-t}\sum_{1\le k\le m_j}\frac{2\pi|b_{j,k}|^2e^{-r_{j,k}}}{\gamma_{j,k}+1}=G_j(t).
 \end{equation}
It follows from Remark \ref{rem:1.2} and equality \eqref{eq:1122f}  that
\begin{equation*}
	\label{eq:1122c}h_{j,1}=b_jg_{j}(P_j)_*(f_{u_{j}}(\Pi_{1\le k\le m_j}f_{z_{j,k}}^{\gamma_{j,k}+1})(\sum_{1\le k\le m_j}p_{j,k}\frac{d{f_{z_{j,k}}}}{f_{z_{j,k}}})),
\end{equation*}
where $b_j$ is a constant,
hence we have
\begin{equation}
	\label{eq:1122g}F=\wedge_{1\le j\le n} \pi_j^*(b_jg_{j}(P_j)_*(f_{u_{j}}(\Pi_{1\le k\le m_j}f_{z_{j,k}}^{\gamma_{j,k}+1})(\sum_{1\le k\le m_j}p_{j,k}\frac{d{f_{z_{j,k}}}}{f_{z_{j,k}}}))).
\end{equation}
As $(F-f,z_{\beta})\in(\mathcal{O}(K_{M})\otimes\mathcal{I}(\psi))_{z_{\beta}}$ and $f=(c_{\beta}\Pi_{1\le j\le n}w_{j,\beta_j}^{\gamma_{j,\beta_j}}+g_{\beta})dw_{1,\beta_1}\wedge dw_{2,\beta_2}\wedge...\wedge dw_{n,\beta_n}$ on $V_{\beta}$ for any $\beta\in I_1$, it follows from Lemma \ref{l:0} that
\begin{displaymath}
	\lim_{z\rightarrow z_{\beta}}\frac{c_{\beta}\Pi_{1\le j\le n}w_{j,\beta_j}^{\gamma_{j,\beta_j}}dw_{1,\beta_1}\wedge dw_{2,\beta_2}\wedge...\wedge dw_{n,\beta_n}}{\wedge_{1\le j\le n}\pi_{j}^*(g_j(P_{j})_*(f_{u_j}(\Pi_{1\le k\le m_j}f_{z_{j,k}}^{\gamma_{j,k}+1})(\sum_{1\le k\le m_j}p_{j,k}\frac{df_{z_{j,k}}}{f_{z_{j,k}}})))}=\Pi_{1\le j\le n}b_j
\end{displaymath}for any $\beta\in I_1$. Denote that $c_0=\Pi_{1\le j\le n}b_j$.

Thus, we prove the necessity of Theorem \ref{thm:prod-finite-point}.

\subsection{Proof of Remark \ref{r:1.2}}\label{sec:pf-r1.2}
\

Following from Corollary \ref{c:linear}, we have the uniqueness, thus it suffices to prove the existence.
 It follows from Section \ref{sec:n-2} that
 $$F=c_0\wedge_{1\le j\le n}\pi_j^*(g_j(P_j)_*(f_{u_j}(\Pi_{1\le k\le m_j}f_{z_{j,k}}^{\gamma_{j,k}+1})(\sum_{1\le k\le m_j}p_{j,k}\frac{df_{z_{j,k}}}{f_{j,k}})))$$
 and
 \begin{displaymath}
  	\begin{split}
  		G(t;\tilde{c}\equiv1 )&=\int_{\{\psi<-t\}}|F|^2e^{-\varphi}\\
  		&=e^{-t}\Pi_{1\le j\le n}(\sum_{1\le k\le m_j}\frac{2\pi|b_{j,k}|^2e^{-r_{j,k}}}{\gamma_{j,k}+1}).
  	\end{split}
  \end{displaymath}
As $(F-f,z_{\beta})\in(\mathcal{O}(K_M)\otimes\mathcal{I}(\psi))_{z_\beta}$ for any $\beta\in I_1$, it follows from \eqref{eq:1128f} and Lemma \ref{l:0} that
$$\Pi_{1\le j\le n}b_{j,\beta_j}=c_{\beta}$$
for any $\beta\in I_1$.
 As
$\Pi_{1\le j\le n}e^{-r_{j,\beta_j}}=\frac{e^{-\varphi(z_{\beta})}}{\Pi_{1\le j\le n}c_{j,\beta_j}^{2(\gamma_{j,\beta_j}+1)}},$
  we obtain that
  \begin{displaymath}
  	\begin{split}
  		G(t;\tilde{c}\equiv1 )&=\int_{\{\psi<-t\}}|F|^2e^{-\varphi}\\
  		&=e^{-t}\sum_{\beta\in I_1}\frac{|c_{\beta}|^2(2\pi)^ne^{-\varphi(z_{\beta})}}{\Pi_{1\le j\le n}(\gamma_{j,\beta_j}+1)c_{j,\beta_j}^{2(\gamma_{j,\beta_j}+1)}}.
  	\end{split}
  \end{displaymath}
Following from equality \eqref{eq:1121a} in Section \ref{sec:n-2} and Corollary \ref{c:linear} that
\begin{displaymath}
	\begin{split}
		G(t;c)&=\int_{\{\psi<-t\}}|F|^2e^{-\varphi}c(-\psi)\\
		&=(\int_t^{+\infty}c(s)e^{-s}ds)\sum_{\beta\in I_1}\frac{|c_{\beta}|^2(2\pi)^ne^{-\varphi(z_{\beta})}}{\Pi_{1\le j\le n}(\gamma_{j,\beta_j}+1)c_{j,\beta_j}^{2(\gamma_{j,\beta_j}+1)}},
	\end{split}
\end{displaymath}
hence $c_0\wedge_{1\le j\le n}\pi_j^*(g_j(P_j)_*(f_{u_j}(\Pi_{1\le k\le m_j}f_{z_{j,k}}^{\gamma_{j,k}+1})(\sum_{1\le k\le m_j}p_{j,k}\frac{df_{z_{j,k}}}{f_{j,k}})))$ is the $(n,0)$ form satisfying the conditions in Remark \ref{r:1.2}.

\section{Proof of Theorem \ref{thm:prod-infinite-point}}
In this Section, we prove Theorem \ref{thm:prod-infinite-point} by contradiction. We assume that $G(h^{-1}(r))$ is linear.

Following from the linearity of $G(h^{-1}(r))$ and Corollary \ref{c:linear}, there exists a holomorphic $(n,0)$ form $F$ on $M$, such that $(F-f,z_{\beta})\in(\mathcal{O}(K_M)\otimes\mathcal{I}(\psi))_{z_{\beta}}$ for any $\beta\in\tilde I_1$ and
\begin{equation}
	\label{eq:1130a}G(t)=\int_{\{\psi<-t\}}|F|^2e^{-\varphi}c(-\psi).
\end{equation}
As $\limsup_{t\rightarrow+\infty}c(t)<+\infty$, if $\tilde{F}$ is a holomorphic $(n,0)$ form on $\{\psi<-t\}$ satisfying $\int_{\{\psi<-t\}}|\tilde{F}|^2e^{-\varphi}<+\infty$, we have $\int_{\{\psi<-t\}}|F|^2e^{-\varphi}c(-\psi)<+\infty$, where $t>0$.
 Thus, it follows from Corollary \ref{c:linear}  that
 $$G(t;\tilde{c}\equiv1)=\int_{\{\psi<-t\}}|F|^2e^{-\varphi}=e^{-t}\frac{G(0)}{\int_0^{+\infty}c(l)e^{-l}dl}.$$
Theorem \ref{thm:general_concave} shows that $G(0,\tilde{c})=\lim_{t\rightarrow0+0}G(t;\tilde{c})$. Thus we have $G(-\log r,\tilde{c})$ is linear with respect to $r\in(0,1].$

For any $\beta\in\tilde I_1$. It follows from Lemma \ref{l:green-sup2} and Lemma \ref{l:G-compact} that there exists $t_{\beta}>0$ such that $\{\psi<-t_\beta\}\cap V_{\beta}\Subset V_{\beta}$ and $\{z\in\Omega_j:2\sum_{1\le k<\tilde m_j}p_{j,k}G_{\Omega_j}(\cdot,z_{j,k})<-t_\beta\}\cap V_{z_{j,\beta_j}}$ is simply connected for any $j\in\{1,2,...,n\}$. Denote\begin{equation*}
\begin{split}
\inf\{\int_{\{\psi<-t\}\cap V_\beta}|\tilde{f}|^{2}e^{-\varphi}:(\tilde{f}-f,z_{\beta})\in&(\mathcal{O}(K_M)\otimes\mathcal{I}(\psi))_{z_{\beta}}\\&\&{\,}\tilde{f}\in H^{0}(\{\psi<-t\}\cap V_{\beta},\mathcal{O}(K_{M}))\},
\end{split}
\end{equation*}
by $G_{\beta}(t)$,  where $t\in[t_\beta,+\infty)$. Note that  $\{\psi<-t\}\cap V_{\beta}=\Pi_{1\le j\le n}(\{z\in\Omega_j:2\sum_{1\le k<\tilde m_j}p_{j,k}G_{\Omega_j}(\cdot,z_{j,k})<-t\}\cap V_{z_{j,\beta_j}})$ for any $t\ge t_\beta$.
Denote\begin{equation*}
\begin{split}
\inf\{\int_{\{\psi<-t\}\backslash V_{\beta}}|\tilde{f}|^{2}e^{-\varphi}:(\tilde{f}-f,z_{\tilde\beta})\in&(\mathcal{O}(K_M)\otimes\mathcal{I}(\psi))_{z_{\tilde\beta}}\mbox{ for any $\tilde\beta\in \tilde{I}_1\backslash\{\beta\}$} \\&\&{\,}\tilde{f}\in H^{0}(\{\psi<-t\}\backslash V_{\beta},\mathcal{O}(K_{M}))\},
\end{split}
\end{equation*}
by $\tilde G_{\beta}(t)$,  where $t\in[t_\beta,+\infty)$.
 By the definition of $G(t;\tilde{c})$, $G_{\beta}(t)$ and $\tilde{G}_{\beta}(t)$, we have $G(t;\tilde{c})=G_{\beta}(t)+\tilde{G}_{\beta}(t)$ for $t\ge t_\beta$, hence
 $$G_{\beta}(t)=\int_{\{\psi<-t\}\cap V_{\beta}}|F|^2e^{-\varphi}$$
 holds for any $t\ge t_{\beta}$.
  Theorem \ref{thm:general_concave} tells us that $G_{\beta}(-\log r)$  and $\tilde{G}_{\beta}(-\log r)$ are concave with respect to $r\in(0,e^{-t_\beta})$. As $G(-\log r;\tilde{c})$ is linear with respect to $r$, we have $G_{\beta}(-\log r)$ is linear with respect to $r\in(0,e^{-t_\beta}]$.

 As $f=w_{\beta^*}^{\alpha_{\beta_*}}dw_{1,1}\wedge dw_{2,1}\wedge...\wedge dw_{n,1}$ on $V_{\beta^*}$ and $\frac{1}{2p_{j,1}}(2\sum_{1\le k<\tilde m_j}p_{j,k}G_{\Omega_j}(\cdot,z_{j,k})+t_{\beta^*})$ is the Green function on $\{z\in\Omega_j:2\sum_{1\le k<\tilde{m}_j}p_{j,k}G_{\Omega_j}(\cdot,z_{j,k})<-t_{\beta^*}\}\cap V_{z_{j,1}}$, where $\beta^*=(1,1,...,1)\in\tilde I_1$. Using Theorem \ref{thm:linear-2d} and Remark \ref{r:1.1}, we obtain that there exists a holomorphic $(1,0)$ form $h_{j,0}$ on $\{z\in\Omega_j:2\sum_{1\le k<\tilde{m}_j}p_{j,k}G_{\Omega_j}(\cdot,z_{j,k})<-t_{\beta^*}\}\cap V_{z_{j,1}}$ for any $j\in\{1,2,...,n\}$ such that $F=\wedge_{1\le j\le n}\pi_{j}^*(h_{j,0})$ on $\{\psi<-t_{\beta^*}\}\cap V_{\beta^*}$.  It follows from Lemma \ref{l:decom} that there exists a holomorphic $(1,0)$ form $h_{j,1}$ on $\Omega_j$ such that
$$F=\wedge_{1\le j\le n}\pi_j^*(h_{j,1})$$
on $M$ and $h_{j,0}=h_{j,1}$ on $\{z\in\Omega_j:2\sum_{1\le k<\tilde{m}_j}p_{j,k}G_{\Omega_j}(\cdot,z_{j,k})<-t_{\beta^*}\}\cap V_{z_{j,1}}$.

Denote that $\gamma_{j,k}:=ord_{z_{j,k}}h_{j,1}$ for any $j\in\{1,2,...,n\}$ and $1\le k<\tilde{m}_j$. As $G_{\beta}(-\log r)$ is linear with respect to $r\in(0,e^{-t_\beta}]$ and $G_{\beta}(t)=\int_{\{\psi<-t\}\cap V_{\beta}}|F|^2e^{-\varphi}>0$ for any $t\ge t_\beta$, it follows from Theorem \ref{thm:linear-2d} and Lemma \ref{l:0} that
$$\sum_{1\le j\le n}\frac{\gamma_{j,\beta_j}+1}{p_{j,\beta_j}}=1$$ and
$$f=(c_{\beta}\Pi_{1\le j\le n}w_{j,\beta_j}^{\gamma_{j,\beta_j}}+g_{\beta})dw_{1,\beta_1}\wedge dw_{2,\beta_2}\wedge...\wedge dw_{n,\beta_n}$$
 on $V_{\beta}$ for any $\beta\in\tilde I_1$, where $c_{\beta}$ is a constant and $g_{\beta}$ is a holomorphic function on $V_{\beta}$ such that $(g_{\beta},z_{\beta})\in\mathcal{I}(\psi)_{z_{\beta}}$.

As  $\sum_{1\le j\le n}\frac{\gamma_{j,\beta_j}+1}{p_{j,\beta_j}}=1$ for any $\beta\in\tilde I_1$, we have
$$\frac{\gamma_{j,1}+1}{p_{j,1}}=\frac{\gamma_{j,k}+1}{p_{j,k}}$$
 for any $j\in\{1,2,...,n\}$ and $1\le k<\tilde m_j$. Denote that $\Psi_j:=2\sum_{1\le k<\tilde m_j}(\gamma_{j,k}+1)G_{\Omega_j}(\cdot,z_{j,k})$. For any $j\in\{1,2,...,n\}$, denote
\begin{equation*}
\begin{split}
\inf\{\int_{\{\Psi_j<-t\}}|\tilde{f}|^{2}e^{-\varphi_j}:(\tilde{f}-h_{j,1},z_{j,k})\in(\mathcal{O}(K_{\Omega_j})\otimes&\mathcal{I}(\Psi_j))_{z_{j,k}}\,\mbox{for any $1\le k<\tilde m_j$}\\&\&{\,}\tilde{f}\in H^{0}(\{\Psi_j<-t\},\mathcal{O}(K_{\Omega_j}))\},
\end{split}
\end{equation*}
by $G_{j}(t)$,  where $t\in[0,+\infty)$. Note that $\psi=\max_{1\le j\le n}\{\pi_j^*(\frac{p_{j,1}}{\gamma_{j,1}+1}\Psi_j)\}$ and $f=(c_{\beta}\Pi_{1\le j\le n}w_{j,\beta_j}^{\gamma_{j,\beta_j}}+g_{\beta})dw_{1,\beta_1}\wedge dw_{2,\beta_2}\wedge...\wedge dw_{n,\beta_n}$ on $V_{\beta}$ for any $\beta\in\tilde I_1$, where $c_{\beta}$ is a constant and $g_{\beta}$ is a holomorphic function on $V_{\beta}$ such that $(g_{\beta},z_{\beta})\in\mathcal{I}(\psi)_{z_{\beta}}$. For any $j\in\{1,2,...,n\}$ and $t\ge 0$, choosing any $\tilde{f}_j\in H^{0}(\{\frac{p_{j,1}}{\gamma_{j,1}+1}\Psi_j<-t\},\mathcal{O}(K_{\Omega_j}))$ satisfying $(\tilde{f}_j-h_{j,1},z_{j,k})\in(\mathcal{O}(K_{\Omega_j})\otimes\mathcal{I}(\Psi_j))_{z_{j,k}}$ for any $1\le k<\tilde m_j$, we have $\wedge_{1\le j\le n}\pi_j^*(\tilde{f}_j)\in H^0(\{\psi<-t\},\mathcal{O}(K_M))$ and $(\wedge_{1\le j\le n}\pi_j^*(\tilde{f}_j)-f,z_{\beta})\in(\mathcal{O}(K_M)\otimes\mathcal{I}(\psi))_{z_{\beta}}$ for any $\beta\in\tilde I_1$, which implies that
\begin{equation}
	\label{eq:1130b}G(t;\tilde{c}\equiv1)=\Pi_{1\le j\le n}G_{j}(\frac{\gamma_{j,1}+1}{p_{j,1}}t)
\end{equation}
holds for any $t\ge0$. Denote that
$$r_{j,k} :=\varphi_j(z_{j,k})+\lim_{z\rightarrow z_{j,k}}(\Psi_j-2(\gamma_{j,k}+1)\log|w_{j,k}|)$$
and
$$b_{j,k}:=\lim_{z\rightarrow z_{j,k}}\frac{h_{j,1}}{w_{j,k}^{\gamma_{j,k}}dw_{j,k}}$$
for any $j\in\{1,2,...,n\}$ and $1\le k<\tilde m_j$. It follows from Proposition \ref{p:exten-pro-finite} that
\begin{equation}
\label{eq:1130c}
	G_j(t)\le e^{-t}\sum_{1\le k<\tilde m_j}\frac{2\pi|b_{j,k}|^2e^{-r_{j,k}}}{\gamma_{j,k}+1}
\end{equation}
for any $j\in\{1,2,...,n\}$. As $\sum_{1\le j\le n}\frac{\gamma_{j,1}+1}{p_{j,1}}=1$, following from equality \eqref{eq:1130b} and inequality \eqref{eq:1130c}, we obtain that
\begin{equation}
	\label{eq:1130d}
	G(t;\tilde{c}\equiv1)\le e^{-t}\Pi_{1\le j\le n}(\sum_{1\le k<\tilde m_j}\frac{2\pi|b_{j,k}|^2e^{-r_{j,k}}}{\gamma_{j,k}+1}).
\end{equation}

As $M$ is a Stein manifold and $\varphi$ is plurisubharmonic function on $M$, there exist smooth plurisubharmonic functions $\Phi_l$ on $M$, which are decreasingly convergent to $\varphi$ with respect to $l$. As $\{\psi<-t_\beta\}\cap V_{\beta}\Subset V_{\beta}$ for any $\beta\in\tilde I_1$, we have
\begin{equation}
	\label{eq:1130e}\begin{split}\lim_{t\rightarrow+\infty}e^{t}G(t;\tilde{c}\equiv1 )&=\lim_{t\rightarrow+\infty}e^{t}\int_{\{\psi<-t\}}|F|^2e^{-\varphi}	\\
	&\ge\sum_{\beta\in\tilde I_1}\liminf_{t\rightarrow+\infty}e^{t}\int_{\{\psi<-t\}\cap {V}_{\beta}}|F|^2e^{-\Phi_l}.
\end{split}\end{equation}
It follows from Lemma \ref{l:green-sup} and Lemma \ref{l:green-sup2} that there exists a local coordinate $\tilde{w}_{j,k}$ on a neighborhood $\tilde{V}_{z_{j,k}}\Subset V_{z_{j,k}}$ of $z_{j,k}$ satisfying $\tilde{w}_{j,k}(z_{j,k})=0$ and $|\tilde{w}_{j,k}|=e^{\frac{1}{p_{j,k}}\sum_{1\le k_1<\tilde m_j}p_{j,k_1}G_{\Omega_j}(\cdot,z_{j,k_1})}$ on $\tilde{V}_{z_{j,k}}$ for any $j\in\{1,2,...,n\}$ and $1\le k<\tilde m_j$. Denote that $\tilde{V}_{\beta}:=\Pi_{1\le j\le n}\tilde V_{z_{j,\beta_j}}$ for any $\beta\in\tilde I_1$.

It follows from Lemma \ref{l:0} that
\begin{equation*}
	(F-(\Pi_{1\le j\le n}b_{j,\beta_j}w_{j,\beta_j}^{\gamma_{j,\beta_j}})d{w}_{1,\beta_1}\wedge d{w}_{2,\beta_2}\wedge...\wedge d{w}_{n,\beta_n},z_{\beta})\in (\mathcal{O}(K_{M})\otimes\mathcal{I}(\psi))_{z_{\beta}}
\end{equation*}
 for any $\beta\in\tilde I_1$, which implies that
 \begin{equation}
 	\label{eq:1130f}F=(\Pi_{1\le j\le n}\tilde{b}_{j,\beta_j}\tilde{w}_{j,\beta_j}^{\gamma_{j,\beta_j}}+\tilde{g}_{\beta})d\tilde{w}_{1,\beta_1}\wedge d\tilde{w}_{2,\beta_2}\wedge...\wedge d\tilde{w}_{n,\beta_n}, \end{equation}
where $\tilde{b}_{j,\beta_j}=b_{j,\beta_j}(\lim_{z\rightarrow z_{j,\beta}}\frac{w_{j,\beta_j}}{\tilde{w}_{j,\beta_j}})^{\gamma_{j,\beta_j}+1}$ and $(\tilde{g}_{\beta},z_{\beta})\in\mathcal{I}(\psi)_{z_\beta}$ for any $\beta\in\tilde I_1$. It follows from    inequality \eqref{eq:1130e}, equality \eqref{eq:1130f} and Lemma \ref{l:m2} that
\begin{equation}
	\label{eq:1130g}\begin{split}
	\lim_{t\rightarrow+\infty}e^tG(t;\tilde{c}\equiv1 )\ge \sum_{\beta\in\tilde I_1}(2\pi)^ne^{-\Phi_l(z_{\beta})}\Pi_{1\le j\le n}\frac{|\tilde{b}_{j,\beta_j}|^2}{\gamma_{j,\beta_j}+1}.		
	\end{split}\end{equation}
	Note that
	\begin{displaymath}
		\begin{split}
			e^{-\varphi(z_\beta)}\Pi_{1\le j\le n}|\tilde{b}_{j,\beta_j}|^2&=			e^{-\varphi(z_\beta)}\Pi_{1\le j\le n}|{b}_{j,\beta_j}|^2\lim_{z\rightarrow z_{j,\beta}}|\frac{w_{j,\beta_j}}{\tilde{w}_{j,\beta_j}}|^{2\gamma_{j,\beta_j}+2}\\
			&=\Pi_{1\le j\le n}|{b}_{j,\beta_j}|^2e^{-(\varphi_j(z_{j,\beta_j})+\lim_{z\rightarrow z_{j,\beta_j}}(\Psi_j-(2\gamma_{j,\beta_j}+2)\log|w_{j,\beta_j}|))}\\
			&=\Pi_{1\le j\le n}|{b}_{j,\beta_j}|^2e^{-r_{j,\beta_j}}.
		\end{split}
	\end{displaymath}
Letting $l\rightarrow+\infty$, inequality \eqref{eq:1130g} shows that
\begin{equation}
	\label{eq:1130h}\begin{split}
	\liminf_{t\rightarrow+\infty}e^{t}G(t;\tilde{c}\equiv1)&\ge\sum_{\beta\in\tilde I_1}(2\pi)^n\Pi_{1\le j\le n}\frac{|{b}_{j,\beta_j}|^2e^{-r_{j,\beta_j}}}{\gamma_{j,\beta_j}+1}\\
	&=\Pi_{1\le j\le n}(\sum_{1\le k<\tilde m_j}\frac{2\pi|b_{j,k}|^2e^{-r_{j,k}}}{\gamma_{j,k}+1}).
	\end{split}
\end{equation}
As $G(-\log r;\tilde{c}\equiv1)$ is linear with respect to $r$, it follows from inequality \eqref{eq:1130d} and equality \eqref{eq:1130h} that
\begin{equation*}
G(t;\tilde{c}\equiv1)= e^{-t}\Pi_{1\le j\le n}(\sum_{1\le k\le m_j}\frac{2\pi|b_{j,k}|^2e^{-r_{j,k}}}{\gamma_{j,k}+1}),
\end{equation*}
which implies that
\begin{equation*}
		G_j(t)= e^{-t}\sum_{1\le k\le m_j}\frac{2\pi|b_{j,k}|^2e^{-r_{j,k}}}{\gamma_{j,k}+1},
\end{equation*}
i.e. $G_j(-\log r)$ is linear with respect to $r\in(0,1]$. As there exists $j_0\in\{1,2,...,n\}$ such that $\tilde m_j=+\infty$, it follows from $\frac{\gamma_{j_0,1}+1}{p_{j_0,1}}=\frac{\gamma_{j_0,k}+1}{p_{j_0,k}}$ for any $1\le k<\tilde m_j$ that $\sum_{1\le k<\tilde m_j}p_{j_0,k}=+\infty$, which contradicts to Proposition \ref{p:infinite}. Hence, we obtain that $G(h^{-1}(r))$ is not linear with respect to $r$.

\section{Proofs of Theorem \ref{thm:2d-jet}, Theorem \ref{thm:prod-finite-jet} and Theorem \ref{thm:prod-infinite-jet}}

In this section, we prove Theorem \ref{thm:2d-jet}, Theorem \ref{thm:prod-finite-jet} and Theorem \ref{thm:prod-infinite-jet}.

\subsection{Proof of Theorem \ref{thm:2d-jet}}
\

As $c(s)e^{-s}$ is decreasing with respect to $s$ and $\Psi\le0$, it follows from Proposition \ref{p:exten-pro-finite} that
there exists a holomorphic $(n,0)$ form $F$ on $M$ satisfying that $(F-f,z_0)\in(\mathcal{O}(K_M)\otimes\mathcal{I}(\max_{1\le j\le n}\{2p_j\pi_j^{*}(G_{\Omega_j}(\cdot,z_j))\}))_{z_0}$ and
\begin{equation}\label{eq:1127a}
	\begin{split}
	\int_{M}|F|^2e^{-\varphi}c(-\psi)&\le \int_{M}|F|^2e^{-\varphi-\Psi}c(-\psi+\Psi)\\
	&\le(\int_0^{+\infty}c(s)e^{-s}ds)\sum_{\alpha\in E}\frac{|d_{\alpha}|^2(2\pi)^ne^{-(\varphi+\Psi)(z_{0})}}{\Pi_{1\le j\le n}(\alpha_j+1)c_{j}(z_j)^{2\alpha_{j}+2}}.	\end{split}	
\end{equation}

In the following, we prove the characterization for the holding of the equality $(\int_0^{+\infty}c(s)e^{-s}ds)\sum_{\alpha\in E}\frac{|d_{\alpha}|^2(2\pi)^ne^{-(\varphi+\Psi)(z_{0})}}{\Pi_{1\le j\le n}(\alpha_j+1)c_{j}(z_j)^{2\alpha_{j}+2}}=\inf\{\int_{M}|\tilde{F}|^2e^{-\varphi}c(-\psi):\tilde{F}$ is a holomorphic $(n,0)$ form on $M$ such that $(\tilde{F}-f,z_0)\in(\mathcal{O}(K_{M})\otimes\mathcal{I}(\psi-\Psi))_{z_0}\}$.

Following from Remark \ref{r:1.1}, we obtain the sufficiency. Now, we prove the necessity.
Inequality \eqref{eq:1127a} shows that $\int_{M}|F|^2e^{-\varphi}c(-\psi)= \int_{M}|F|^2e^{-\varphi-\Psi}c(-\psi+\Psi)$. As  $c(s)e^{-s}$ is decreasing with respect to $s$ and $\Psi\le0$, it follows from Lemma \ref{l:psi=G} that $\Psi\equiv0$, then $\psi=\max_{1\le j\le n}\{\pi_j^*(2p_jG_{\Omega_j}(\cdot,z_j))\}.$

Let $t\ge0$. It follows from proposition \ref{p:exten-pro-finite}  ($M\sim\{\psi<-t\}$, $\psi\sim\psi+t$ and $c(\cdot)\sim c(\cdot+t)$, here $\sim$ means the former replaced by the latter), there exists a holomorphic $(n,0)$ form $F_t$ on $\{\psi<-t\}$ satisfying that $(F-f,z_0)\in(\mathcal{O}(K_M)\otimes\mathcal{I}(\psi))_{z_0}$ and
\begin{equation*}
\int_{\{\psi<-t\}}|F_t|^2e^{-\varphi}c(-\psi)\le(\int_t^{+\infty}c(s)e^{-s}ds)\sum_{\alpha\in E}\frac{|d_{\alpha}|^2(2\pi)^ne^{-\varphi(z_{0})}}{\Pi_{1\le j\le n}(\alpha_j+1)c_{j}(z_j)^{2\alpha_{j}+2}}.
\end{equation*}
By the definition of $G(t)$, then we obtain that inequality
\begin{equation*}
	\frac{G(t)}{\int_t^{+\infty}c(s)e^{-s}ds}\leq	\sum_{\alpha\in E}\frac{|d_{\alpha}|^2(2\pi)^ne^{-\varphi(z_{0})}}{\Pi_{1\le j\le n}(\alpha_j+1)c_{j}(z_j)^{2\alpha_{j}+2}}=\frac{G(0)}{\int_0^{+\infty}c(s)e^{-s}ds}\end{equation*}
holds for any $t\geq0$.
 Using Theorem \ref{thm:general_concave}, we obtain that  $G(\hat{h}^{-1}(r))$ is linear with respect to $r$.
It follows from Theorem \ref{thm:linear-2d} that the three statements in Theorem \ref{thm:2d-jet} hold.

Thus, Theorem \ref{thm:2d-jet} holds.

\subsection{Proof of Theorem \ref{thm:prod-finite-jet}}
\

 As $c(s)e^{-s}$ is decreasing with respect to $s$ and $\Psi\le0$, it follows from Proposition \ref{p:exten-pro-finite} that
there exists a holomorphic $(n,0)$ form $F$ on $M$ satisfying that $(F-f,z_\beta)\in(\mathcal{O}(K_M)\otimes\mathcal{I}(\max_{1\le j\le n}\{\pi_j^*(2\sum_{1\le k\le m_j}p_{j,k}G_{\Omega_j}(\cdot,z_{j,k}))\}))_{z_\beta}$ for any $\beta\in I_1$ and
\begin{equation}\label{eq:1127c}
	\begin{split}
	\int_{M}|F|^2e^{-\varphi}c(-\psi)&\le \int_{M}|F|^2e^{-\varphi-\Psi}c(-\psi+\Psi)\\
	&\le(\int_0^{+\infty}c(s)e^{-s}ds)\sum_{\beta\in I_1}\sum_{\alpha\in E_{\beta}}\frac{|d_{\beta,\alpha}|^2(2\pi)^ne^{-(\varphi+\Psi)(z_{\beta})}}{\Pi_{1\le j\le n}(\alpha_j+1)c_{j,\beta_j}^{2\alpha_{j}+2}}.	\end{split}	
\end{equation}

In the following, we prove the characterization for the holding of the equality $(\int_0^{+\infty}c(s)e^{-s}ds)\sum_{\beta\in I_1}\sum_{\alpha\in E_{\beta}}\frac{|d_{\beta,\alpha}|^2(2\pi)^ne^{-(\varphi+\Psi)(z_{\beta})}}{\Pi_{1\le j\le n}(\alpha_j+1)c_{j,\beta_j}^{2\alpha_{j}+2}}=\inf\{\int_{M}|\tilde{F}|^2e^{-\varphi}c(-\psi):\tilde{F}$ is a holomorphic $(n,0)$ form on $M$ such that $(\tilde{F}-f,z_0)\in(\mathcal{O}(K_{M})\otimes\mathcal{I}(\psi-\Psi))_{z_\beta}$ for any $\beta\in I_1\}$.

Following from Remark \ref{r:1.2}, we obtain the sufficiency. Now, we prove the necessity.
Inequality \eqref{eq:1127c} shows that $\int_{M}|F|^2e^{-\varphi}c(-\psi)= \int_{M}|F|^2e^{-\varphi-\Psi}c(-\psi+\Psi)$. As  $c(s)e^{-s}$ is decreasing with respect to $s$ and $\Psi\le0$, it follows from Lemma \ref{l:psi=G} that $\Psi\equiv0$, then $\psi=\max_{1\le j\le n}\{\pi_j^*(2\sum_{1\le k\le m_j}p_{j,k}G_{\Omega_j}(\cdot,z_{j,k}))\}.$

Let $t\ge0$. It follows from proposition \ref{p:exten-pro-finite}  ($M\sim\{\psi<-t\}$, $\psi\sim\psi+t$ and $c(\cdot)\sim c(\cdot+t)$, here $\sim$ means the former replaced by the latter), there exists a holomorphic $(n,0)$ form $F_t$ on $\{\psi<-t\}$ satisfying that $(F-f,z_\beta)\in(\mathcal{O}(K_M)\otimes\mathcal{I}(\psi))_{z_\beta}$ for any $\beta\in I_1$ and
\begin{equation*}
\int_{\{\psi<-t\}}|F_t|^2e^{-\varphi}c(-\psi)\le(\int_t^{+\infty}c(s)e^{-s}ds)\sum_{\beta\in I_1}\sum_{\alpha\in E_{\beta}}\frac{|d_{\beta,\alpha}|^2(2\pi)^ne^{-\varphi(z_{\beta})}}{\Pi_{1\le j\le n}(\alpha_j+1)c_{j,\beta_j}^{2\alpha_{j}+2}}.
\end{equation*}
By the definition of $G(t)$, then we obtain that inequality
\begin{equation*}
	\frac{G(t)}{\int_t^{+\infty}c(l)e^{-l}dl}\leq	\sum_{\beta\in I_1}\sum_{\alpha\in E_{\beta}}\frac{|d_{\beta,\alpha}|^2(2\pi)^ne^{-\varphi(z_{\beta})}}{\Pi_{1\le j\le n}(\alpha_j+1)c_{j,\beta_j}^{2\alpha_{j}+2}}=\frac{G(0)}{\int_0^{+\infty}c(s)e^{-s}ds}\end{equation*}
holds for any $t\geq0$.
 Using Theorem \ref{thm:general_concave}, we obtain that  $G(\hat{h}^{-1}(r))$ is linear with respect to $r$.
It follows from Theorem \ref{thm:prod-finite-point} that the four statements in Theorem \ref{thm:prod-finite-jet} hold.

Thus, Theorem \ref{thm:prod-finite-jet} holds.

\subsection{Proof of Theorem \ref{thm:prod-infinite-jet}}
\

 As $c(s)e^{-s}$ is decreasing with respect to $s$ and $\Psi\le0$, it follows from Proposition \ref{p:exten-pro-finite} that
there exists a holomorphic $(n,0)$ form $F$ on $M$ satisfying that $(F-f,z_\beta)\in(\mathcal{O}(K_M)\otimes\mathcal{I}(\max_{1\le j\le n}\{2\sum_{1\le k<\tilde m_j}p_{j,k}\pi_j^{*}(G_{\Omega_j}(\cdot,z_{j,k}))\}))_{z_\beta}$ for any $\beta\in\tilde I_1$ and
\begin{equation}\label{eq:1130i}
	\begin{split}
	\int_{M}|F|^2e^{-\varphi}c(-\psi)&\le \int_{M}|F|^2e^{-\varphi-\Psi}c(-\psi+\Psi)\\
	&\le(\int_0^{+\infty}c(s)e^{-s}ds)\sum_{\beta\in\tilde I_1}\sum_{\alpha\in E_{\beta}}\frac{|d_{\beta,\alpha}|^2(2\pi)^ne^{-(\varphi+\Psi)(z_{\beta})}}{\Pi_{1\le j\le n}(\alpha_j+1)c_{j,\beta_j}^{2\alpha_{j}+2}}.	\end{split}	
\end{equation}

In the following, we assume that $(\int_0^{+\infty}c(s)e^{-s}ds)\sum_{\beta\in\tilde I_1}\sum_{\alpha\in E_{\beta}}\frac{|d_{\beta,\alpha}|^2(2\pi)^ne^{-(\varphi+\Psi)(z_{\beta})}}{\Pi_{1\le j\le n}(\alpha_j+1)c_{j,\beta_j}^{2\alpha_{j}+2}}=\inf\{\int_{M}|\tilde{F}|^2e^{-\varphi}c(-\psi):\tilde{F}$ is a holomorphic $(n,0)$ form on $M$ such that $(\tilde{F}-f,z_0)\in(\mathcal{O}(K_{M})\otimes\mathcal{I}(\max_{1\le j\le n}\{2\sum_{1\le k<\tilde m_j}p_{j,k}\pi_j^{*}(G_{\Omega_j}(\cdot,z_{j,k}))\}))_{z_\beta}$ for any $\beta\in \tilde I_1\}$ to get a contradiction.	

Inequality \eqref{eq:1130i} shows that $\int_{M}|F|^2e^{-\varphi}c(-\psi)= \int_{M}|F|^2e^{-\varphi-\Psi}c(-\psi+\Psi)$. As  $c(s)e^{-s}$ is decreasing with respect to $s$ and $\Psi\le0$, it follows from Lemma \ref{l:psi=G} that $\Psi\equiv0$, then $\psi=\max_{1\le j\le n}\{\pi_j^*(2\sum_{1\le k<\tilde m_j}p_{j,k}G_{\Omega_j}(\cdot,z_{j,k}))\}.$
Let $t\ge0$. It follows from proposition \ref{p:exten-pro-finite}  ($M\sim\{\psi<-t\}$, $\psi\sim\psi+t$ and $c(\cdot)\sim c(\cdot+t)$, here $\sim$ means the former replaced by the latter), there exists a holomorphic $(n,0)$ form $F_t$ on $\{\psi<-t\}$ satisfying that $(F-f,z_\beta)\in(\mathcal{O}(K_M)\otimes\mathcal{I}(\psi))_{z_\beta}$ for any $\beta\in\tilde I_1$ and
\begin{equation*}
\int_{\{\psi<-t\}}|F_t|^2e^{-\varphi}c(-\psi)\le(\int_t^{+\infty}c(s)e^{-s}ds)\sum_{\beta\in\tilde I_1}\sum_{\alpha\in E_{\beta}}\frac{|d_{\beta,\alpha}|^2(2\pi)^ne^{-\varphi(z_{\beta})}}{\Pi_{1\le j\le n}(\alpha_j+1)c_{j,\beta_j}^{2\alpha_{j}+2}}.
\end{equation*}
By the definition of $G(t)$, then we obtain that inequality
\begin{equation*}
	\frac{G(t)}{\int_t^{+\infty}c(l)e^{-l}dl}\leq	\sum_{\beta\in\tilde I_1}\sum_{\alpha\in E_{\beta}}\frac{|d_{\beta,\alpha}|^2(2\pi)^ne^{-\varphi(z_{\beta})}}{\Pi_{1\le j\le n}(\alpha_j+1)c_{j,\beta_j}^{2\alpha_{j}+2}}=\frac{G(0)}{\int_0^{+\infty}c(s)e^{-s}ds}\end{equation*}
holds for any $t\geq0$.
 Using Theorem \ref{thm:general_concave}, we obtain that  $G(\hat{h}^{-1}(r))$ is linear with respect to $r$, which contradicts to Theorem \ref{thm:prod-infinite-point}. Thus, we have $\inf\{\int_{M}|\tilde{F}|^2e^{-\varphi}c(-\psi):\tilde{F}$ is a holomorphic $(n,0)$ form on $M$ such that $(\tilde{F}-f,z_0)\in(\mathcal{O}(K_{M})\otimes\mathcal{I}(\psi))_{z_\beta}$ for any $\beta\in \tilde I_1\}<(\int_0^{+\infty}c(s)e^{-s}ds)\sum_{\beta\in\tilde I_1}\sum_{\alpha\in E_{\beta}}\frac{|d_{\beta,\alpha}|^2(2\pi)^ne^{-(\varphi+\Psi)(z_{\beta})}}{\Pi_{1\le j\le n}(\alpha_j+1)c_{j,\beta_j}^{2\alpha_{j}+2}}$, which implies that there exists a holomorphic $(n,0)$ form $F$ on $\Omega$ such that $(\tilde{F}-f,z_0)\in(\mathcal{O}(K_{M})\otimes\mathcal{I}(\psi))_{z_\beta}$ for any $\beta\in \tilde I_1$ and
$$\int_{\Omega}|F|^2e^{-\varphi}c(-\psi)<(\int_0^{+\infty}c(s)e^{-s}ds)\sum_{\beta\in\tilde I_1}\sum_{\alpha\in E_{\beta}}\frac{|d_{\beta,\alpha}|^2(2\pi)^ne^{-(\varphi+\Psi)(z_{\beta})}}{\Pi_{1\le j\le n}(\alpha_j+1)c_{j,\beta_j}^{2\alpha_{j}+2}}.$$

Thus, Theorem \ref{thm:prod-infinite-jet} holds.

\section{Proofs of Theorem \ref{thm:ohsawa}, Theorem \ref{thm:suita} and Theorem \ref{thm:extend}}
In this section, we prove Theorem \ref{thm:ohsawa}, Theorem \ref{thm:suita} and Theorem \ref{thm:extend}.

\subsection{Proof of Theorem \ref{thm:ohsawa}}
\

Let $z_0=(z_1,z_2,...,z_n)\in S\subset M$.
Let $w_j$ be a local coordinate on a neighborhood $V_{z_j}$ of $z_j\in\Omega_j$ satisfying $w_j(z_j)=0$. Denote that $V_0:=\prod_{1\le j\le n}V_{z_j}$, and $w:=(w_1,w_2,...,w_n)$ is a local coordinate on $V_0$ of $z_0\in M$. Let $f=dw_1\wedge dw_2\wedge...\wedge dw_n$ on $V_0$.
 Denote
	\begin{equation*}
\begin{split}
\inf\{\int_{\{\Psi<-t\}}|\tilde{f}|^{2}:(\tilde{f}-f,z_0)\in&(\mathcal{O}(K_M)\otimes\mathcal{I}(\Psi))_{z_0} \\&\&{\,}\tilde{f}\in H^{0}(\{\Psi<-t\},\mathcal{O}(K_{M}))\},
\end{split}
\end{equation*}
by $G_{\Psi}(t)$ for $t\ge0$ and $\Psi\in\Delta(\{z_0\})$. It follows from Proposition \ref{p:pluricomplex} that $G(\cdot,\{z_0\})=\max_{1\le j\le n}\{\pi_j^*(2nG_{\Omega_j}(\cdot,z_j))\}$ and $G(\cdot,S)-G(\cdot,\{z_0\})\le 0$.

Firstly, we prove the necessity. Note that
$$\frac{\kappa_{M}(z_0)}{dw_1\wedge dw_2\wedge...\wedge dw_n\otimes\overline{dw_1\wedge dw_2\wedge...\wedge dw_n}}=\frac{2^n}{G_{G(\cdot,S)}(0)}=\frac{2^n}{G_{G(\cdot,\{z_0\})}(0)}$$ and
\begin{displaymath}
	\begin{split}
		&\frac{\kappa_{M/S}(z_0)}{dw_1\wedge dw_2\wedge...\wedge dw_n\otimes\overline{dw_1\wedge dw_2\wedge...\wedge dw_n}}\\
		=&\frac{1}{2^{-n}\sqrt{-1}^{n^2}\int_{z_0}\frac{dw_1\wedge dw_2\wedge...\wedge dw_n\wedge\overline{dw_1\wedge dw_2\wedge...\wedge dw_n}}{dV_{M}}dV_{M}[G(\cdot,S)]},
	\end{split}
\end{displaymath}
$$$$
which implies that $G_{G(\cdot,S)}(0)=\frac{\pi^n}{n!}\int_{z_0}\frac{|dw_1\wedge dw_2\wedge...\wedge dw_n|^2}{dV_M}dV_{M}[G(\cdot,S)]$.
Note that $dV_{M}[G(\cdot,S)+t]=e^{-t}dV_{M}[G(\cdot,S)]$. It follows from Theorem \ref{gz:L2} that
\begin{equation}
	\label{eq:1201c}G_{G(\cdot,S)}(t)\le e^{-t}\frac{\pi^n}{n!}\int_{z_0}\frac{|dw_1\wedge dw_2\wedge...\wedge dw_n|^2}{dV_M}dV_{M}[G(\cdot,S)].
\end{equation}
Combining Theorem \ref{thm:general_concave}, $G_{G(\cdot,S)}(0)=\frac{\pi^n}{n!}\int_{z_0}\frac{|dw_1\wedge dw_2\wedge...\wedge dw_n|^2}{dV_M}dV_{M}[G(\cdot,S)]$ and inequality \eqref{eq:1201c}, we have
$G_{G(\cdot,S)}(-\log r)$ is linear with respect to $r$. It follows from Thereom \ref{thm:general_concave} and $G(\cdot,S)-G(\cdot,\{z_0\})\le 0$  that $G_{G(\cdot,\{z_0\})}(-\log r)$ is concave with respect to $r$ and $G_{G(\cdot,\{z_0\})}(t)\le G_{G(\cdot,S)}(t)$ for any $t\ge 0$. Note that $G(\cdot,\{z_0\})=\max_{1\le j\le n}\{\pi_j^*(2nG_{\Omega_j}(\cdot,z_j))\}$. As $G_{G(\cdot,\{z_0\})}(0)= G_{G(\cdot,S)}(0)$, we have $G_{G(\cdot,\{z_0\})}(t)=G_{G(\cdot,S)}(t)$ for any $t\ge0$, which implies that $G(\cdot,S)\equiv G(\cdot,\{z_0\})$, hence $S$ is a single point set. As $G_{G(\cdot,\{z_0\})}(-\log r)$ is linear with respect to $r$, it follows from Theorem \ref{thm:linear-2d} that $\chi_{j,z_j}=1$ for any $j\in\{1,2,...,n\}$. There exists a holomorphic function $f_j$ on $\Omega_j$ such that $|f_j|=e^{G_{\Omega_j}(\cdot,z_j)}$, thus $\Omega_j$ is conformally equivalent to the unit disc less a (possible) closed set of inner capacity zero (see \cite{suita72}, see also \cite{Yamada} and \cite{GZ15}) for any $j\in\{1,2,...,n\}$.

Now, we prove sufficiency. As $\Omega_j$ is conformally equivalent to the unit disc less a (possible) closed set of inner capacity zero  for any $j\in\{1,2,...,n\}$, we know $\chi_{j,z_j}=1$ for any $j\in\{1,2,...,n\}$. It follows from Lemma \ref{l:m1} that
\begin{displaymath}
	\begin{split}
		&\frac{\kappa_{M/S}(z_0)}{dw_1\wedge dw_2\wedge...\wedge dw_n\otimes\overline{dw_1\wedge dw_2\wedge...\wedge dw_n}}\\
		=&\frac{1}{2^{-n}\int_{z_0}\frac{|dw_1\wedge dw_2\wedge...\wedge dw_n|^2}{dV_{M}}dV_{M}[G(\cdot,\{z_0\})]}\\
		=&\frac{\Pi_{1\le j\le n}c_{j}(z_j)^2}{n!}.
	\end{split}
\end{displaymath}
It follows from Theorem \ref{thm:2d-jet} that $G_{G(\cdot,\{z_0\})}(0)=\frac{(2\pi)^n}{\Pi_{1\le j\le n}c_j(z_j)^2}$, which implies that $\frac{\pi^n}{n!}\kappa_M(z_0)=\kappa_{M/S}(z_0)$.

Thus, we prove Theorem \ref{thm:ohsawa}.

\subsection{Proof of Theorem \ref{thm:suita}}
\

 Let $f=dw_1\wedge dw_2\wedge...\wedge dw_n$ on $V_0$ and $\psi=\max_{1\le j\le n}\{\pi_j^*(2nG_{\Omega_j}(\cdot,z_j))\}$. It is clear that  $\frac{2^n}{B_M(z_0)}=\inf\{\int_{M}|\tilde{f}|^2:(\tilde{f}-f,z_0)\in(\mathcal{O}(K_M)\otimes\mathcal{I}(\psi))_{z_0}\, \&{\,}\tilde{f}\in H^{0}(M,\mathcal{O}(K_{M}))\}$. If  $\Omega_j$ is conformally equivalent to the unit disc less a (possible) closed set of inner capacity zero, we know $\chi_{j,z_j}=1$. If $\chi_{j,z_j}=1$, there exists a holomorphic function $f_j$ on $\Omega_j$ such that $|f_j|=e^{G_{\Omega_j}(\cdot,z_j)}$, thus $\Omega_j$ is conformally equivalent to the unit disc less a (possible) closed set of inner capacity zero (see \cite{suita72}, see also \cite{Yamada} and \cite{GZ15}). Thus, Theorem \ref{thm:2d-jet} implies that Theorem \ref{thm:suita} holds.

\subsection{Proof of Theorem \ref{thm:extend}}
\

 Let $f=dw_1\wedge dw_2\wedge...\wedge dw_n$ on $V_0$ and $\psi=\max_{1\le j\le n}\{\pi_j^*(2nG_{\Omega_j}(\cdot,z_j))\}$. It is clear that  $\frac{2^n}{B_{M,\rho}(z_0)}=\inf\{\int_{M}|\tilde{f}|^2\rho:(\tilde{f}-f,z_0)\in(\mathcal{O}(K_M)\otimes\mathcal{I}(\psi))_{z_0}\, \&{\,}\tilde{f}\in H^{0}(M,\mathcal{O}(K_{M}))\}$.  Thus, Theorem \ref{thm:2d-jet} implies that Theorem \ref{thm:extend} holds.

\section{Appendix}

Let $\Omega_j$  be an open Riemann surface, which admits a nontrivial Green function $G_{\Omega_j}$ for any  $1\le j\le n$. Let $M=\prod_{1\le j\le n}\Omega_j$ be an $n-$dimensional complex manifold, and let $\pi_j$ be the natural projection from $M$ to $\Omega_j$.
Let $z_j\in\Omega_j$ and $z_0\in M$ such that $\pi_j(z_0)=z_j$ for any $1\le j\le n$.
Let $w_j$ be a local coordinate on a neighborhood $V_{z_j}$ of $z_j\in\Omega_j$ satisfying $w_j(z_j)=0$. Denote that $V_0:=\prod_{1\le j\le n}V_{z_j}$, and $w:=(w_1,w_2,...,w_n)$ is a local coordinate on $V_0$ of $z_0\in M$.

We recall a well-known result of pluricomplex Green functions. For convenience of readers, we give a proof.
\begin{Proposition}\label{p:pluricomplex}
	$\max_{1\le j\le n}\{\pi_j^*(G_{\Omega_j}(\cdot,z_j))\}=\sup\{\psi:{\psi\in\Delta_M^*(z_0)}\}$, where $\Delta_M^*(z_0)$ is the set of negative plurisubharmonic functions on $M$ such that $\psi(z)-\log|w(z)|$ has a locally finite upper bound near $z_0$.
\end{Proposition}
\begin{proof}
	As $\max_{1\le j\le n}\{\log|w_j|\}<\log|w|$, we have $\max_{1\le j\le n}\{\pi_j^*(G_{\Omega_j}(\cdot,z_j))\}\in \Delta_M^*(z_0)$.
	We prove Proposition \ref{p:pluricomplex} by contradiction: if not, there exists a negative plurisubharmonic function $\psi$  on $M$ such that $\psi(z)-\log|w(z)|$ has a locally finite upper bound near $z_0$, $\psi\ge\max_{1\le j\le n}\{\pi_j^*(G_{\Omega_j}(\cdot,z_j))\}$ and $\psi\not\equiv\max_{1\le j\le n}\{\pi_j^*(G_{\Omega_j}(\cdot,z_j))\}$ (as $\max\{\psi_1,\psi_2\}\in\Delta_M^*(z_0)$ for any $\psi_1\in\Delta_M^*(z_0)$ and $\psi_2\in\Delta_M^*(z_0)$). It follows from Remark \ref{r:chi} that there exists a harmonic function $u_j$ on $\Omega_j$ such that $\chi_{j,-u_j}=\chi_{j,z_j}$ for any $j\in\{1,2,...,n\}$. Let $\varphi=\sum_{1\le j\le n}\pi_j^*(2u_j)$, and let $f=dw_1\wedge dw_2\wedge...\wedge dw_n$ on $V_0$. Denote
	\begin{equation*}
\begin{split}
\inf\{\int_{\{\Psi<-t\}}|\tilde{f}|^{2}e^{-\varphi}:(\tilde{f}-f,z_0)\in&(\mathcal{O}(K_M)\otimes\mathcal{I}(\Psi))_{z_0} \\&\&{\,}\tilde{f}\in H^{0}(\{\Psi<-t\},\mathcal{O}(K_{M}))\}
\end{split}
\end{equation*}
by $G_{\Psi}(t)$ for $t\ge0$ and $\Psi\in\Delta_M^*(z_0)$.
 It is clear that there exists a neighborhood $U$ of the origin $o$ in $\mathbb{C}^n$ such that
$$\sup_{w\in U\backslash\{o\}}|\max_{1\le j\le n}\{\log|w_j|\}-\log|w||<+\infty,$$ which implies that
$(f,z_0)\in\mathcal{I}(2n\max_{1\le j\le n}\{\pi_j^*(G_{\Omega_j}(\cdot,z_j))\})_{z_0}$ if and only if $f(z_0)=0$. As $\psi\geq\max_{1\le j\le n}\{\pi_j^*(G_{\Omega_j}(\cdot,z_j))\}$ and $\psi(z)-\log|w(z)|$ has a locally finite upper bound near $z_0$, we have $(f,z_0)\in\mathcal{I}(2n\psi)_{z_0}$ if and only if $f(z_0)=0$. Thus, we have
\begin{equation}
	\label{eq:1201a}G_{2n\psi}(0)=G_{2n\max_{1\le j\le n}\{\pi_j^*(G_{\Omega_j}(\cdot,z_j))\}}(0).\end{equation}
 As $\psi\ge\max_{1\le j\le n}\{\pi_j^*(G_{\Omega_j}(\cdot,z_j))\}$ and $\psi\not\equiv\max_{1\le j\le n}\{\pi_j^*(G_{\Omega_j}(\cdot,z_j))\}$, we obtain that there exists $t>0$ such that
\begin{equation}
	\label{eq:1201b}G_{2n\psi}(t)<G_{2n\max_{1\le j\le n}\{\pi_j^*(G_{\Omega_j}(\cdot,z_j))\}}(t).\end{equation}
It follows from Theorem \ref{thm:linear-2d} that $G_{2n\max_{1\le j\le n}\{\pi_j^*(G_{\Omega_j}(\cdot,z_j))\}}(-\log r)$ is linear with respect to $r$. Theorem \ref{thm:general_concave} shows that $G_{2n\psi}(-\log r)$ is concave with respect to $r$. Inequality \eqref{eq:1201b} implies that $G_{2n\psi}(0)<G_{2n\max_{1\le j\le n}\{\pi_j^*(G_{\Omega_j}(\cdot,z_j))\}}(0)$, which contradicts to equality \eqref{eq:1201a}.
Thus, we get that $\max_{1\le j\le n}\{\pi_j^*(G_{\Omega_j}(\cdot,z_j))\}=\sup\{\psi:{\psi\in\Delta_M^*(z_0)}\}$.
\end{proof}


\vspace{.1in} {\em Acknowledgements}.
The authors would like to thank Dr. Shijie Bao and Dr. Zhitong Mi for checking the manuscript and  pointing out some typos. The first named author was supported by National Key R\&D Program of China 2021YFA1003100, NSFC-11825101, NSFC-11522101 and NSFC-11431013.
\bibliographystyle{references}
\bibliography{xbib}

\begin{thebibliography}{100}





\bibitem{Berndtsson2}B. Berndtsson,
The openness conjecture for plurisubharmonic functions, arXiv:1305.5781.

\bibitem{berndtsson20}B. Berndtsson, Lelong numbers and vector bundles, J. Geom. Anal. 30 (2020), no. 3, 2361-2376.

\bibitem{Blo13}Z. B\l ocki, Suita  conjecture  and  the  Ohsawa-Takegoshi extension  theorem,  Invent.  Math.  193(2013), 149-158.



\bibitem{cao17}J.Y. Cao, Ohsawa-Takegoshi extension theorem for compact K{\"a}hler manifolds and applications, Complex and symplectic geometry, 19-38,
Springer INdAM Ser., 21, Springer, Cham, 2017.

\bibitem{cdM17}J.Y. Cao, J-P. Demailly and S. Matsumura, A general extension theorem for cohomology classes on non reduced analytic subspaces, Sci. China Math. 60 (2017), no. 6, 949-962, DOI 10.1007/s11425-017-9066-0.


\bibitem{DEL18}T. Darvas, E. Di Nezza and H.C. Lu,
Monotonicity of nonpluripolar products and complex Monge-Amp{\'e}re equations with prescribed singularity,
Anal. PDE 11 (2018), no. 8, 2049-2087.

\bibitem{DEL21}T. Darvas, E. Di Nezza and H.C. Lu, The metric geometry of singularity types,
J. Reine Angew. Math. 771 (2021), 137-170.



\bibitem{Demaillybook}J.-P Demailly,
Complex analytic and differential geometry,
electronically accessible at https://www-fourier.ujf-grenoble.fr/~demailly/manuscripts/agbook.pdf.

\bibitem{DemaillyAG}J.-P Demailly,
Analytic Methods in Algebraic Geometry,
Higher Education Press, Beijing, 2010.



\bibitem{DemaillySoc}J.-P Demailly,
Multiplier ideal sheaves and analytic methods in algebraic geometry, School on Vanishing Theorems and Effective Result in Algebraic Geometry (Trieste,2000),1-148,ICTP lECT.Notes, 6, Abdus Salam Int. Cent. Theoret. Phys., Trieste, 2001.






\bibitem{DEL}J.-P Demailly, L. Ein and R. Lazarsfeld,
A subadditivity property of multiplier ideals,
Michigan Math. J. 48 (2000) 137-156.

\bibitem{DK01}J.-P Demailly and J. Koll\'ar,
Semi-continuity of complex singularity exponents and K\"ahler-Einstein metrics on Fano orbifolds,
Ann. Sci. \'Ec. Norm. Sup\'er. (4) 34 (4) (2001) 525-556.

\bibitem{DP03}J.-P Demailly and T. Peternell,
A Kawamata-Viehweg vanishing theorem on compact K\"ahler manifolds,
J. Differential Geom. 63 (2) (2003) 231-277.

\bibitem{FavreJonsson}C. Favre and M. Jonsson,
Valuations and multiplier ideals,
J. Amer. Math. Soc. 18 (2005),
no. 3, 655-684.


\bibitem{OF81}O. Forster, Lectures on Riemann surfaces, Grad. Texts in Math., 81, Springer-Verlag, New York-Berlin, 1981.

\bibitem{FoW18}J.E. Forn{\ae}ss and J.J.  Wu,
A global approximation result by Bert Alan Taylor and the strong openness conjecture in $\mathbb{C}^n$,
J. Geom. Anal. 28 (2018), no. 1, 1-12.

\bibitem{FoW20}J.E. Forn{\ae}ss and J.J.  Wu, Weighted approximation in $\mathbb{C}$, Math. Z. 294 (2020), no. 3-4, 1051-1064.


\bibitem{G-R}H. Grauert and R. Remmert, Coherent analytic sheaves, Grundlehren der mathematischen Wissenchaften, 265, Springer-Verlag, Berlin, 1984.

\bibitem{G2018}Q.A. Guan,
Genral concavity of minimal $L^2$ integrals related to multiplier
sheaves,
arXiv:1811.03261.v4.

\bibitem{G16}Q.A. Guan,
A sharp effectiveness result of Demailly's strong Openness conjecture,
Adv.Math. 348 (2019) :51-80.

\bibitem{Guan2019}Q.A. Guan, A proof of Saitoh's conjecture for conjugate Hardy H2 kernels. J. Math. Soc. Japan 71 (2019), no. 4, 1173--1179.

\bibitem{guan-20}Q.A. Guan, Decreasing equisingular approximations with analytic singularities, J. Geom. Anal. 30 (2020), no. 1, 484-492.

\bibitem{GM}
Q.A. Guan and Z.T. Mi, Concavity of minimal $L^2$ integrals related to multiplier ideal
sheaves, arXiv:2106.05089v2.

\bibitem{GM_Sci} Q.A. Guan and Z.T. Mi, Concavity of minimal $L^2$ integrals related to multiplier ideal
sheaves on weakly pseudoconvex K\"ahler manifolds, submitted.

\bibitem{GMY-concavity2}Q.A. Guan, Z.T. Mi and Z. Yuan, Concavity property of minimal $L^2$ integrals with Lebesgue measurable gain \uppercase\expandafter{\romannumeral2}, https://www.researchgate.net/publication/354464147.



\bibitem{GY-concavity}Q.A. Guan and Z. Yuan, Concavity property of minimal $L^2$ integrals with Lebesgue measurable gain,  https://www.researchgate.net/publication/353794984.

\bibitem{GY-support}Q.A. Guan and Z. Yuan,  An optimal support function related to the strong openness property, arXiv:2105.07755v2.

\bibitem{GY-lp-effe}Q.A. Guan and Z. Yuan, Effectiveness of strong openness property in $L^p$, arXiv:2106.03552v3.

\bibitem{GY-twisted}Q.A. Guan and Z. Yuan, Twisted version of strong openness property in $L^p$, arXiv:2109.00353.

\bibitem{GY-concavity3}Q.A. Guan and Z. Yuan, Concavity property of minimal $L^2$ integrals with Lebesgue measurable gain III-----open Riemann surfaces, https://www.researchgate.net/publication/356171464.

\bibitem{gz12}Q.A. Guan and X.Y Zhou, Optimal constant problem in the $L^2$ extension theorem, C. R. Math. Acad. Sci. Paris 350 (2012), 753-756. MR 2981347. Zbl 1256.32009. http://dx.doi.org/10.1016/j.crma.2012.08.007.

\bibitem{GZ15}Q.A. Guan and X.Y Zhou,
Optimal constant in an $L^2$ extension problem and a proof of a conjecture
of Ohsawa,
Sci.China Math., 2015 , 58(1) :35-59.


\bibitem{GZSOC}Q.A. Guan and X.Y Zhou,
A proof of Demailly's strong openness conjecture, Ann. of Math.
(2) 182 (2015), no. 2, 605-616.


\bibitem{guan-zhou13ap}Q.A. Guan and X.Y. Zhou, A solution of an $L^{2}$ extension problem with an optimal estimate and applications,
Ann. of Math. (2) 181 (2015), no. 3, 1139--1208.

\bibitem{GZeff}Q.A. Guan and X.Y Zhou,
Effectiveness of Demailly's strong openness conjecture and
related problems, Invent. Math. 202 (2015), no. 2, 635-676.




\bibitem{GZ20}Q.A. Guan and X.Y. Zhou, Restriction formula and subadditivity property related to multiplier ideal sheaves, J. Reine Angew. Math. 769, 1-33 (2020).

\bibitem{Guenancia}H. Guenancia,
Toric plurisubharmonic functions and analytic adjoint ideal sheaves, Math. Z. 271 (3-4) (2012) 1011-1035.

\bibitem{JonssonMustata}M. Jonsson and M. Musta\c{t}$\breve{a}$,
Valuations and asymptotic invariants for sequences of ideals, Annales de L'Institut Fourier A. 2012, vol. 62, no.6, pp. 2145-2209.



\bibitem{K16}D. Kim, Skoda division of line bundle sections and pseudo-division, Internat. J. Math. 27 (2016), no. 5, 1650042, 12 pp.


\bibitem{KS20}D. Kim and H. Seo, Jumping numbers of analytic multiplier ideals (with an appendix by Sebastien Boucksom), Ann. Polon. Math., 124 (2020), 257-280.


\bibitem{Lazarsfeld}R. Lazarsfeld,
Positivity in Algebraic Geometry. \uppercase\expandafter{\romannumeral1}. Classical Setting: Line Bundles and Linear Series. Ergebnisse der Mathematik und ihrer Grenzgebiete. 3. Folge. A Series of Modern Surveys in Mathematics [Results in Mathematics and Related Areas. 3rd Series. A Series of Modern Surveys in Mathematics], 48. Springer-Verlag, Berlin, 2004;\\
R. Lazarsfeld,
Positivity in Algebraic Geometry. \uppercase\expandafter{\romannumeral2}. Positivity for vector bundles, and multiplier ideals. Ergebnisse der Mathematik und ihrer Grenzgebiete. 3. Folge. A Series of Modern Surveys in Mathematics [Results in Mathematics and Related Areas. 3rd Series. A Series of Modern Surveys in Mathematics], 49. Springer-Verlag, Berlin, 2004.




\bibitem{Nadel}A. Nadel,
Multiplier ideal sheaves and K\"ahler-Einstein metrics of positive scalar curvature, Ann. of Math. (2) 132 (3) (1990) 549-596.

\bibitem{Ohsawa5}T. Ohsawa,
On the extension of $L^2$ holomorphic functions. V. Effects of generalization, Nagoya Math. J. 161 (2001), 1-21. Erratum to: ¡±On the extension of $L^2$ holomorphic functions. V. Effects of generalization¡± [Nagoya Math. J. 161 (2001), 1-21]. Nagoya Math.J. 163 (2001), 229.





\bibitem{S-O69}L. Sario and K. Oikawa, Capacity functions, Grundl. Math. Wissen. 149, Springer-Verlag, New York, 1969. Mr 0065652. Zbl 0059.06901.

\bibitem{Siu96}
Y.T. Siu,
The Fujita conjecture and the extension theorem of Ohsawa-Takegoshi,
 Geometric Complex Analysis, World Scientific, Hayama, 1996, pp.223-277.

\bibitem{Siu05}Y.T. Siu,
Multiplier ideal sheaves in complex and algebraic geometry,
Sci. China Ser. A 48 (suppl.) (2005) 1-31.


\bibitem{Siu09}Y.T. Siu,
Dynamic multiplier ideal sheaves and the construction of rational curves in Fano manifolds,
Complex Analysis and Digtial Geometry, in: Acta Univ. Upsaliensis Skr. Uppsala Univ. C Organ. Hist., vol.86, Uppsala Universitet, Uppsala, 2009, pp.323-360.


\bibitem{suita72}N. Suita, Capacities and kernels on Riemann surfaces, Arch. Rational Mech. Anal. 46 (1972), 212-217.

\bibitem{Tian}G. Tian,
On K\"ahler-Einstein metrics on certain K\"ahler manifolds with $C_1(M)>0$, Invent. Math. 89 (2) (1987) 225-246.



\bibitem{Tsuji}M. Tsuji, Potential theory in modern function theory, Maruzen Co., Ltd., Tokyo, 1959. MR 0114894. Zbl 0087.28401.



\bibitem{Yamada}A. Yamada,
Topics related to reproducing kemels, theta functions and the Suita conjecture
(Japanese), The theory of reproducing kemels and their applications (Kyoto 1998), S$\bar{u}$rikaisekikenky$\bar{u}$sho K$\bar{o}$ky$\bar{u}$roku, 1998, 1067(1067):39-47.


\bibitem{ZZ2018}X.Y Zhou and L.F.Zhu,
An optimal $L^2$ extension theorem on weakly pseudoconvex K\"ahler
manifolds,
J. Differential Geom.110(2018), no.1, 135-186.


\bibitem{ZZ2019}
X.Y Zhou and L.F.Zhu,
Optimal $L^2$ extension of sections from subvarieties in weakly pseudoconvex manifolds. Pacific J. Math. 309 (2020), no. 2, 475-510.


\bibitem{ZhouZhu20siu's}X.Y. Zhou and L.F. Zhu, Siu's lemma, optimal L2 extension and applications to twisted pluricanonical sheaves, Math. Ann. 377 (2020), no. 1-2, 675-722.

\end{thebibliography}

\end{document}